\documentclass[10pt,reqno]{amsart}

\usepackage{amsthm, mathrsfs, amsmath, amstext, amsxtra, amsfonts, dsfont, amssymb, color}
\usepackage{lmodern}
\definecolor{green_dark}{rgb}{0,0.6,0}
\usepackage[colorlinks, linkcolor=red, citecolor=blue, urlcolor=green_dark, pagebackref, hypertexnames=false]{hyperref}

\usepackage{xpatch}

\xpatchcmd{\paragraph}{\normalfont}{{\normalfont\bfseries}}{}{}

\newtheorem{theorem}{Theorem}[section]

\newtheorem{lem}[theorem]{Lemma} 
\newtheorem{prop}[theorem]{Proposition}
\newtheorem{coro}[theorem]{Corollary} 
\theoremstyle{definition}
\newtheorem{rem}[theorem]{Remark}

\newcommand{\N}{\mathbb N}
\newcommand{\Z}{\mathbb Z}

\newcommand{\R}{\mathbb R}
\newcommand{\C}{\mathbb C}

\newcommand{\Lc}{\mathcal L}
\newcommand{\Kc}{\mathcal K}
\newcommand{\Hc}{\mathcal H}
\newcommand{\Jc}{\mathcal J}
\newcommand{\Rc}{\mathcal R}
\newcommand{\Sch}{\mathscr S}
\newcommand{\Lch}{\mathscr L}
\newcommand{\ep}{\epsilon}
\newcommand{\Ga}{\Gamma}
\newcommand{\re}[1]{\mbox{Re} \ #1} 
\newcommand{\im}[1]{\mbox{Im} \ #1} 
\newcommand{\scal}[1]{\left\langle #1 \right\rangle} 

\newcommand{\defendproof}{\hfill $\Box$} 

\title[Global Strichartz Fractional Schr\"odinger Equation]{Global in time Strichartz estimates for the fractional Schr\"odinger equations on asymptotically Euclidean manifolds} 

\author[V. D. Dinh]{Van Duong Dinh}
\address[V. D. Dinh]{Institut de Math\'ematiques de Toulouse, Universit\'e Toulouse III Paul Sabatier, 31062 Toulouse Cedex 9, France and Department of Mathematics, HCMC University of Pedagogy, 280 An Duong Vuong, Ho Chi Minh, Vietnam}
\email{dinhvan.duong@math.univ-toulouse.fr}


\begin{document}

\maketitle
\begin{abstract}
In this paper, we prove global in time Strichartz estimates for the fractional Schr\"odinger operators, namely $e^{-it\Lambda_g^\sigma}$ with $\sigma \in (0,\infty)\backslash \{1\}$ and $\Lambda_g:=\sqrt{-\Delta_g}$ where $\Delta_g$ is the Laplace-Beltrami operator on asymptotically Euclidean manifolds $(\R^d,g)$. Let $f_0\in C^\infty_0(\R)$ be a smooth cutoff equal 1 near zero. We firstly show that the high frequency part $(1-f_0)(P)e^{-it\Lambda_g^\sigma}$ satisfies global in time Strichartz estimates as on $\R^d$ of dimension $d\geq 2$ inside a compact set under non-trapping condition. On the other hand, under the moderate trapping assumption $(\ref{assump resolvent})$, the high frequency part also satisfies the global in time Strichartz estimates outside a compact set. We next prove that the low frequency part $f_0(P)e^{-it\Lambda_g^\sigma}$ satisfies global in time Strichartz estimates as on $\R^d$ of dimension $d\geq 3$ without using any geometric assumption on $g$. As a byproduct, we prove global in time Strichartz estimates for the fractional Schr\"odinger and wave equations on $(\R^d, g), d\geq 3$ under non-trapping condition.
\end{abstract}

\section{Introduction}
\setcounter{equation}{0}
Let $(M,g)$ be a $d$-dimensional Riemannian manifold. We consider the time dependent fractional Schr\"odinger equation on $(M,g)$, namely
\begin{align}
i\partial_tu -\Lambda_g^\sigma u =0, \quad u_{\vert t=0} =u_0, \label{fractional schrodinger equation}
\end{align}
with $\sigma \in (0,\infty)\backslash \{1\}, \Lambda_g=\sqrt{-\Delta_g}$ where $\Delta_g$ is the Laplace-Beltrami operator associated to the metric $g$. The fractional Schr\"odinger equation $(\ref{fractional schrodinger equation})$ arises in many physical contexts. When $\sigma \in (0,2)\backslash \{1\}$, the fractional Schr\"odinger equation was discovered by N. Laskin (see \cite{Laskin2000}, \cite{Laskin2002}) as a result of extending the Feynman path integral, from the Brownian-like to L\'evy-like quantum mechanical paths. This type of equation also appears in the water wave models (see \cite{IonescuPusateri}, \cite{Nguyen}). When $\sigma=2$, it corresponds to the well-known Schr\"odinger equation. In the case $\sigma=4$, it is the fourth-order Schr\"odinger equation introduced by Karpman \cite{Karpman} and Karpman and Shagalov \cite{KarpmanShagalov} to take into account the role of small fourth-order dispersion terms in the propagation of intense laser beams in a bulk medium with Kerr nonlinearity. \newline
\indent When $M=\R^d$ and $g=\text{Id}$, i.e. the flat Euclidean space, the solution to $(\ref{fractional schrodinger equation})$ enjoys the following global in time Strichartz estimates (see \cite{Dinh}),
\[
\|u\|_{L^p(\R, L^q(\R^d))} \lesssim \|u_0\|_{\dot{H}^{\gamma_{p,q}}(\R^d)},
\]
where $(p,q)$ satisfies the fractional admissible condition, i.e.
\begin{align}
p\in [2,\infty],\quad q \in [2, \infty), \quad (p,q,d) \ne (2,\infty,2), \quad \frac{2}{p}+\frac{d}{q} \leq \frac{d}{2}, \label{fractional admissible}
\end{align}
with
\begin{align}
\gamma_{p,q}=\frac{d}{2}-\frac{d}{q}-\frac{\sigma}{p}. \label{define gamma pq}
\end{align}
Remark that one also has global in time Strichartz estimates for $q=\infty$, but one has to replace the Lebesgue norm $L^\infty(\R^d)$ by a corresponding Besov norm due to the Littlewood-Paley theorem. We refer the reader to \cite{Dinh} for more details. \newline
\indent When $M$ is a compact Riemannian manifold without boundary and $g$ is smooth, we also have (see \cite{Dinhcompact}) Strichartz estimates but only local in time,
\[
\|u\|_{L^p([0,1],L^q(M))} \lesssim \|u_0\|_{H^\gamma(M)}.
\]
In the case $\sigma \in (0,1)$, 
we have the same (local in time) Strichartz estimates as in $(\R^d, \text{Id})$, i.e. $\gamma=\gamma_{p,q}$. In the case $\sigma \in (1,\infty)$, there is a ``loss'' of derivatives $(\sigma-1)/p$ comparing to the one on $(\R^d,\text{Id})$, i.e. $\gamma=\gamma_{p,q}+(\sigma-1)/p$. \newline
\indent When $M$ is a non-compact Riemannian manifold, global in time Strichartz estimates for the Schr\"odinger equation (i.e. $\sigma=2$) have been studied intensively. In \cite{BTglobalstrichartz}, Bouclet-Tzvetkov established global in time Strichartz estimates on asymptotically Euclidean manifold, i.e. $\R^d$ equipped with a long range perturbation metric $g$ (see $(\ref{assump long range})$) with a low frequency cutoff under non-trapping condition. The first breakthrough on this topic was done by Tataru in \cite{Tataru} where the author considered long range and globally small perturbations of the Euclidean metric with $C^2$ and time dependent coefficients. In this setting, no trapping could occur. Later, Marzuola-Metcalfe-Tataru in \cite{MarzuolaMetcalfeTataru} improved the results considering more general perturbations in a compact set, including some weak trapping. Afterward, Hassell-Zhang in \cite{HassellZhang} extended those results for general geometric framework of asymptotically conic manifolds and including very short range potentials with non-trapping condition. Recently, Bouclet-Mizutani in \cite{BoucletMizutani} established global in time Strichartz estimates for a more general class of asymptotically conic manifolds including all usual smooth long range perturbations of the Euclidean metric. After that, Zhang-Zheng \cite{ZhangZheng-scattering} extended the result of Hassell-Zhang \cite{HassellZhang} and proved global in time Strichartz estimates for Schr\"odinger operators with potentials on assymptotically conic manifoldswith non-trapping condition. They also extended Bouclet-Mizutani's result \cite{BoucletMizutani} by considering Schr\"odinger operators with short range potentials on asymptotically conic manifolds with hyperbobic trapping condition. Recently, Zhang-Zheng \cite{ZhangZheng-cone} established global in time Strichartz estimates for Schr\"odinger operators on metric cone.\newline
\indent In order to prove Strichartz estimates on curved backgrounds, one uses the Littlewood-Paley decomposition to localize the solution in frequency. One then uses microlocal techniques to derive dispersive estimates and obtain Strichartz estimates for each spectrally localized components. By summing over all frequency pieces, one gets Strichartz estimates for the solution. For local in time Strichartz estimates, this usual scheme works very well. However, for global in time Strichartz estimates, one has to face a difficulty arising at low frequency. Due to the uncertainty principle, one can only use microlocal techniques for data supported outside compact sets at low frequency. Therefore, one has to use another technique for data supported inside compact sets. Note also that on $\R^d$, one can use the scaling technique to reduce the analysis at low frequency to the study at frequency one, but this technique does not work on manifolds in general. \newline
\indent The goal of this paper is to study global in time Strichartz estimates for the fractional Schr\"odinger equation on asymptotically Euclidean manifolds. In the case of Schr\"odinger equation, it can be seen as a completion of those in \cite{BTglobalstrichartz} of spatial dimensions greater than or equal to 3. In order to achieve this goal, we will use the techniques of \cite{BoucletMizutani} combined with the analysis of \cite{BTglobalstrichartz}. Note that since we consider a larger range of admissible condition comparing to the sharp Schr\"odinger admissible condition (i.e. the inequality in $(\ref{fractional admissible})$ is replaced by the equality) of \cite{BoucletMizutani}, we have to be more careful in order to apply the techniques of \cite{BoucletMizutani}.  \newline
\indent Before giving the main results, let us introduce some notations. Let $g(x)=(g_{jk}(x))_{j,k=1}^d$ be a metric on $\R^d, d\geq 2$, and denote $G(x)=(g^{jk}(x))_{j,k=1}^d:=g^{-1}(x)$. We consider the Laplace-Beltrami operator associated to $g$, i.e.
\[
\Delta_g= \sum_{j,k=1}^{d}|g(x)|^{-1} \partial_{x_j} \left( g^{jk}(x) |g(x)| \partial_{x_k}\right),
\]
where $|g(x)|:=\sqrt{\det g(x)}$. Throughout the paper, we assume that $g$ satisfies the following assumptions.
\begin{enumerate}
\item There exists $C>0$ such that for all $x,\xi\in \R^d$, 
\begin{align}
C^{-1}|\xi|^2 \leq \sum_{j,k=1}^{d}g^{jk}(x)\xi_j\xi_k \leq C |\xi|^2. \label{assump elliptic}
\end{align}
\item There exists $\rho>0$ such that for all $\alpha \in \N^d$, there exists $C_\alpha>0$ such that for all $x \in \R^d$,
\begin{align}
\left| \partial^\alpha \left(g^{jk}(x) -\delta_{jk} \right) \right| \leq C_\alpha \scal{x}^{-\rho -|\alpha|}. \label{assump long range}
\end{align}
\end{enumerate}
The elliptic assumption $(\ref{assump elliptic})$ implies that $|g(x)|$ is bounded from below and above by positive constants. Thus for $1\leq q \leq \infty$, the spaces $L^q(\R^d, d_gx)$ where $d_gx = |g(x)|dx$ and $L^q(\R^d)$ coincide and have equivalent norms. In the sequel, we will use the same notation $L^q(\R^d)$. It is well-known that $-\Delta_g$ is essentially self-adjoint on $C^\infty_0(\R^d)$ under the assumptions $(\ref{assump elliptic})$ and $(\ref{assump long range})$. We denote the unique self-adjoint extension on $L^2(\R^d)$ by $P$. Note that the principal symbol of $P$ is 
\begin{align}
p(x,\xi)=\xi^t G(x)\xi = \sum_{j,k=1}^d g^{jk}(x)\xi_j \xi_k. \label{principle symbol p}
\end{align}
\indent Now let $\gamma \in \R$ and $q \in [1,\infty]$. The inhomogeneous Sobolev space $W^{\gamma,q}_g(\R^d)$ associated to $P$ is defined as a closure of the Schwartz space $\Sch(\R^d)$ under the norm 
\[
\|u\|_{W^{\gamma,q}_g(\R^d)}:= \|\scal{\Lambda_g}^{\gamma}u\|_{L^q(\R^d)}.
\]
It is very useful to recall that for all $\gamma \in \R$ and $q \in (1, \infty)$, there exists $C>1$ such that 
\begin{align}
C^{-1} \|\scal{\Lambda}^{\gamma}u\|_{L^q(\R^d)} \leq \|u\|_{W^{\gamma,q}_g(\R^d)} \leq C \|\scal{\Lambda}^{\gamma}u\|_{L^q(\R^d)}, \label{equivalent sobolev norms}
\end{align}
with $\scal{\Lambda}=\sqrt{1-\Delta}$ where $\Delta$ is the free Laplace operator on $\R^d$. This fact follows from the $L^q$-boundedness of zero order pseudo-differential operators (see e.g \cite[Theorem 3.1.6]{SoggeFIO}). The estimates $(\ref{equivalent sobolev norms})$ allow us to use the Sobolev embedding as on $\R^d$. For the homogeneous Sobolev space associated to $P$, one should be careful since the Schwartz space is not a good candidate due to the singularity at 0 of $\lambda \mapsto |\lambda|^\gamma$. Recall that (see \cite[Appendix]{GinibreVelo85}, \cite[chapter 5]{Triebel} and \cite[Chapter 6]{BerghLofstom}) on $\R^d$, the homogeneous Sobolev space $\dot{W}^{\gamma,q}(\R^d)$ is the closure of $\Lch(\R^d)$ under the norm
\[
\|u\|_{\dot{W}^{\gamma,q}(\R^d)} :=\|\Lambda^\gamma u \|_{L^q(\R^d)},
\]
where 
\[
\Lch(\R^d):= \left\{ u \in \Sch(\R^d) \ | \ D^\alpha \hat{u}(0)=0, \  \forall \alpha \in \N^d \right\}.
\]
Here $\hat{\cdot}$ is the spatial Fourier transform. Since there is no Fourier transform on manifolds, we need to use the spectral theory instead. We denote
\begin{align}
\Lch_g(\R^d):= \left\{\varphi(P) u \ | \ u \in \Sch(\R^d), \varphi \in C^\infty_0((0,\infty)) \right\}. \label{dense space}
\end{align}
We define the homogeneous Sobolev space $\dot{W}^{\gamma,q}_g(\R^d)$ associated to $P$ as the closure of $\Lch_g(\R^d)$ under the norm
\[
\|u\|_{\dot{W}^{\gamma,q}_g(\R^d)}:= \|\Lambda_g^{\gamma} u\|_{L^q(\R^d)}.
\]
When $q=2$, we use $H^\gamma(\R^d), \dot{H}^\gamma(\R^d), H^\gamma_g(\R^d)$ and $\dot{H}^\gamma_g(\R^d)$ instead of $W^{\gamma,2}(\R^d), \dot{W}^{\gamma,2}(\R^d), W^{\gamma,2}
_g(\R^d)$ and $\dot{H}^\gamma_g(\R^d)$ respectively. Thanks to the equivalence $(\ref{equivalent sobolev norms})$, we will only use the usual notation $H^\gamma(\R^d)$ in the sequel. It is important to note (see \cite[Proposition 2.3]{Boucletlowfreq} or \cite[Lemma 2.4]{SoggeWang}) that for $d\geq 2$,
\begin{align}
\|u\|^2_{\dot{H}^1_g(\R^d)} =(\Lambda_gu, \Lambda_g u) = (u, Pu) \simeq \|\nabla u\|^2_{L^2(\R^d)} =\|u\|^2_{\dot{H}^1(\R^d)}. \label{equivalent homogeneous sobolev norms}
\end{align}
\indent By the Stone theorem, the solution to $(\ref{fractional schrodinger equation})$ is given by $u(t)=e^{-it\Lambda_g^\sigma} u_0$. Let $f_0 \in C^\infty_0(\R)$ be such that $f_0=1$ on $[-1,1]$. We split 
\[
u(t)=u_{\text{low}}(t)+u_{\text{high}}(t),
\]
where 
\begin{align}
u_{\text{low}}(t):= f_0(P)e^{-it\Lambda_g^\sigma} u_0, \quad u_{\text{high}}(t)=(1-f_0)(P) e^{-it\Lambda_g^\sigma}u_0. \label{split low high freqs}
\end{align}
We see that $u_{\text{low}}(t)$ and $u_{\text{high}}(t)$ correspond to the low and high frequencies respectively. By the Littlewood-Paley decomposition which is very similar to the one given in \cite[Subsection 4.2]{BoucletMizutani} (see Subsection $\ref{subsection littlewood paley}$), we split the high frequency term into two parts: inside and outside a compact set. Our first result concerns the global in time Strichartz estimates for the high frequency term inside a compact set.
\begin{theorem} \label{theorem strichartz inside compact}
Consider $\R^d, d\geq 2$ equipped with a smooth metric $g$ satisfying $(\ref{assump elliptic}), (\ref{assump long range})$ and assume that the geodesic flow associated to $g$ is non-trapping. Then for all $\chi \in C^\infty_0(\R^d)$ and all $(p,q)$ fractional admissible, there exists $C>0$ such that for all $u_0 \in \Lch_g(\R^d)$,
\begin{align}
\|\chi u_{\emph{high}}\|_{L^p(\R,L^q(\R^d))} \leq C \|u_0\|_{\dot{H}^{\gamma_{p,q}}_g(\R^d)}. \label{strichartz inside compact}
\end{align} 
\end{theorem}
The proof of Theorem $\ref{theorem strichartz inside compact}$ is based on local in time Strichartz estimates and global $L^2$ integrability estimates of the fractional Schr\"odinger operator. This strategy was first used in \cite{StaTata} for the Schr\"odinger equation. We will make use of dispersive estimates given in \cite{Dinhcompact} to get Strichartz estimates with a high frequency spectral cutoff on a small time interval. Thanks to global $L^2$ integrability estimates, we can upgrade these local in time Strichartz estimates in to global in time Strichartz estimates. This strategy depends heavily on the non-trapping condition. We believe that one can improve this result to allow some weak trapped condition such as hyperbolic trapping in \cite{BuGH}. We hope to come back this interesting question in a future work. \newline
\indent Our next result is the following global in time Strichartz estimates for the high frequency term outside a compact set.
\begin{theorem} \label{theorem strichartz outside compact}
Consider $\R^d, d\geq 2$ equipped with a smooth metric $g$ satisfying $(\ref{assump elliptic}), (\ref{assump long range})$ and assume that there exists $M>0$ large enough such that for all $\chi \in C^\infty_0(\R^d)$,
\begin{align}
\|\chi (P-\lambda\pm i0)^{-1}\chi\|_{\Lc(L^2(\R^d))} \lesssim_{\chi} \lambda^{M}, \quad \lambda \geq 1. \label{assump resolvent}
\end{align}
Then there exists $R > 0$ large enough such that for all $(p,q)$ fractional admissible, there exists $C>0$ such that for all $u_0 \in \Lch_g(\R^d)$,
\begin{align}
\|\mathds{1}_{\{|x|>R\}} u_{\emph{high}}\|_{L^p(\R, L^q(\R^d))} \leq C \|u_0\|_{\dot{H}^{\gamma_{p,q}}_g(\R^d)}. \label{strichartz outside compact}
\end{align}
\end{theorem}
The assumption $(\ref{assump resolvent})$ is known to hold in certain trapping situations (see e.g. \cite{Datch}, \cite{NonZwor} or \cite{BuGH}) as well as in the non-trapping case (see \cite{Robert92} or \cite{Vodev}). We remark that under the trapping condition of \cite{Datch}, \cite{NonZwor} or \cite{BuGH}, we have 
\[
\|\chi(P-\lambda \pm i0)^{-1}\chi\|_{\Lc(L^2(\R^d))} \lesssim_{\chi} \lambda^{-1/2} \log \lambda, \quad \lambda \geq 1,
\]
and under non-trapping condition, we have (see e.g. \cite{Burq}, \cite{Robert92}) that
\[
\|\chi (P-\lambda \pm i0)^{-1}\chi\|_{\Lc(L^2(\R^d))} \lesssim_{\chi} \lambda^{-1/2}, \quad \lambda \geq 1.
\]
The proof of Theorem $\ref{theorem strichartz outside compact}$ relies on the so called Isozaki-Kitada parametrix (see \cite{BTglobalstrichartz}) and resolvent estimates given in \cite{BoucletMizutani} using $(\ref{assump resolvent})$. Recall that the Isozaki-Kitada parametrix was first introduced on $\R^d$ to study the scattering theory of Schr\"odinger operators with long range potentials \cite{IsozakiKitada}. It was then modified and successfully used to show the Strichartz estimates for Schr\"odinger equation outside a compact set in many papers (see e.g. \cite{BTlocalstrichartz}, \cite{BTglobalstrichartz}, \cite{Bouclethyper}, \cite{Mizutani}, \cite{Mizutaniscattering} or \cite{BoucletMizutani}). \newline
\indent The low frequency term in $(\ref{split low high freqs})$ enjoys the following global in time Strichartz estimates. 
\begin{theorem} \label{theorem strichartz low freq}
Consider $\R^d, d\geq 3$ equipped with a smooth metric $g$ satisfying $(\ref{assump elliptic}), (\ref{assump long range})$. Then for all $(p,q)$ fractional admissible, there exists $C>0$ such that for all $u_0 \in \Lch_g(\R^d)$,
\begin{align}
\|u_{\emph{low}}\|_{L^p(\R,L^q(\R^d))} \leq C \|u_0\|_{\dot{H}^{\gamma_{p,q}}_g(\R^d)}. \label{strichartz low freq}
\end{align}
\end{theorem}
As mentioned earlier, since we consider a larger range of admissible condition than the one studied in \cite{BoucletMizutani}, we can not apply directly the low frequency Littlewood-Paley decomposition given in \cite{BoucletMizutani}. We thus need a ``refined'' version of Littlewood-Paley decomposition. To do so, we will take the advantage of heat kernel estimates (see Subsection $\ref{subsection littlewood paley}$). As a result, we split the low frequency term into two parts: one supported outside a compact set and another one localized in a weak sense, i.e. by means of a spatial decay weight. The term with a spatial decay weight is treated easily by using global $L^p$ integrability estimates of the fractional Schr\"odinger operator at low frequency. Note that this type of global $L^p$ integrability estimate relies on the low frequency resolvent estimates of \cite{BoucletRoyer} which is available for spatial dimensions $d\geq 3$. We expect that global in time Strichartz estimates for the fractional Schr\"odinger equation at low frequency may hold in dimension $d=2$ as well. However, we do not know how to prove it at the moment. For the term outside a compact set, we make use of microlocal techniques and a low frequency version of the Isozaki-Kitada parametrix. We refer the reader to Section $\ref{section strichartz estimates outside compact}$ for more details. \newline
\indent Combining Theorem $\ref{theorem strichartz inside compact}$, Theorem $\ref{theorem strichartz outside compact}$ and Theorem $\ref{theorem strichartz low freq}$, we have the following result.
\begin{theorem} \label{theorem global strichartz frac schro}
Consider $\R^d, d\geq 3$ equipped with a smooth metric $g$ satisfying $(\ref{assump elliptic}), (\ref{assump long range})$ and assume that the geodesic flow associated to $g$ is non-trapping. Let $u$ be a weak solution to $(\ref{fractional schrodinger equation})$. Then for all $(p,q)$ fractional admissible, there exists $C>0$ such that for all $u_0 \in \Lch_g(\R^d)$,
\begin{align}
\|u\|_{L^p(\R,L^q(\R^d))} \leq C \|u_0\|_{\dot{H}^{\gamma_{p,q}}_g(\R^d)}. \label{global strichartz frac schro}
\end{align} 
\end{theorem}
Using the homogeneous Strichartz estimate $(\ref{global strichartz frac schro})$ and the Christ-Kiselev Lemma, we get the following inhomogeneous Strichartz estimates.
\begin{prop} \label{prop global inhomogeneous strichartz frac schro}
Consider $\R^d, d\geq 3$ equipped with a smooth metric $g$ satisfying $(\ref{assump elliptic}), (\ref{assump long range})$ and assume that the geodesic flow associated to $g$ is non-trapping. Let $\sigma \in (0,\infty)\backslash \{1\}$ and $u$ be a weak solution to the Cauchy problem
\begin{align}
\left\{
\begin{array}{rcl}
i\partial_t u(t,x) - \Lambda_g^\sigma u(t,x)&=& F(t,x), \quad (t,x) \in \R \times \R^d, \\
u(0,x)&=&u_0(x), \quad x \in \R^d,
\end{array}
\right.
\label{linear fractional schrodinger equation}
\end{align} 
with data $u_0 \in \Lch_g$ and $F\in C(\R, \Lch_g)$. Then for all $(p,q)$ and $(a,b)$ fractional admissible, there exists $C>0$ such that 
\begin{align}
\|u\|_{L^p(\R,L^q(\R^d))} + \|u\|_{L^\infty(\R, \dot{H}^{\gamma_{p,q}}_g(\R^d))} \leq C \Big( \|u_0\|_{\dot{H}^{\gamma_{p,q}}_g(\R^d)} + \|F\|_{L^{a'}(\R,L^{b'}(\R^d))} \Big), \label{global inhomogeneous strichartz frac schro}
\end{align} 
provided that $(p,a) \ne (2,2)$ and 
\begin{align}
\gamma_{p,q}=\gamma_{a',b'}+\sigma. \label{gap condition frac schro}
\end{align}
\end{prop}
\begin{rem}\label{rem global inhomogeneous strichartz frac schro}
\begin{itemize}
\item[1.] The homogeneous Strichartz estimates $(\ref{global strichartz frac schro})$ and the Minkowski inequality imply
\begin{align}
\|u\|_{L^p(\R,L^q(\R^d))} \leq C \Big(\|u_0\|_{\dot{H}^{\gamma_{p,q}}_g(\R^d)} + \|F\|_{L^1(\R, \dot{H}^{\gamma_{p,q}}_g(\R^d))}\Big). \label{global inhomogeneous strichartz frac schro remark}
\end{align}
\item[2.] When $\sigma \in (0,2)\backslash \{1\}$, we always have $\gamma_{p,q}>0$ for any fractional admissible pair $(p,q)$ except $(p,q)=(\infty,2)$. Thus, condition $(\ref{gap condition frac schro})$ implies that $(p,a) \ne (2,2)$, and $(\ref{global inhomogeneous strichartz frac schro})$ includes the endpoint case. When $\sigma\geq2$, the estimates $(\ref{global inhomogeneous strichartz frac schro})$ do not include the endpoint estimate.
\item[3.] In the case $\sigma \in (0,2]\backslash \{1\}$, we can replace the homogeneous Sobolev norms in $(\ref{global inhomogeneous strichartz frac schro})$ and $(\ref{global inhomogeneous strichartz frac schro remark})$ by the inhomogeneous ones.
\end{itemize}
\end{rem}
\begin{prop} \label{prop global inhomogeneous strichartz frac wave}
Consider $\R^d, d\geq 3$ equipped with a smooth metric $g$ satisfying $(\ref{assump elliptic}), (\ref{assump long range})$ and assume that the geodesic flow associated to $g$ is non-trapping. Let $\sigma\in (0,\infty)\backslash \{1\}$ and $v$ be a weak solution to the Cauchy problem
\begin{align}
\left\{
\begin{array}{rcl}
\partial^2_t v(t,x) + \Lambda_g^{2\sigma} v(t,x) &=&F(t,x), \quad (t,x) \in \R \times \R^d, \\
v(0,x)=v_0(x), && \partial_t v(0,x)=v_1(x), \quad x \in \R^d,
\end{array}
\right.
\label{linear fractional wave equation}
\end{align} 
with data $v_0, v_1 \in \Lch_g$ and $F \in C(\R, \Lch_g)$. Then for all $(p,q)$ and $(a,b)$ fractional admissible, there exists $C>0$ such that 
\begin{align}
\|v\|_{L^p(\R,L^q(\R^d))} + \|[v]\|_{L^\infty(\R, \dot{H}^{\gamma_{p,q}}_g(\R^d))} 
\leq C \left( \|[v](0)\|_{\dot{H}^{\gamma_{p,q}}_g(\R^d)} + \|F\|_{L^{a'}(\R,L^{b'}(\R^d))} \right), \label{global inhomogeneous strichartz frac wave}
\end{align} 
where $[v](t):=(v(t), \partial_t v(t))$ and 
\[
\|[v]\|_{L^\infty(\R, \dot{H}^{\gamma_{p,q}}_g(\R^d))}:=\|v\|_{L^\infty(\R, \dot{H}^{\gamma_{p,q}}_g(\R^d))} +\|\partial_t v\|_{L^\infty(\R,\dot{H}^{\gamma_{p,q}-\sigma}_g(\R^d))}
\]
provided that $(p,a) \ne (2,2)$ and 
\begin{align}
\gamma_{p,q}=\gamma_{a',b'}+2\sigma. \label{gap condition frac wave}
\end{align}
\end{prop}
\begin{rem}\label{rem global inhomogeneous strichartz frac wave}
As in Remark $\ref{rem global inhomogeneous strichartz frac schro}$, we have
\begin{align}
\|v\|_{L^p(\R,L^q(\R^d))} \leq C \Big(\|[v](0)\|_{\dot{H}^{\gamma_{p,q}}_g(\R^d)} + \|F\|_{L^1(\R, \dot{H}^{\gamma_{p,q}-\sigma}_g(\R^d))}\Big). \label{global inhomogeneous strichartz frac wave remark}
\end{align}
\end{rem}
We finally emphasize that this paper is only devoted to study global in time Strichartz estimates. We refer the reader to \cite{Dinh} for applications of Strichartz estimates to the local well-posedness of the nonlinear fractional Schr\"odinger and wave equations. \newline
\indent The paper is organized as follows. In Section 2, we recall some properties of the semi-classical and rescaled pseudo-differential operators, and prove some propagation estimates related to our problem. We then prove a ``refined'' version of the Littlewood-Paley decomposition at low frequency and give a reduction of global in time Strichartz estimates in Section 3. In Section 4, we prove global in time Strichartz estimates inside compact sets at high frequency. Section 5 is devoted to the construction of the Isozaki-Kitada parametrix for the fractional Schr\"odinger operator both at high and low frequencies and to the proofs of Strichartz estimates outside compact sets. Finally, we give in Section 6 the proofs for the inhomogeneous Strichartz estimates due to the Christ-Kiselev Lemma. \newline
\textbf{Notation.} In this paper the constant may change from line to line and will be denoted by the same $C$. The constants with a subscript $C_1,C_2,...$ will be used when we need to compare them to one another. The notation $A \lesssim_D B$ means that there exists a universal constant $C>0$ depending on $D$ such that $A \leq C B$. For Banach spaces $X$ and  $Y$, the notation $\|\cdot\|_{\Lc(X,Y)}$ denotes the operator norm from $X$ to $Y$ and $\|\cdot\|_{\Lc(X)}:=\|\cdot\|_{\Lc(X,X)}$. The one $T=O_{\Lc(X,Y)}(A)$ means that there exists $C>0$ such that $\|T\|_{\Lc(X,Y)} \leq C A$. In order to simplify the presentation, we shall drop the notation $\R^d$ and only write $L^q, \Sch, \Lch_g, W^{\gamma,q},  \dot{W}^{\gamma,q}_g, H^\gamma, \dot{H}^\gamma_g$. 
\section{Functional calculus and propagation estimates} \label{section functional calculus}
\setcounter{equation}{0}
In this section, we recall some well-known results on pseudo-differential operators and prove some propagation estimates related to our problem. 
\subsection{Pseudo-differential operators.}
Let $\mu, m \in \R$. We consider the symbol class $S(\mu, m)$ the space of smooth functions $a$ on $\R^{2d}$ satisfying 
\[
\left|\partial^\alpha_x \partial^\beta_\xi  a(x,\xi)\right| \leq C_{\alpha\beta} \scal{x}^{\mu-|\alpha|} \scal{\xi}^{m-|\beta|}.
\]
In practice, we mainly use $S(\mu,-\infty):= \cap_{m\in \R} S(\mu,m)$. \newline
\indent For $a \in S(\mu,m)$ and $h \in (0,1]$, we consider the \textbf{semi-classical pseudo-differential operator} $Op^h(a)$ which is defined by
\begin{align}
Op^h(a) u(x)= (2\pi h)^{-d}\iint_{\R^{2d}} e^{ih^{-1}(x-y)\xi} a(x,\xi) u(y)dyd\xi. \label{semi PDO definition}
\end{align}
By the long range assumption $(\ref{assump long range})$, we see that $h^2P= Op^h(p)+hOp^h(p_1)$ with $p \in S(0,2)$ given in $(\ref{principle symbol p})$ and $p_1 \in S(-\rho-1,1) \subset S(-1,1)$. We recall that for $a \in S(\mu_1, m_1)$ and $b\in S(\mu_2,m_2)$, the composition $Op^h(a)Op^h(b)$ is given by
\begin{align}
Op^h(a)Op^h(b) = \sum_{j=0}^{N-1} h^j Op^h((a\#b)_j) +h^N Op^h(r^\#_N(h)), \label{composition PDO}
\end{align}
where $(a\#b)_j= \sum_{|\alpha|=j} \frac{1}{\alpha!} \partial^\alpha_\xi a D^\alpha_x b \in S(\mu_1+\mu_2-j,m_1+m_2-j)$ and $(r^\#_N(h))_{h\in(0,1]}$ is a bounded family in $S(\mu_1+\mu_2-N,m_1+m_2-N)$. The adjoint with respect to the Lebesgue measure $Op^h(a)^\star$ is given by
\begin{align}
Op^h(a)^\star = \sum_{j=0}^{N-1} h^j Op^h(a^\star_j) +h^N Op^h(r^\star_N(h)), \label{adjoint PDO}
\end{align}
where $a^\star_j= \sum_{|\alpha|=j} \frac{1}{\alpha!} \partial^\alpha_\xi D^\alpha_x \overline{a} \in S(\mu_1-j,m_1-j)$ and $(r^\star_N(h))_{h\in(0,1]}$ is a bounded family in $S(\mu_1-N,m_1-N)$. \newline
\indent We next recall the definition of rescaled pseudo-differential operator which is essentially given in \cite{BoucletMizutani}. This type of operator is very useful for the analysis at low frequency. Let $a \in S(\mu,m)$ and $\ep \in (0,1]$. The \textbf{rescaled pseudo-differential operator} $Op_\ep(a)$ is defined by
\[
Op_\ep(a) u(x) = (2\pi)^{-d} \iint_{\R^{2d}} e^{i(x-y)\xi} a(\ep x, \ep^{-1}\xi) u(y) dy d\xi.
\]
Setting $D_\ep u(x):= \ep^{d/2}u(\ep x)$. It is easy to see that $D_\ep$ is a unitary map on $L^2$ and
\begin{align}
Op_\ep(a)= D_\ep Op(a) D^{-1}_\ep, \label{define D_epsilon}
\end{align}
where $D^{-1}_\ep u(x)=\ep^{-d/2}u(\ep^{-1}x)$ and $Op(a):=Op^1(a)$, i.e. $h=1$ in $(\ref{semi PDO definition})$. Thanks to $(\ref{composition PDO}), (\ref{adjoint PDO})$ and $(\ref{define D_epsilon})$, the composition $Op_\ep(a)Op_\ep(b)$ and the adjoint with respect to the Lebesgue measure $Op_\ep(a)^\star$ with $a \in S(\mu_1, m_1)$ and $b\in S(\mu_2,m_2)$ are given by
\[
Op_\ep(a)Op_\ep(b) = \sum_{j=0}^{N-1} Op_\ep((a\#b)_j) +Op_\ep(r^\#_N),\quad Op_\ep(a)^\star = \sum_{j=0}^{N-1} Op_\ep(a^\star_j) + Op_\ep(r^\star_N).
\]
\subsection{Functional calculus.}
In this subsection, we will recall the approximations for $\phi(h^2P)$ and $\zeta(\ep x) \phi(\ep^{-2}P)$ in terms of semi-classical and rescaled pseudo-differential operators respectively where $\phi \in C^\infty_0(\R)$ and $\zeta \in C^\infty(\R^d)$ is supported outside $B(0,1)$ and equal to 1 near infinity. Here $B(0,1)$ is the open unit ball in $\R^d$.\newline
\indent We firstly recall the following $\Lc(L^q, L^r)$-bound of pseudo-differential operators (see e.g. \cite[Proposition 2.4]{BTlocalstrichartz}). 
\begin{prop} \label{prop lq lr bounds PDO}
Let $m >d$ and $a$ be a continuous function on $\R^{2d}$ smooth with respect to the second variable satisfying for all $\beta \in \N^d$, there exists $C_\beta>0$ such that for all $x, \xi \in \R^d$,
\[
|\partial^\beta_\xi a(x,\xi)| \leq C_\beta \scal{\xi}^{-m}.
\]
Then for $1 \leq q \leq r \leq \infty$, there exists $C>0$ such that for all $h \in (0,1]$,
\[
\|Op^h(a)\|_{\Lc(L^q, L^r)} \leq C h^{d/r-d/q}.
\]
\end{prop}
The following proposition gives an approximation of $\phi(h^2P)$ in terms of semi-classical pseudo-differential operators (see e.g. \cite{BTlocalstrichartz} or \cite{Robert}).
\begin{prop} \label{prop parametrix phi}
Consider $\R^d$ equipped with a smooth metric $g$ satisfying $(\ref{assump elliptic})$ and $(\ref{assump long range})$. Then for a given $\phi \in C^\infty_0(\R)$, there exist a sequence of symbols $q_j \in S(-j,-\infty)$ satisfying $q_0= \phi \circ p$ and $\emph{supp}(q_j) \subset \emph{supp}(\phi \circ p)$ such that for all $N \geq 1$,
\[
\phi(h^2P)= \sum_{j=0}^{N-1} h^j Op^h(q_j)+ h^N R_N(h),
\]
and for $m \geq 0$ and $1 \leq q \leq r \leq \infty$, there exists $C>0$ such that for all $h \in (0,1]$,
\begin{align}
\|R_N(h) \scal{x}^N\|_{\Lc(L^q, L^r)} &\leq C h^{d/r-d/q}, \nonumber \\
\|R_N(h) \scal{x}^N\|_{\Lc(H^{-m}, H^m)} &\leq C h^{-2m}. \nonumber
\end{align}
\end{prop}
Combining Proposition $\ref{prop lq lr bounds PDO}$ and Proposition $\ref{prop parametrix phi}$, one has the following result (see e.g. \cite[Proposition 2.9]{BTlocalstrichartz}).
\begin{prop} \label{prop Lq Lr bound of phi}
Consider $\R^d$ equipped with a smooth metric $g$ satisfying $(\ref{assump elliptic})$ and $(\ref{assump long range})$. Let $\phi \in C^\infty_0(\R)$. Then for $1 \leq q \leq r \leq \infty$, there exists $C>0$ such that for all $h \in (0,1]$,
\[
\|\phi(h^2P)\|_{\Lc(L^q, L^r)} \leq C h^{d/r-d/q}.
\]
\end{prop}
It is also known (see \cite{BoucletMizutani}) that the rescaled pseudo-differential operator is very useful to approximate the low frequency localization of $P$, i.e. operators of the form $\phi(\ep^{-2}P)$. By the uncertainty principle, one can only expect to get such approximation whenever $|x|$ is large, typically $|x| \gtrsim \ep^{-1}$. 
\begin{rem}\label{rem property rescaled symbol}
Let  $\mu \leq 0, m \in \R$ and $a \in S(\mu,m)$. If we set
\[
a_\ep(x,\xi):= \ep^\mu a(\ep^{-1}x,\xi),
\]
then for all $\alpha,\beta \in \N^d$, there exists $C_{\alpha\beta}>0$ such that for all $|x| \geq 1, \xi \in \R^d$,
\[
|\partial^\alpha_x \partial^\beta_\xi a_\ep(x,\xi)| \leq C_{\alpha\beta} \scal{\xi}^{m-|\beta|}, \quad \forall \ep \in (0,1].
\]
\end{rem}
We next rewrite $\ep^{-2}P$ as $D_\ep (D_\ep^{-1}(\ep^{-2}P)D_\ep)D_\ep^{-1}$. A direct computation gives
\[
D_\ep^{-1}(\ep^{-2}P)D_\ep = Op(p_\ep)+Op(p_{\ep,1})=:P_\ep,
\]
where $p_{\ep}(x,\xi) = p(\ep^{-1}x,\xi)$ and $p_{\ep,1}(x,\xi)=\ep^{-1}p_1(\ep^{-1}x,\xi)$. We thus obtain
\begin{align}
\ep^{-2}P = Op_\ep(p_\ep)+ Op_\ep(p_{\ep,1}). \label{define p_epsilon}
\end{align}
Using the fact that $p\in S(0,2), p_1\in S(-1,1)$, Remark $\ref{rem property rescaled symbol}$ allows us to construct the parametrix for the resolvent $\zeta(\ep x)(\ep^{-2}P-z)^{-k}$ with $\zeta \in C^\infty(\R^d)$ supported outside $B(0,1)$ and equal to 1 near infinity. Indeed, by writing $\zeta(\ep x)(\ep^{-2}P-z)^{-k} = D_\ep \left[\zeta(x) ( P_\ep -z )^{-k} \right] D_\ep^{-1}$, we can apply the standard elliptic parametrix for $\zeta(x) ( P_\ep -z )^{-k}$ and we have (see e.g. \cite{BTlocalstrichartz} or \cite{BoucletMizutani}) the following result.
\begin{prop} \label{prop low freq parametrix} 
Let $\zeta, \tilde{\zeta}, \tilde{\tilde{\zeta}} \in C^\infty(\R^d)$ be supported outside $B(0,1)$ and equal to $1$ near infinity such that $\tilde{\zeta}=1$ near $\emph{supp}(\zeta)$ and $\tilde{\tilde{\zeta}} =1$ near $\emph{supp}(\tilde{\zeta})$. Then for all $k, N \geq 1$ integers and $z \in \C \backslash [0,+\infty)$, we have for $\ep \in (0,1]$,
\[
\zeta(\ep x) (\ep^{-2}P -z)^{-k}= \sum_{j=0}^{N-1} \zeta(\ep x) Op_\ep (b_{\ep,j}(z)) \tilde{\zeta}(\ep x) + R_N(z,\ep),
\]
where $(b_{\ep,j}(z))_{\ep \in (0,1]}$ is a bounded family in $S(-j,-2k-j)$ which is a linear combination of $d_{\ep,l} (p_\ep-z)^{-k-l}$ with $(d_{\ep,l})_{\ep \in (0,1]}$ a bounded family in $S(-j,2l-j)$ and 
\[
R_N(z,\ep)= \zeta(\ep x) Op_\ep(r_N(z,\ep))\tilde{\tilde{\zeta}}(\ep x) (\ep^{-2}P -z)^{-k}
\]
where $r_N(z,\ep) \in S(-N,-N)$ has seminorms growing polynomially in $1/\emph{dist}(z,\R^+)$ uniformly in $\ep \in (0,1]$ as long as $z$ belongs to a bounded set of $\C \backslash [0,+\infty)$.
\end{prop}
A first application of Proposition $\ref{prop low freq parametrix}$ is the following result.
\begin{prop} \label{prop L2 Lq resolvent low freq}
Using the notations given in $\emph{Proposition}$ $\ref{prop low freq parametrix}$, let $k>d/2$ and $2\leq q \leq \infty$. Then there exists $C>0$ such that for all $\ep \in (0,1]$,
\begin{align}
\|\zeta(\ep x) (\ep^{-2}P+1)^{-k}\|_{\Lc(L^{2}, L^{q})} \leq C \ep^{d/2-d/q}. \label{L2 Lq estimate low freq parametrix}
\end{align}
\end{prop}
\begin{proof}
We apply Proposition $\ref{prop low freq parametrix}$ with $N>d$, we see that 
\begin{align*}
\zeta(\ep x) (\ep^{-2}P+1)^{-k}&= \sum_{j=0}^{N-1} \zeta(\ep x) Op_\ep (b_{\ep,j}(-1)) \tilde{\zeta}(\ep x) + \zeta(\ep x) Op_\ep(r_N(-1,\ep))\tilde{\tilde{\zeta}}(\ep x) (\ep^{-2}P +1)^{-k}, \nonumber \\
&= \sum_{j=0}^{N-1} D_\ep \left\{\zeta(x) Op(b_{\ep,j}(-1)) \tilde{\zeta}(x)+ \zeta(x) Op(r_N(-1,\ep))\tilde{\tilde{\zeta}} (x) (P_\ep +1)^{-k}\right\} D_\ep^{-1}, \nonumber
\end{align*}
where $(b_{\ep,j}(-1))_{\ep \in (0,1]}, (r_N(-1,\ep))_{\ep\in (0,1]}$ are bounded in $S(-j,-2k-j)$ and $S(-N,-N)$ respectively. The result then follows from Proposition $\ref{prop lq lr bounds PDO}$ with $h=1$ and that 
\[
\|D_\ep\|_{\Lc(L^q)} = \ep^{d/2-d/q},  \quad \|D_\ep^{-1}\|_{\Lc(L^2)} = 1.
\]
We also use that $\|(P_\ep+1)^{-k}\|_{\Lc(L^2)} \leq 1$ for the remainder term.
\end{proof}
Another application of Proposition $\ref{prop low freq parametrix}$ is the following approximation of $\zeta(\ep x)\phi(\ep^{-2}P)$ in terms of rescaled pseudo-differential operators.
\begin{prop} \label{prop parametrix phi low freq}
Consider $\R^d$ equipped with a smooth metric $g$ satisfying $(\ref{assump elliptic})$ and $(\ref{assump long range})$. Let $\phi \in C^\infty_0(\R)$ and $\zeta, \tilde{\zeta}, \tilde{\tilde{\zeta}}$ be as in \emph{Proposition } $\ref{prop low freq parametrix}$. Then there exists a sequence of bounded families of symbols $(q_{\ep,j})_{\ep \in (0,1]} \in S(-j,-\infty)$ with $q_{\ep,0}= \phi\circ p_\ep$ and $\emph{supp}(q_{\ep,j}) \subset \emph{supp}(\phi \circ p_\ep)$ such that for all $N \geq 1$,
\begin{align}
\zeta(\ep x)\phi(\ep^{-2}P)= \sum_{j=0}^{N-1} \zeta(\ep x) Op_\ep(q_{\ep,j}) \tilde{\zeta}(\ep x)+ R_N(\ep). \label{parametrix phi low freq}
\end{align}
Moreover, for any $m \geq 0$, there exists $C>0$ such that for all $\ep \in (0,1]$,
\begin{align}
\| (\ep^{-2}P+1)^m R_N(\ep) \scal{\ep x}^{N} \|_{\Lc(L^2)} \leq C. \label{estimate remainder low freq}
\end{align}
\end{prop}
\begin{proof}
By using Proposition $\ref{prop low freq parametrix}$ with $k=1$ and the Helffer-Sj\"ostrand formula (see \cite{DimaSjos}) namely
\[
\phi(\ep^{-2}P) = -\frac{1}{\pi} \int_{\C} \overline{\partial} \widetilde{\phi}(z) 
(\ep^{-2}P-z)^{-1}dL(z),
\]
where $\widetilde{\phi}$ is an almost analytic extension of $\phi$, the Cauchy formula gives $(\ref{parametrix phi low freq})$ with 
\begin{align}
R_N(\ep)= \frac{1}{\pi} \int_{\C} \overline{\partial} \widetilde{\phi}(z) \zeta(\ep x) Op_\ep(r_N(z,\ep))\tilde{\tilde{\zeta}}(\ep x) (\ep^{-2}P -z)^{-1} dL(z). \label{explicit form remainder}
\end{align}
Here $(r_N(z,\ep))_{\ep \in (0,1]}$ is bounded in $S(-N,-N)$ and has semi-norms growing polynomially in $|\im{z}|^{-1}$ which is harmless since 
$\overline{\partial} \widetilde{\phi}(z) =O(|\im{z}|^{\infty})$. The left hand side of $(\ref{estimate remainder low freq})$ is bounded by
\[
\frac{1}{\pi} \int_{\C} |\overline{\partial}\widetilde{\phi}(z)| \|  (\ep^{-2}P+1)^m \zeta(\ep x) Op_\ep(r_N(z,\ep)) \tilde{\tilde{\zeta}}(\ep x) (\ep^{-2}P-z)^{-1} \scal{\ep x}^{N} \|_{\Lc(L^2)}dL(z).
\]
By choosing $\zeta_1 \in C^\infty(\R^d)$ supported outside $B(0,1)$ such that $\zeta_1=1$ near $\text{supp}(\tilde{\tilde{\zeta}})$, we can write
\[
(\ep^{-2}P-z)^{-1}=(\ep^{-2}P-z)^{-1} (1-\zeta_1)(\ep x)+(\ep^{-2}P-z)^{-1} \zeta_1(\ep x).
\] 
We note that $(1-\zeta_1)(\ep x) \scal{\ep x}^N$ is of size $O_{\Lc(L^2)}(1)$ due to the compact support in $\ep x$, and $(\ep^{-2}P+1)(\ep^{-2}P-z)^{-1}$ is of size $O_{\Lc(L^2)}(|\im{z}|^{-1})$ by functional calculus. Moreover, using $(\ref{define p_epsilon})$ and the same process as in the proof of Proposition $\ref{prop L2 Lq resolvent low freq}$, there exists $\tau(m)\in \N$ such that 
\[
\|(\ep^{-2}P+1)^m \zeta(\ep x) Op_\ep(r_N(z,\ep)) \tilde{\tilde{\zeta}}(\ep x) (\ep^{-2}P+1)^{-1}\|_{\Lc(L^2)} \leq C |\im{z}|^{-\tau(m)}.
\]
This shows that 
\begin{align}
\|(\ep^{-2}P+1)^m R_N(\ep) (1-\zeta_1)(\ep x) \scal{\ep x}^N\|_{\Lc(L^2)} \leq C. \label{estimate term 1}
\end{align}
For the term $(\ep^{-2}P+1)^m R_N(\ep) \zeta_1(\ep x) \scal{\ep x}^N$, using Proposition $\ref{prop low freq parametrix}$ (by taking the adjoint), we see that
\[
(\ep^{-2}P-z)^{-1}\zeta_1(\ep x) = \sum_{j=0}^{{N'}-1} \tilde{\zeta}_1(\ep x) Op_\ep (\tilde{b}_{\ep,j}(z)) \zeta_1(\ep x) + \tilde{R}_{N'}(z,\ep),
\]
where $(\tilde{b}_{\ep,j}(z))_{\ep \in (0,1]}$ is a bounded family in $S(-j,-2-j)$ and
\[
\tilde{R}_{N'}(z,\ep)= (\ep^{-2}P -z)^{-1}\tilde{\tilde{\zeta}}_1(\ep x) Op_\ep(\tilde{r}_{N'}(z,\ep))\zeta_1(\ep x),
\]
where $\tilde{r}_{N'}(z,\ep) \in S(-N',-N')$ has seminorms growing polynomially in $|\im{z}|^{-1}$ uniformly in $\ep \in (0,1]$. By the same argument as above, we obtain
\begin{align}
\|(\ep^{-2}P+1)^m R_N(\ep) \zeta_1(\ep x) \scal{\ep x}^N\|_{\Lc(L^2)} \leq C. \label{estimate term 2}
\end{align}
Combining $(\ref{estimate term 1})$ and $(\ref{estimate term 2})$, we prove $(\ref{estimate remainder low freq})$.
\end{proof}
As a consequence of Proposition $\ref{prop parametrix phi low freq}$, we have the following result.
\begin{coro} \label{coro Lq Lr estimate of parametrix low freq}
Let $\phi \in C^\infty_0(\R)$. Then for $2 \leq q \leq r \leq \infty$, there exists $C>0$ such that for all $\ep \in (0,1]$,
\begin{align}
\|\zeta(\ep x) \phi (\ep^{-2}P)\|_{\Lc(L^q, L^r)} &\leq C \ep^{d/q-d/r}. \label{estimate low freq appro Lq-Lr} 
\end{align}
\end{coro}
\begin{proof}
By $(\ref{parametrix phi low freq})$ and $(\ref{estimate remainder low freq})$ (see also $(\ref{explicit form remainder})$), we can write for any $N\geq 1$ and any $m\geq 0$,
\[
\zeta(\ep x)\phi(\ep^{-2}P)= \sum_{j=0}^{N-1} \zeta(\ep x) Op_\ep(q_{\ep,j}) \tilde{\zeta}(\ep x)+ R_N(\ep),
\]
where
\[
R_N(\ep)=\tilde{\zeta}(\ep x)(\ep^{-2}P+1)^{-m} B_\ep \scal{\ep x}^{-N}
\]
with $B_\ep = O_{\Lc(L^2)}(1)$ uniformly in $\ep \in (0,1]$. The main terms can be estimated by using Proposition $\ref{prop lq lr bounds PDO}$ (see also the proof of Proposition $\ref{prop L2 Lq resolvent low freq}$). It remains to treat the remainder term. We firstly note that $\scal{\ep x}^{-N}=O_{\Lc(L^q,L^2)}(\ep^{d/q-d/2})$ provided $N>\frac{d(q-2)}{2q}$. Using this bound together with $B_\ep = O_{\Lc(L^2)}(1)$ and $(\ref{L2 Lq estimate low freq parametrix})$, we see that
\begin{align*}
\|R_N(\ep)\|_{\Lc(L^q,L^r)} &\lesssim \|\tilde{\zeta}(\ep x)(\ep^{-2}P+1)^{-m}\|_{\Lc(L^2,L^r)} \|B_\ep\|_{\Lc(L^2)} \|\scal{\ep x}^{-N}\|_{\Lc(L^q,L^2)} \\
&\lesssim \ep^{d/2-d/r} \ep^{d/q-d/2} \lesssim \ep^{d/q-d/r}.
\end{align*}
This proves $(\ref{estimate low freq appro Lq-Lr})$.
\end{proof}
Another consequence of Proposition $\ref{prop parametrix phi low freq}$ is the following estimate.
\begin{coro} \label{coro rescaled speudo-differential}
Let $\phi \in C^\infty_0(\R)$. For $m\geq 0$, there exists $C>0$ such that for all $\ep \in (0,1]$,
\begin{align}
\|\scal{\ep x}^{-m} \phi(\ep^{-2}P) \scal{\ep x}^{m}\|_{\Lc(L^2)} \leq C. \label{L2 bound low freq}
\end{align}
\end{coro}
\begin{proof}
By choosing $\zeta \in C^\infty(\R^d)$ supported outside $B(0,1)$ and equal to $1$ near infinity, we can write $\scal{\ep x}^{-m} \phi(\ep^{-2}P) \scal{\ep x}^{m}$ as
\[
\scal{\ep x}^{-m} \phi(\ep^{-2}P) \zeta(\ep x)\scal{\ep x}^{m}+ \scal{\ep x}^{-m} \phi(\ep^{-2}P) (1-\zeta)(\ep x)\scal{\ep x}^{m}.
\]
The $\Lc(L^2)$-boundedness of the first term follows from the parametrix of $\phi(\ep^{-2}P) \zeta(\ep x)$ which is obtained by taking the adjoint of $(\ref{parametrix phi low freq})$. The second term follows from the fact that $(1-\zeta)(\ep x)\scal{\ep x}^{m}$ is bounded in $\Lc(L^2)$ since $1-\zeta$ vanishes outside a compact set.
\end{proof}
\subsection{Propagation estimates.} 
In this subsection, we recall some results on resolvent estimates and prove some propagation estimates both at high and low frequencies. Let us start with the following result.
\begin{prop} \label{prop power resolvent estimates}
\begin{itemize}
\item[1.] Consider $\R^d, d\geq 2$ equipped with a smooth metric $g$ satisfying $(\ref{assump elliptic}), (\ref{assump long range})$ and suppose that the assumption $(\ref{assump resolvent})$ holds. Then for $k \geq 0$, there exist $C>0$ and non-decreasing $N_k \in \N$ such that for all $h\in (0,1]$ and all $\lambda$ belongs to a relatively compact interval of $(0,+\infty)$,
\begin{align}
\|\scal{x}^{-1-k} (h^2P - \lambda \mp i0)^{-1-k} \scal{x}^{-1-k}\|_{\Lc(L^2)} \leq C h^{-N_k}. \label{power resolvent estimates}
\end{align}
\item[2.] Consider $\R^d, d\geq 3$ equipped with a smooth metric $g$ satisfying $(\ref{assump elliptic}), (\ref{assump long range})$. Then for $k \geq 0$, there exists $C>0$ such that for all $\ep \in (0,1]$ and all $\lambda$ belongs to a relatively compact interval of $(0,+\infty)$,
\begin{align}
\|\scal{\ep x}^{-1-k} (\ep^{-2}P - \lambda \mp i0)^{-1-k} \scal{\ep x}^{-1-k}\|_{\Lc(L^2)} \leq C. \label{power resolvent estimates low freq}
\end{align}
\end{itemize}
\end{prop} 
The high frequency resolvent estimates $(\ref{power resolvent estimates})$ are given in \cite[Proposition 7.5]{BoucletMizutani} and the low frequency resolvent estimates $(\ref{power resolvent estimates low freq})$ are given in \cite[Theorem 1.2]{BoucletRoyer}. Note that under the non-trapping condition, the estimates $(\ref{power resolvent estimates})$ hold with $N_k =k+1$ (see e.g. \cite[Theorem 2.8]{Robertsmoothing}). We next use the resolvent estimates given in Proposition $\ref{prop power resolvent estimates}$ to have the following resolvent estimates for the fractional Schr\"odinger operator.
\begin{prop} \label{prop power resolvent estimates fractional schrodinger}
Let $\sigma \in (0,\infty)$.
\begin{itemize}
\item[1.] Consider $\R^d, d\geq 2$ equipped with a smooth metric $g$ satisfying $(\ref{assump elliptic}), (\ref{assump long range})$ and suppose that the assumption $(\ref{assump resolvent})$ holds. Then for $k \geq 0$, there exist $C>0$ and non-decreasing $N_k \in \N$ such that for all $h\in (0,1]$ and all $\mu \in I \Subset (0,+\infty)$,
\begin{align}
\|\scal{x}^{-1-k} ((h\Lambda_g)^\sigma - \mu \mp i0)^{-1-k} \scal{x}^{-1-k}\|_{\Lc(L^2)} \leq C h^{-N_k}. \label{power resolvent estimates fractional schrodinger}
\end{align}
\item[2.] Consider $\R^d, d\geq 3$ equipped with a smooth metric $g$ satisfying $(\ref{assump elliptic}), (\ref{assump long range})$. Then for $k \geq 0$, there exists $C>0$ such that for all $\ep \in (0,1]$ and all $\mu \in I \Subset (0,+\infty)$,
\begin{align}
\|\scal{\ep x}^{-1-k} ((\ep^{-1}\Lambda_g)^\sigma- \mu \mp i0)^{-1-k} \scal{\ep x}^{-1-k}\|_{\Lc(L^2)} \leq C. \label{power resolvent estimates fractional schrodinger low freq}
\end{align}
\end{itemize}
\end{prop} 
\begin{proof}
We only give the proof for $(\ref{power resolvent estimates fractional schrodinger})$, the one for $(\ref{power resolvent estimates fractional schrodinger low freq})$ is similar using $(\ref{L2 bound low freq})$. We firstly note that the estimates $(\ref{power resolvent estimates fractional schrodinger})$ are equivalent to 
\[
\|\scal{x}^{-1-k} ((h\Lambda_g)^\sigma - \mu \mp i0)^{-1-k} \phi(h^2P) \scal{x}^{-1-k}\|_{\Lc(L^2)} \leq C h^{-N_k},
\]
where $\phi \in C^\infty_0((0,+\infty))$ satisfying $\phi=1$ near $I$. Note that here $\Lambda_g=\sqrt{P}$. Next, we write $\mu= \lambda^{\sigma/2}$ with $\lambda$ lied in a relatively compact interval of $(0,+\infty)$. By functional calculus, we write
\[
(h\Lambda_g)^\sigma - \mu \mp i0 = (h^2P -\lambda \mp i0) Q(h^2P, \mu),
\]
where $Q(\cdot,\mu)$ is smooth and non vanishing on the support of $\phi$. This implies for all $k \geq 0$,
\[
((h\Lambda_g)^\sigma - \mu \mp i0)^{-1-k} \phi(h^2P) = (h^2P -\lambda \mp i0)^{-1-k} \tilde{Q}(h^2P, \mu),
\]
where $\tilde{Q}(h^2P,\mu)= \phi(h^2P) Q^{-1-k}(h^2P,\mu)$. This allows us to approximate $\tilde{Q}(h^2P,\mu)$ by pseudo-differential operators by means of Proposition $\ref{prop parametrix phi}$. Thus, we have that $\scal{x}^{1+k} \tilde{Q}(h^2P,\mu) \scal{x}^{-1-k}$ is of size $O_{\Lc(L^2)}(1)$ uniformly in $\mu \in I \Subset(0,+\infty)$ and $h\in (0,1]$. Therefore, $(\ref{power resolvent estimates fractional schrodinger})$ follows from $(\ref{power resolvent estimates})$. The proof is complete.
\end{proof}
We now give an application of resolvent estimates given in Proposition $\ref{prop power resolvent estimates fractional schrodinger}$ when $k=0$ and obtain the following global $L^2$ integrability estimates for the fractional Schr\"odinger operators both at high and low frequencies.
\begin{prop} \label{prop global L2 integrability fractional schrodinger}
Let $\sigma \in (0,\infty)$ and $f \in C^\infty_0((0,+\infty))$.
\begin{itemize}
\item[1.] Consider $\R^d, d\geq 2$ equipped with a smooth metric $g$ satisfying $(\ref{assump elliptic}), (\ref{assump long range})$ and suppose that the assumption $(\ref{assump resolvent})$ holds. Then there exists $C>0$ such that for all $u_0 \in L^2$ and all $h \in (0,1]$,
\begin{align}
\|\scal{x}^{-1} f(h^2P) e^{-ith^{-1}(h\Lambda_g)^\sigma} u_0\|_{L^2(\R,L^2)} \leq C h^{{(1-N_0)}/2} \|u_0\|_{L^2}. \label{L2 global integrability}
\end{align}
\item[2.] Consider $\R^d, d\geq 3$ equipped with a smooth metric $g$ satisfying $(\ref{assump elliptic}), (\ref{assump long range})$. Then there exists $C>0$ such that for all $u_0\in L^2$ and all $\ep \in (0,1]$,
\begin{align}
\|\scal{\ep x}^{-1} f(\ep^{-2}P) e^{-it\ep(\ep^{-1}\Lambda_g)^\sigma} u_0 \|_{L^2(\R,L^2)} \leq C \ep^{-1/2}\|u_0\|_{L^2}. \label{L2 global integrability low freq}
\end{align}
\end{itemize}
\end{prop}
\begin{rem} \label{rem global Lp integrability fractional schrodinger}
\begin{itemize}
\item[1.] By interpolating between $L^2(\R)$ and $L^\infty(\R)$, we get the following $L^p$ integrability estimates
\begin{align}
\|\scal{x}^{-1} f(h^2P) e^{-ith^{-1}(h\Lambda_g)^\sigma} u_0\|_{L^p(\R,L^2)} &\leq C h^{(1-N_0)/p} \|u_0\|_{L^2}. \label{Lp global integrability} \\
\|\scal{\ep x}^{-1} f(\ep^{-2}P) e^{-it\ep(\ep^{-1}\Lambda_g)^\sigma} u_0\|_{L^p(\R,L^2)} &\leq C \ep^{-1/p}\|u_0\|_{L^2}. \label{Lp global integrability low freq}
\end{align} 
\item[2.] Thanks to the fact that $P$ is non-negative, these estimates are still true for $f \in C^\infty_0(\R \backslash \{0\})$. Moreover, we can replace $\|u_0\|_{L^2}$ in the right hand side of $(\ref{L2 global integrability})$ and $(\ref{Lp global integrability})$ (resp. $(\ref{L2 global integrability low freq})$ and $(\ref{Lp global integrability low freq})$) by $\|f(h^2P)u_0\|_{L^2}$ (resp. $\|f(\ep^{-2}P)u_0\|_{L^2}$). Indeed, we choose $\tilde{f} \in C^\infty_0(\R \backslash 0)$ such that $\tilde{f}=1$ near $\text{supp}(f)$ and write $f(h^2P)=\tilde{f}(h^2P)f(h^2P)$. We apply $(\ref{L2 global integrability})$ and $(\ref{Lp global integrability})$ with $\tilde{f}$ instead of $f$. Similarly for the low frequency case.
\end{itemize}
\end{rem}
\noindent \textit{Proof of \emph{Proposition} $\ref{prop global L2 integrability fractional schrodinger}$.}
We again only consider the high frequency case, the low frequency one is completely similar. By the limiting absorption principle (see \cite[Theorem XIII.25]{ReebSimon}), we see that $\|\scal{x}^{-1}  f(h^2P) e^{-it (h\Lambda_g)^\sigma} u_0\|^2_{L^2(\R,L^2)}$ is bounded by
\[
2 \pi \sup_{\mu \in \R \atop \ep>0} \|\scal{x}^{-1} f(h^2P) ((h\Lambda_g)^\sigma -\mu -i\ep)^{-1} f(h^2P) \scal{x}^{-1}\|_{\Lc(L^2)} \|u_0\|^2_{L^2}.
\]
By functional calculus and the holomorphy of the resolvent, it suffices to bound $\|\scal{x}^{-1} f(h^2P)((h\Lambda_g)^\sigma-\mu -i 0)^{-1} f(h^2P)\scal{x}^{-1}\|_{\Lc(L^2)}$, uniformly with respect to $\mu \in \R$. As a function of $h\Lambda_g$, the operator $f(h^2P) ((h\Lambda_g)^\sigma -\mu -i0)^{-1} f(h^2P)$ reads $f(\lambda^2) (\lambda^\sigma -\mu -i0)^{-1} f(\lambda^2)$. Assume that $\text{supp}(f) \subset \left[1/c^2,c^2\right]$ for some $c>1$, so $ \lambda \in \left[ 1/c, c \right]$. \newline
\indent In the case $\mu \geq 2c^\sigma$ or $\mu \leq 1/2c^\sigma$, we have that $\mu -\lambda^\sigma \geq c^\sigma$ or $\lambda^\sigma -\mu \geq 1/2c^\sigma$. The functional calculus gives
\[
\|f(h^2P)((h\Lambda_g)^\sigma -\mu -i0)^{-1} f(h^2P)\|_{\Lc(L^2)} \leq 2c^\sigma\|f\|^2_{L^\infty(\R)}.
\]
\indent Thus we can assume that $\mu \in [1/2c^\sigma, 2c^\sigma]$. Using $(\ref{power resolvent estimates fractional schrodinger})$ with $k=0$, we have 
\[
\|\scal{x}^{-1} ((h\Lambda_g)^\sigma -\mu \mp i0 )^{-1} \scal{x}^{-1} \|_{\Lc(L^2)} \leq  C h^{-N_0}.
\]
On the other hand, $\scal{x}^{-1}f(h^2P) \scal{x}$ is bounded in $\Lc(L^2)$ by pseudo-differential calculus. This implies
\[
\|\scal{x}^{-1}  f(h^2P) e^{-it(h\Lambda_g)^\sigma} u_0\|_{L^2(\R,L^2)} \leq C h^{-N_0/2} \|u_0\|_{L^2}.
\]
By scaling in time, this gives the result. 
\defendproof \newline
\indent Another application of the resolvent estimates given in Proposition $\ref{prop power resolvent estimates fractional schrodinger}$ is the following local energy decays for the fractional Schr\"odinger operators both at high and low frequencies.
\begin{prop} \label{prop local energy decay fractional schrodinger}
Let $\sigma \in (0,\infty)$ and $f \in C^\infty_0(\R \backslash \{0\})$.
\begin{itemize}
\item[1.] Consider $\R^d, d\geq 2$ equipped with a smooth metric $g$ satisfying $(\ref{assump elliptic}), (\ref{assump long range})$ and suppose that the assumption $(\ref{assump resolvent})$ holds. Then for $k \geq 0$, there exists $C>0$ such that for all $t\in \R$ and all $h\in (0,1]$,
\begin{align} 
\|\scal{x}^{-1-k} e^{-ith^{-1}(h\Lambda_g)^\sigma} f(h^2P) \scal{x}^{-1-k}\|_{\Lc(L^2)} \leq C h^{-N_k} \scal{th^{-1}}^{-k}. \label{local energy decay fractional schrodinger}
\end{align}
\item[2.] Consider $\R^d, d\geq 3$ equipped with a smooth metric $g$ satisfying $(\ref{assump elliptic}), (\ref{assump long range})$. Then for $k \geq 0$, there exists $C>0$ such that for all $t\in \R$ and all $\ep \in (0,1]$,
\begin{align} 
\|\scal{\ep x}^{-1-k} e^{-it\ep (\ep^{-1}\Lambda_g)^\sigma} f(\ep^{-2}P) \scal{\ep x}^{-1-k}\|_{\Lc(L^2)} \leq C  \scal{\ep t}^{-k}. \label{local energy decay fractional schrodinger low freq} 
\end{align}
\end{itemize}
\end{prop}
\begin{proof}
As above, we only give the proof for the high frequency case. Using the Stone formula, the operator $e^{-it(h\Lambda_g)^\sigma}f(h^2P)$ reads
\[
\frac{1}{2i \pi}\int_{\R} e^{-it\mu} f(\mu^{2/\sigma}) (((h\Lambda_g)^\sigma-\mu-i0)^{-1} - ((h\Lambda_g)^\sigma-\mu+i0)^{-1}) d\mu.
\]
We use the same trick as in \cite{BTglobalstrichartz}. By multiplying to above equality with $(it)^k$ and using integration by parts in the weighted spaces $\scal{x}^{-1-k}L^2$, we see that $(it)^k  e^{-it(h\Lambda_g)^\sigma}f(h^2P)$ is a linear combination with $l+n=k$ of terms of the form
\[
\int_{\R}  e^{-it\mu}  \partial^l_\mu ( f(\mu^{2/\sigma})) ( ((h\Lambda_g)^\sigma-\mu-i0)^{-1-n} - ((h\Lambda_g)^\sigma-\mu+i0)^{-1-n} ) d\mu.
\]
The compact support of $f$ implies that $\mu$ is bounded from above and below. The resolvent estimates $(\ref{power resolvent estimates fractional schrodinger})$ then imply
\[
\|\scal{x}^{-1-k} e^{-it(h\Lambda_g)^\sigma} f(h^2P) \scal{x}^{-1-k}\|_{\Lc(L^2)} \leq C h^{-N_k} \scal{t}^{-k}.
\]
Here we use that $N_m$ is non-decreasing with respect to $m$. By scaling in time, we have $(\ref{local energy decay fractional schrodinger})$. The proof is complete.
\end{proof}
\section{Reduction of the problem} \label{section reduction of the problem}
\setcounter{equation}{0}
\subsection{The Littlewood-Paley theorems}\label{subsection littlewood paley}
In this subsection, we recall some Littlewood-Paley type estimates which are essentially given in \cite{BoucletMizutani}. Let us introduce $f(\lambda)=f_0(\lambda)-f_0(2\lambda)$, where $f_0$ given as in $(\ref{split low high freqs})$. We have
$f\in C^\infty_0(\R \backslash \{0\})$ and 
\[
\sum_{k=1}^{\infty} f(2^{-k}\lambda)= (1-f_0)(\lambda), \quad \sum_{k=0}^{\infty}f(2^k \lambda)= \mathds{1}_{\R \backslash \{0\}}(\lambda) f_0(\lambda), \quad \lambda\in \R.
\]
The Spectral Theorem implies that
\begin{align}
(1-f_0)(P)= \sum_{k=1}^{\infty}f(2^{-k}P), \quad f_0(P)=\sum_{k=0}^{\infty}f(2^kP). \label{spectral localization}
\end{align}
In the second sum, we use the fact that $0$ is not an eigenvalue of $P$ (see e.g. \cite{TataruKoch}).
\begin{theorem} \label{theorem littlewood paley estimates}
\begin{itemize}
\item[1.] Let $N \geq 1 $ and $\chi \in C^\infty_0(\R^d)$. Then for $q\in [2,\infty)$, there exists $C>0$ such that
\begin{align}
\|(1-\chi)(1-f_0)(P) v\|_{L^q} \leq  C \Big(\sum_{h^2=2^{-k} } \|(1-\chi) f(h^2P) v\|^2_{L^q} +h^N \|\scal{x}^{-N}f(h^2P) v\|^2_{L^2} \Big)^{1/2}, \label{littlewood paley 1-chi}
\end{align}
for all $v \in \Sch(\R^d)$, where $k\in \N \backslash \{0\}$. The same estimates hold for $\chi$ in place of $1-\chi$.
\item[2.] Let $\chi \in C^\infty_0(\R^d)$ be such that $\chi (x)=1$ for $|x| \leq 1$. Then for $q\in (2,\infty)$, there exists $C>0$ such that for all $v\in L^2$,
\begin{align}
\|f_0(P)v\|_{L^q} \leq C \Big(\sum_{\ep^{-2}=2^{k}} \|(1-\chi)(\ep x) f(\ep^{-2}P) v\|^2_{L^q}  + \|\ep^{d/2-d/q}\scal{\ep x}^{-1} f(\ep^{-2}P) v\|^2_{L^2} \Big)^{1/2}. \label{littlewood paley low freq}
\end{align}
Here we use in the sum that $k \in \N$.
\end{itemize}
\end{theorem}
Note that the Littlewood-Paley theorem at low frequency is slightly different from the one in \cite[Theorem 4.1]{BoucletMizutani}. In \cite{BoucletMizutani}, Bouclet-Mizutani considered the sharp Schr\"odinger admissible condition. This allows to interpolate between the trivial Strichartz estimate for $(\infty,2)$ and the endpoint Strichartz estimate for the endpoint pair $(2,2^\star)$. The proof of the low frequency Littlewood-Paley theorem given in \cite{BoucletMizutani} makes use of the homogeneous Sobolev embedding
\begin{align}
\|v\|_{L^{2^\star}} \leq C \|\Lambda_g v\|_{L^2}, \quad 2^\star = \frac{2d}{d-2}. \label{homogeneous sobolev embedding}
\end{align}
Since we consider a larger range of admissible condition $(\ref{fractional admissible})$, we can not apply this interpolation technique. To overcome this difficulty, we will take the advantage of heat kernel estimates. Our estimate $(\ref{littlewood paley low freq})$ is robust and can be applied for another types of dispersive equations such as the wave or Klein-Gordon equations. \newline
\indent Let $K(t,x,y)$ be the kernel of the heat operator $e^{-tP}, t>0$, i.e.
\[
e^{-tP}u(x)= \int_{\R^d} K(t,x,y) u(y) dy.
\]
We recall some properties (see e.g. \cite{Chavel}, \cite{Grigo99}) of the heat kernel on arbitrary Riemannian manifold.
\begin{lem} \label{lem heat kernel property}
Let $(M,g)$ be an arbitrary Riemannian manifold. Then the heat kernel $K$ satisfies the following properties:
\begin{itemize}
\item[(i)] $K$ is a strictly positive $C^\infty$ function on $(0,\infty)\times M \times M$.
\item[(ii)] $K$ is symmetric in the space components.
\item[(iii)] (Maximum principle)
\[
\int_M K(t,x,y) d_g(y) \leq 1.
\]
\item[(iv)] (Semi-group property)
\[
\int_M K(s,x,y) K(t,y,z) d_g(y)=K(s+t,x,z).
\]
\end{itemize}
\end{lem}
In order to obtain the heat kernel estimate, we will make use of the Nash inequality (see e.g. \cite[Theorem 3.2.1]{Saloff-Coste}), namely
\begin{align}
\|u\|_{L^2} \leq C\|u\|_{L^1}^{\frac{2}{d+2}} \|\nabla u\|_{L^2}^{\frac{d}{d+2}}. \label{nash inequality Rd}
\end{align}
Note that the Nash inequality on $\R^d$ is valid for any $d\geq 1$. Thanks to $(\ref{equivalent homogeneous sobolev norms})$, we have for $d\geq 2$,
\begin{align}
\|u\|_{L^2} \leq C\|u\|_{L^1}^{\frac{2}{d+2}} \|\Lambda_g u\|_{L^2}^{\frac{d}{d+2}}. \label{nash inequality manifold}
\end{align}
Using $(\ref{nash inequality manifold})$, we have the following upper bound for the heat kernel. 
\begin{theorem}
There exists $C>0$ such that for all $x, y \in \R^d$ and all $t>0$ such that 
\begin{align}
K(t,x,x) &\leq C t^{-d/2}, \label{upperbound on-diagonal} \\
K(t,x,y) &\leq C t^{-d/2} \exp \Big( - \frac{|x-y|^2}{Ct} \Big). \label{upperbound off-diagonal usual distance}
\end{align}
In particular, 
\begin{align}
\|e^{-tP}\|_{\Lc(L^1,L^\infty)} \leq Ct^{-d/2}, \quad t>0. \label{L1 L8 estimate heat operator}
\end{align}
\end{theorem}
\begin{proof}
The proof is similar to the one given in \cite[Theorem 6.1]{Grigo99} where the author shows how to get $(\ref{upperbound on-diagonal})$ from the homogeneous Sobolev embedding $(\ref{homogeneous sobolev embedding})$. For the reader's convenience, we give a sketch of the proof. Fix $x \in \R^d$ and denote $v(t,y) = K(t,y,x)$ and
\[
J(t):= \|v(t)\|^2_{L^2}. 
\]
Using the fact that $\partial_t v(t,y) = -Pv(t,y)$, we have
\[
J'(t)= 2 \scal{v(t), \partial_t v(t)} = -2 \scal{v(t),P v(t)} = - 2 \|\Lambda_gv(t)\|^2_{L^2}.
\]
This implies that $J(t)$ is non-increasing. On the other hand, the maximum principle (see also \cite{Grigo99}) shows that
\[
\|v(t)\|_{L^1}= \int_{\R^d} K(t,x,y) dy \leq 1.
\]
This together with $(\ref{nash inequality manifold})$ yield 
\[
\|v(t)\|^2_{L^2} \leq C\|v(t)\|^\frac{4}{d+2}_{L^1} \|\Lambda_g v(t)\|^{\frac{2d}{d+2}}_{L^2} \leq C\|\Lambda_g v(t)\|^{\frac{2d}{d+2}}_{L^2}.
\]
We thus get
\[
J'(t) \leq - C \|v(t)\|^{\frac{2(d+2)}{d}}_{L^2} = -C J(t)^{\frac{d+2}{d}}.
\]
Integrating between $0$ and $t$ with the fact that the non-increasing property of $J(t)$, we obtain
\[
J(t) \leq C t^{-d/2}.
\] 
The estimate $(\ref{upperbound on-diagonal})$ then follows by the symmetric property of $K(t,x,y)$, i.e. $J(t) = K(2t, x,x)$. Using $(\ref{upperbound on-diagonal})$, the off-diagonal argument (see also \cite{Grigo99}) implies the following upper bound for the heat kernel
\begin{align*}
K(t,x,y) \leq C t^{-d/2} \exp \Big( - \frac{d^2(x,y)}{Ct} \Big), \quad \forall x,y \in \R^d, t>0, 
\end{align*}
where $d(x,y)$ is the geodesic distance from $x$ to $y$. Thanks to the elliptic condition $(\ref{assump elliptic})$ of the metric $g$, it is easy to see that
\[
d(x,y)\sim |x-y|.
\]
This shows $(\ref{upperbound off-diagonal usual distance})$ and the proof is complete.
\end{proof}
We now give some applications of the upper bound $(\ref{upperbound off-diagonal usual distance})$. A first application is the following homogeneous Sobolev embedding.
\begin{lem} \label{lem homogeneous Sobolev embedding}
Let $q\in (2,\infty)$ and $\alpha=\frac{d}{2}-\frac{d}{q}$. Then the operator $\Lambda_g^{-\alpha}$ maps $L^2$ to $L^q$. In particular, there exists $C>0$ such that
\begin{align}
\|u\|_{L^q} \leq C \|\Lambda_g^\alpha u\|_{L^2}. \label{homogeneous sobolev embedding general}
\end{align}
\end{lem}
\begin{proof}
We firstly recall the following version of Hardy-Littlewood-Sobolev theorem.
\begin{theorem}[\cite{HardyLittlewood,Sobolev}] \label{theorem hardy littlewood sobolev}
Let $1<p<q<\infty$, $\gamma=d+\frac{d}{q}-\frac{d}{p}$ and $K_\gamma(x):=|x|^{-\gamma}$. Then the convolution operator $T_\gamma:=f * K_\gamma$ maps $L^p$ to $L^q$. In particular, there exists $C>0$ such that
\[
\|T_\gamma u\|_{L^q} \leq C\|u\|_{L^p}.
\]
\end{theorem}
\noindent Now let $\Ga(z):= \int_0^{\infty} t^{z-1} e^{-t} dt, \re{(z)}>0$ be the Gamma function. The spectral theory with the fact $\Lambda_g=\sqrt{P}$ gives 
\[
\Lambda_g^{-\alpha}= P^{-\alpha/2}=\frac{1}{\Ga(\alpha/2)} \int_{0}^{\infty} e^{-tP} t^{\alpha/2-1} dt.
\]
Let $[\Lambda_g^{-\alpha}](x,y)$ be the kernel of $\Lambda_g^{-\alpha}$. By $(\ref{upperbound off-diagonal usual distance})$, 
\[
|[\Lambda_g^{-\alpha}](x,y)| \leq \frac{C}{\Ga(\alpha/2)} \int_0^\infty t^{-d/2} e^{-\frac{|x-y|^2}{Ct}} t^{\alpha/2-1} dt.
\]
A change of variable shows
\begin{align*}
|[\Lambda_g^{-\alpha}](x,y)| \leq \frac{C}{\Ga(\alpha/2)} |x-y|^{-(d-\alpha)}\int_{0}^{\infty} t^{d/2-\alpha/2-1} e^{-t} dt 
= \frac{C\Ga(d/2-\alpha/2)}{\Ga(\alpha/2)} |x-y|^{-(d-\alpha)}. 
\end{align*}
The result follows by applying Theorem $\ref{theorem hardy littlewood sobolev}$ with $\gamma=d-\alpha$ and $p=2$. 
\end{proof}
Another application of the heat kernel upper bound $(\ref{upperbound off-diagonal usual distance})$ is the following $L^q-L^r$-bound of the heat operator.
\begin{lem} \label{lem Lq Lr estimate heat operator}
Let $1 \leq q \leq r \leq \infty$. The heat operator $e^{-tP}, t>0$ maps $L^q$ to $L^r$. In particular, there exists $C>0$ such that for all $t>0$,
\[
\|e^{-tP}\|_{\Lc(L^q, L^r)} \leq C t^{-\frac{d}{2}\left(\frac{1}{q}-\frac{1}{r}\right)}.
\]
\end{lem}
\begin{proof}
By the symmetric and maximal principle properties of the heat kernel, the Schur's Lemma yields
\begin{align}
\|e^{-tP}\|_{\Lc(L^q)} \leq C, \quad t>0. \label{Lq estimate heat operator}
\end{align} 
Interpolating between $(\ref{L1 L8 estimate heat operator})$ and $(\ref{Lq estimate heat operator})$, we have the result.
\end{proof}
\begin{coro} \label{coro f ep P}
Let $f\in C^\infty_0(\R \backslash \{0\})$ and $q \in [2,\infty]$. Then there exists $C>0$ such that for all $\ep \in (0,1]$,
\[
\|f(\ep^{-2}P)\|_{\Lc(L^2,L^q)} \leq C \ep^{d/2-d/q}.
\]
\end{coro}
\begin{proof}
By writing
\[
f(\ep^{-2}P)= e^{-\ep^{-2}P} (e^{\ep^{-2}P} f(\ep^{-2}P)),
\]
and using Proposition $\ref{lem Lq Lr estimate heat operator}$ with $t= \ep^{-2}$, we get
\[
\|f(\ep^{-2}P)\|_{\Lc(L^2,L^q)} \leq  \|e^{-\ep^{-2}P}\|_{\Lc(L^2,L^q)} \|e^{\ep^{-2}P} f(\ep^{-2}P)\|_{\Lc(L^2)} \leq C \ep^{d/2-d/q}.
\]
Here, using the compactly supported property of $f$ and spectral theorem, we have  $e^{\ep^{-2}P} f(\ep^{-2}P)$ is of size $O_{\Lc(L^2)}(1)$. This gives the result.
\end{proof}
We now are able to prove Theorem $\ref{theorem littlewood paley estimates}$. We only give the proof for the low frequency case. The high frequency one is essentially given in \cite[Theorem 4.6]{BoucletMizutani}. \newline
\noindent \textit{Proof of \emph{Theorem} $\ref{theorem littlewood paley estimates}$.} 
By the second term of $(\ref{spectral localization})$, we have
\begin{align}
\|f_0(P)v\|_{L^q} = \sup_{\|w\|_{L^{q'}}=1} |(w, f_0(P)v)|= \sup_{\|w\|_{L^{q'}}=1} \Big|\lim_{M\rightarrow \infty} \sum_{k=0}^{M}(w,f(\ep^{-2}P)v)\Big|, \label{define norm Lq low freq}
\end{align}
where $\ep^{-2}=2^k$ and $(\cdot,\cdot)$ is the inner product on $L^2$. By choosing $\tilde{f}\in C^\infty_0(\R \backslash \{0\})$ satisfying $\tilde{f}=1$ near $\text{supp}(f)$, we use Proposition $\ref{prop parametrix phi low freq}$ to write $(1-\chi)(\ep x) \tilde{f}(\ep^{-2}P)=Q(\ep)+R(\ep)$, where
\[
Q(\ep)=(1-\chi)(\ep x) Op_\ep(\tilde{f}\circ p_\ep) \zeta(\ep x), \quad R(\ep)=\zeta(\ep x) (\ep^{-2}P+1)^{-m} B(\ep) \scal{\ep x}^{-1},
\]
with $\zeta \in C^\infty(\R^d)$ supported outside $B(0,1)$ and equal to 1 near $\text{supp}(1-\chi)$ and $B(\ep)=O_{\Lc(L^2)}(1)$ uniformly in $\ep \in (0,1]$. We next write
\[
f(\ep^{-2}P)= Q(\ep) (1-\chi)(\ep x) f(\ep^{-2}P) + A(\ep) \ep^{\alpha}\scal{\ep x}^{-1}f(\ep^{-2}P),
\]
with $\alpha=d/2-d/q$ and
\[
A(\ep)=\ep^{-\alpha} \Big((1-\chi)(\ep x) \tilde{f}(\ep^{-2}P) \chi(\ep x) + R(\ep) (1-\chi)(\ep x) + \chi(\ep x) \tilde{f}(\ep^{-2}P)\Big) \scal{\ep x}.
\]
We now bound
\begin{align}
\Big|\sum_{k=0}^{M}(w, f(\ep^{-2}&P)v)\Big| \lesssim \Big|\sum_{k=0}^{M} \left(w, Q(\ep)(1-\chi)(\ep x)f(\ep^{-2}P) v\right) \Big|  + \Big|\sum_{k=0}^{M} (w, A(\ep) \ep^{\alpha}\scal{\ep x}^{-1}f(\ep^{-2}P) v) \Big| \nonumber \\
&\lesssim \Big|\sum_{k=0}^{M} \left(Q^\star(\ep)w, (1-\chi)(\ep x)f(\ep^{-2}P) v\right) \Big|  + \|w\|_{L^{q'}} \Big\| \sum_{k=0}^{M} A(\ep)\ep^{\alpha}\scal{\ep x}^{-1}f(\ep^{-2}P) v \Big\|_{L^q} \nonumber \\
&=: \text{(I)} +\text{(II)}. \label{littlewood paley proof 1}
\end{align}
We use the Cauchy-Schwarz inequality in $k$ and the H\"older inequality in space to have
\[
\text{(I)}\lesssim \|\widetilde{S}_M w\|_{L^{q'}} \|S_M v\|_{L^q},
\]
where 
\[
\widetilde{S}_M w:= \Big(\sum_{k=0}^{M}|Q^\star(\ep) w|^2\Big)^{1/2}, \quad S_Mv:=\Big(\sum_{k=0}^{M}|(1-\chi)(\ep x) f(\ep^{-2}P)v|^2\Big)^{1/2}.
\]
We now make use of the following estimate (see \cite[Proposition 4.3]{BoucletMizutani}).
\begin{prop} \label{prop calderon zygmund low freq}
For $r\in (1,2]$, there exists $C>0$ such that for all $M \geq 0$ and all $w \in \Sch(\R^d)$,
\[
\| \widetilde{S}_Mw\|_{L^r} \leq C \|w\|_{L^r}.
\]
\end{prop}
\noindent We thus get
\begin{align}
\text{(I)} \lesssim \|S_M v\|_{L^q} \|w\|_{L^{q'}} \lesssim \Big(\sum_{k=0}^M \|(1-\chi)(\ep x) f(\ep^{-2}P) v\|^2_{L^q}\Big)^{1/2} \|w\|_{L^{q'}}. \label{littlewood paley proof 2}
\end{align}
For the second term in $(\ref{littlewood paley proof 1})$, we use the homogeneous Sobolev embedding $(\ref{homogeneous sobolev embedding general})$ to have
\[
\Big\| \sum_{k=0}^{M} A(\ep) \ep^{\alpha}\scal{\ep x}^{-1}f(\ep^{-2}P) v \Big\|_{L^q} \lesssim  \Big\| \sum_{k=0}^{M} \Lambda_g^\alpha A(\ep) \ep^\alpha \scal{\ep x}^{-1}f(\ep^{-2}P) v \Big\|_{L^2}.
\]
We next write
\begin{align}
\Lambda_g^{\alpha} A(\ep)= (\ep^{-2}P)^{\alpha/2} (\ep^{-2}P+1)^{-\alpha} D(\ep), \label{littlewood paley proof 3}
\end{align}
with $D(\ep)=O_{\Lc(L^2)}(1)$ uniformly in $\ep \in (0,1]$. It is easy to have $(\ref{littlewood paley proof 3})$ from the first two terms in $A(\ep)$ by using Proposition $\ref{prop parametrix phi low freq}$. The less obvious contribution in $(\ref{littlewood paley proof 2})$ is the uniform $L^2$ boundedness of $(\ep^{-2}P+1)^\alpha \chi(\ep x) \tilde{f}(\ep^{-2}P) \scal{\ep x}$. By the functional calculus, it is enough to show for $N$ large enough the uniform $L^2$ boundedness of 
$(\ep^{-2}P+1)^N \chi(\ep x) \tilde{f}(\ep^{-2}P) \scal{\ep x}$. To see it, we write
\[
(\ep^{-2}P+1)^N \chi(\ep x) \tilde{f}(\ep^{-2}P) \scal{\ep x} = \chi(\ep x) (\ep^{-2}P+1)^N \tilde{f}(\ep^{-2}P) \scal{\ep x} + [(\ep^{-2}P+1)^N, \chi(\ep x)] \tilde{f}(\ep^{-2}P) \scal{\ep x},
\]
where $[\cdot, \cdot]$ is the commutator. The $L^2$ boundedness of $\chi(\ep x) (\ep^{-2}P+1)^N \tilde{f}(\ep^{-2}P) \scal{\ep x}$ follows as in $(\ref{L2 bound low freq})$. On the other hand, note that the commutator $[(\ep^{-2}P+1)^N, \chi(\ep x)]$ can be written as a sum of rescaled pseudo-differential operators vanishing outside the support of $\zeta(\ep x)$ for some $\zeta \in C^\infty(\R^d)$ supported outside $B(0,1)$ and equal to $1$ near infinity. This allows to use Proposition $\ref{prop parametrix phi low freq}$, and the $L^2$ boundedness of $[(\ep^{-2}P+1)^N, \chi(\ep x)] \tilde{f}(\ep^{-2}P) \scal{\ep x}$ follows. We next need to recall the following well-known discrete Schur estimate.
\begin{lem} \label{lem discrete schur}
Let $\theta>0$ and $(T_l)_{l}$ be a sequence of linear operators on a Hilbert space $\Hc$. If $\|T^\star_l T_k\|_{\Lc(\Hc)} \lesssim 2^{-\theta|k-l|}$, then there exits $C>0$ such that for all sequence $(v_k)_k$ of $\Hc$,
\[
\|\sum T_k v_k\|_{\Hc} \leq C\left(\sum \|v_k\|^2_{\Hc}\right)^{1/2}.
\]
\end{lem}
\noindent Now let $T_k=(\ep_k^{-2}P)^{\alpha/2} (\ep_k^{-2}P+1)^{-\alpha} D(\ep_k)$ with $\ep_k^{-2}=2^k$. We see that 
\[
T_l^\star T_k = 2^{\frac{\alpha(l+k)}{2}} D^\star(\ep_l)(2^lP+1)^{-\alpha} P^\alpha (2^kP+1)^{-\alpha} D(\ep_k).
\]
Note that $l+k= -|k-l| + 2k$ for $k\geq l$ and $l+k=-|k-l|+2l$ for $l\geq k$. Thus for $k\geq l$,
\[
\|T_l^\star T_k\|_{\Lc(L^2)} =2^{-\frac{\alpha|k-l|}{2}} \|D^\star(\ep_l) (2^lP+1)^{-\alpha} (2^kP)^\alpha (2^kP+1)^{-\alpha} D(\ep_k)\|_{\Lc(L^2)}\lesssim 2^{-\frac{\alpha|k-l|}{2}}.
\] 
Similarly for $l\geq k$. Therefore, we can apply Lemma $\ref{lem discrete schur}$ for  $T_k=(\ep_k^{-2}P)^{\alpha/2} (\ep_k^{-2}P+1)^{-\alpha} D(\ep_k)$ with $\ep_k^{-2}=2^k, \Hc=L^2$ and $\theta=\alpha/2$ to get
\begin{align}
\sup_{M} \Big\| \sum_{k=0}^{M} \Lambda_g^\alpha A(\ep) \ep^\alpha \scal{\ep x}^{-1} f(\ep^{-2}P) v\Big\|_{L^2} \lesssim \Big(\sum_{k\geq 0} \|\ep^\alpha \scal{\ep x}^{-1} f(\ep^{-2}P) v\|^2_{L^2}\Big)^{1/2}. \label{littlewood paley proof 4}
\end{align}
Collecting $(\ref{define norm Lq low freq}), (\ref{littlewood paley proof 1}), (\ref{littlewood paley proof 2})$ and $(\ref{littlewood paley proof 4})$, we prove $(\ref{littlewood paley low freq})$. The proof of Theorem $\ref{theorem littlewood paley estimates}$ is now complete.
\defendproof
\subsection{Reduction of the high frequency problem}
Let us now consider the high frequency case. For a given $\chi \in C^\infty_0(\R^d)$, we write $u_{\text{high}}=\chi u_{\text{high}} + (1-\chi)u_{\text{high}}$. Using $(\ref{littlewood paley 1-chi})$ and Minkowski inequality with $p,q \geq 2$, we have 
\begin{multline}
\|(1-\chi)u_{\text{high}}\|_{L^p(\R,L^q)} \leq C \Big(\sum_{h^2=2^{-k}} \|(1-\chi) f(h^2P) e^{-it\Lambda_g^\sigma} u_0\|^2_{L^p(\R,L^q)}  \Big.  \\
\Big. +h^N \|\scal{x}^{-N}f(h^2P) e^{-it\Lambda_g^\sigma}u_0\|^2_{L^p(\R,L^2)} \Big)^{1/2}. \label{littlewood paley minkowski 1-chi}
\end{multline}
The same estimate holds for $\|\chi u_{\text{high}}\|_{L^p(\R,L^q)}$ with $\chi$ in place of $1-\chi$. We can apply the Item 2 of Remark $\ref{rem global Lp integrability fractional schrodinger}$ with scaling in time for the second term in the right hand side of above quantity to get
\begin{align}
h^{N/2}\|\scal{x}^{-N}f(h^2P) e^{-it\Lambda_g^\sigma}u_0\|_{L^p(\R,L^2)} \leq C h^{N/2+(\sigma-N_0)/p}\|f(h^2P) u_0\|_{L^2}. \label{remainder littlewood paley minkowski}
\end{align}
By taking $N$ large enough, this term is bounded by $h^{-\gamma_{p,q}}\|f(h^2P)u_0\|_{L^2}$.  
Thus we have the following reduction.
\begin{prop} \label{prop reduction of high freq problem}
\begin{itemize}
\item[1.] Consider $\R^d, d\geq 2$ equipped with a smooth metric $g$ satisfying $(\ref{assump elliptic})$, $(\ref{assump long range})$ and suppose that the geodesic flow associated to $g$ is non-trapping. If for all $\chi \in C^\infty_0(\R^d)$ and all $(p,q)$ fractional admissible, there exists $C>0$ such that for all $u_0 \in \Lch_g$ and all $h \in (0,1]$,
\begin{align}
\|\chi e^{-it\Lambda_g^\sigma} f(h^2P) u_0\|_{L^p(\R,L^q)} \leq C h^{-\gamma_{p,q}}\|f(h^2P)u_0\|_{L^2}, \label{reduction chi u fractional schrodinger} 
\end{align}
then
\begin{align}
\|\chi u_{\emph{high}}\|_{L^p(\R,L^q)} \leq C \|u_0\|_{\dot{H}^{\gamma_{p,q}}_g},\label{strichartz chi high freq}
\end{align}
i.e. $\emph{Theorem}$ $\ref{theorem strichartz inside compact}$ holds true.
\item[2.] Consider $\R^d, d\geq 2$ equipped with a smooth metric $g$ satisfying $(\ref{assump elliptic}), (\ref{assump long range})$ and suppose that $(\ref{assump resolvent})$ is satisfied. If there exists $R>0$ large enough such that for all $(p,q)$ fractional admissible and all $\chi\in C^\infty_0(\R^d)$ satisfying $\chi=1$ for $|x| <R$, there exists $C>0$ such that for all $u_0 \in \Lch_g$ and all $h \in (0,1]$,
\begin{align}
\|(1-\chi)e^{-it\Lambda_g^\sigma} f(h^2P)u_0\|_{L^p(\R,L^q)} \leq C h^{-\gamma_{p,q}}\|f(h^2P)u_0\|_{L^2}, \label{reduction 1-chi u fractional schrodinger} 
\end{align}
then
\begin{align}
\|(1-\chi)u_{\emph{high}}\|_{L^p(\R,L^q)} \leq C \|u_0\|_{\dot{H}^{\gamma_{p,q}}_g}, \label{strichartz 1-chi high freq}
\end{align}
i.e. $\emph{Theorem}$ $\ref{theorem strichartz outside compact}$ holds true. 
\end{itemize}
Moreover, combining $(\ref{strichartz chi high freq})$ and $(\ref{strichartz 1-chi high freq})$, we have
\[
\|u_{\emph{high}}\|_{L^p(\R,L^q)} \leq C \|u_0\|_{\dot{H}^{\gamma_{p,q}}_g}.
\]
\end{prop}
\begin{proof}
We only consider the case $1-\chi$, for $\chi$ it is similar. By using $(\ref{remainder littlewood paley minkowski})$ and $(\ref{reduction 1-chi u fractional schrodinger})$, we see that $(\ref{littlewood paley minkowski 1-chi})$ implies
\[
\|(1-\chi) u_{\text{high}}\|_{L^p(\R,L^q)} \leq C \Big( \sum_{h^2=2^{-k}} h^{-2\gamma_{p,q}} \|f(h^2P)u_0\|^2_{L^2} \Big)^{1/2} \leq C \|u_0\|_{\dot{H}^{\gamma_{p,q}}_g}.
\]
Here we use the almost orthogonality and the support property of $f$ to obtain the last inequality. This proves $(\ref{strichartz 1-chi high freq})$. 
\end{proof}

\subsection{Reduction of the low frequency problem}
Let us consider the low frequency case. We only treat the case $q\in (2,\infty)$ since the Strichartz estimate for $(p,q)=(\infty,2)$ is trivial. We apply the Littlewood-Paley estimates $(\ref{littlewood paley low freq})$ and Minkowski inequality with $p \geq 2$ to have 
\begin{multline}
\|u_{\text{low}}\|_{L^p(\R, L^q)} \leq C \Big(\sum_{\ep^{-2}=2^{k}} \|(1-\chi)(\ep x) f(\ep^{-2}P)e^{-it\Lambda_g^\sigma} u_0\|^2_{L^p(\R,L^q)} \Big. \\
\Big.+ \|\ep^{d/2-d/q}\scal{\ep x}^{-1} f(\ep^{-2}P) e^{-it\Lambda_g^\sigma} u_0\|^2_{L^p(\R,L^2)} \Big)^{1/2}. \nonumber
\end{multline}
We use global $L^p$ integrability estimates $(\ref{Lp global integrability low freq})$ with rescaling in time to bound the second term in the right hand side as
\begin{align}
\|\ep^{d/2-d/q}\scal{\ep x}^{-1} f(\ep^{-2}P) e^{-it\Lambda_g^\sigma} u_0\|_{L^p(\R,L^2)} \leq C \ep^{\gamma_{p,q}} \|f(\ep^{-2}P) u_0\|_{L^2}. \label{remainder littlewood paley low freq application}
\end{align}
Here we recall that $\gamma_{p,q}=d/2-d/q-\sigma/p$. This leads to the following reduction.
\begin{prop}
Consider $\R^d, d\geq 3$ equipped with a smooth metric $g$ satisfying $(\ref{assump elliptic}), (\ref{assump long range})$. If for all $\chi \in C^\infty_0(\R^d)$ satisfying $\chi(x)=1$ for $|x| \leq 1$ and all $(p,q)$ fractional admissible, there exists $C>0$ such that for all $u_0 \in \Lch_g$ and all $\ep \in (0,1]$,
\begin{align}
\|(1-\chi)(\ep x) f(\ep^{-2}P) e^{-it\Lambda_g^\sigma} u_0\|_{L^p(\R, L^q)} &\leq C \ep^{\gamma_{p,q}}\|f(\ep^{-2}P)u_0\|_{L^2}, \label{reduce strichartz outside compact low freq} 
\end{align}
then 
\[
\|u_{\emph{low}}\|_{L^p(\R,L^q)} \leq C \|u_0\|_{\dot{H}^{\gamma_{p,q}}_g}.
\]
\end{prop}
\begin{proof}
Indeed, if the estimates $(\ref{reduce strichartz outside compact low freq})$ hold true, then the Littlewood-Paley estimates $(\ref{littlewood paley low freq})$ and $(\ref{remainder littlewood paley low freq application})$ give
\begin{align}
\|u_{\text{low}}\|_{L^p(\R,L^q)} \leq C \Big( \sum_{\ep^{-2}=2^k} \ep^{2\gamma_{p,q}}\|f(\ep^{-2}P) u_0\|^2_{L^2} \Big)^{1/2}. \nonumber 
\end{align}
Note that
\[
\ep^{\gamma_{p,q}}\|f(\ep^{-2}P)u_0\|_{L^2} \leq \ep^{\gamma_{p,q}}\|\tilde{f}(\ep^{-2}P) \Lambda_g^{-\gamma_{p,q}}\|_{\Lc(L^2)} \|f(\ep^{-2}P) \Lambda_g^{\gamma_{p,q}} u_0\|_{L^2},
\]
where $\tilde{f} \in C^\infty_0(\R \backslash \{0\})$ satisfies $\tilde{f}=1$ near $\text{supp}(f)$. By functional calculus, the first factor in the right hand side is bounded by
\[
\ep^{\gamma_{p,q}} \sup_{\lambda \in \R} \left| \frac{\tilde{f}(\ep^{-2} \lambda^2)}{\lambda^{\gamma_{p,q}}}\right| \leq \ep^{\gamma_{p,q}}\frac{\|\tilde{f}\|_{L^\infty(\R)} }{(\ep/c)^{\gamma_{p,q}}} \leq c^{\gamma_{p,q}}\|\tilde{f}\|_{L^\infty(\R)}.
\]
Here $\ep^{-2}\lambda^2 \in \text{supp}(\tilde{f})$ hence $|\lambda| \in [\ep/c,\ep c]$ for some constant $c>1$. Thus we have
\[
\|u_{\text{low}}\|_{L^p(\R,L^q)} \leq C \Big( \sum_{\ep^{-2}=2^k} \|f(\ep^{-2}P) \Lambda_g^{\gamma_{p,q}} u_0\|^2_{L^2} \Big)^{1/2} \leq C \|u_0\|_{\dot{H}^{\gamma_{p,q}}_g},
\]
the last inequality follows from the almost orthogonality. This completes the proof.
\end{proof}
\section{Strichartz estimates inside compact sets} \label{section strichartz estimates inside compact}
\setcounter{equation}{0}
In this section, we will give the proof of $(\ref{reduction chi u fractional schrodinger})$. Our main tools are the local in time Strichartz estimates which are proved by the WKB method (see \cite{Dinhcompact}) and the $L^2$ integrability estimate at high frequency given in Proposition $\ref{prop global L2 integrability fractional schrodinger}$.
\subsection{The WKB approximations}
Let us start with the following result which is given in \cite[Theorem 2.7]{Dinhcompact}. 
\begin{theorem} \label{theorem wkb for fractional schrodinger}
Let $\sigma \in (0,\infty) \backslash \{1\}$ and $q$ be a smooth function on $\R^{2d}$ compactly support in $\xi$ away from zero and satisfying for all $\alpha, \beta \in \N^d$, there exists $C_{\alpha\beta}>0$ such that for all $x, \xi \in \R^d$,
\[
|\partial^\alpha_x \partial^\beta_\xi q(x,\xi)| \leq C_{\alpha\beta}.
\] 
Then there exist $t_0>0$ small enough, a function $S \in C^\infty([-t_0,t_0]\times \R^{2d})$ and a sequence of smooth functions $a_{j}(t,x, \xi)$ compactly supported in $\xi$ away from zero uniformly in $t\in [-t_0,t_0]$ such that for all $N \geq 1$,
\[
e^{-ith^{-1}(h\Lambda_g)^\sigma} Op^h(q) u_0 = J_N(t) u_0 + R_N(t)u_0,
\]
where 
\[
J_N(t)u_0(x)= (2\pi h)^{-d} \iint_{\R^{2d}} e^{ih^{-1}(S(t,x,\xi)-y\xi)} \sum_{j=0}^{N-1}h^j a_{j}(t,x,\xi) u_0(y)dyd\xi,
\]
$J_N(0)=Op^h(q)$ and the remainder $R_N(t)$ satisfies for all $t\in [-t_0,t_0]$ and all $h\in (0,1]$,
\begin{align}
\|R_N(t)\|_{\Lc(L^2)} \leq C h^{N-1}. \nonumber
\end{align}
Moreover, there exists a constant $C>0$ such that for all $t\in [-t_0,t_0]$ and all $h \in (0,1]$,
\begin{align}
\|J_N(t)\|_{\Lc(L^1, L^\infty)}  \leq  Ch^{-d}(1+|t|h^{-1})^{-d/2}. \nonumber
\end{align}
\end{theorem}
In \cite{Dinhcompact}, we consider the smooth bounded metric, i.e. for all $\alpha \in \N^d$, there exists $C_\alpha>0$ such that for all $x \in \R^d$,
\begin{align}
| \partial^\alpha g^{jk}(x) | \leq C_\alpha, \quad j,k \in \{1,...d\}. \label{assump bounded metric}
\end{align}
It is obvious to see that the assumption $(\ref{assump long range})$ implies $(\ref{assump bounded metric})$. This theorem and the parametrix given in Proposition $\ref{prop parametrix phi}$ give the following dispersive estimates for the fractional Schr\"odinger equations (see also \cite[Remark 2.8]{Dinhcompact}).
\begin{prop} \label{prop dispersive estimates fractional schrodinger}
Let $\sigma\in (0,\infty)\backslash\{1\}$ and $\varphi \in C^\infty_0(\R \backslash \{0\})$. Then there exists $t_0>0$ small enough and $C>0$ such that for all $u_0 \in L^1(\R^d)$ and all $h \in (0,1]$,
\begin{align}
\|e^{-ith^{-1}(h\Lambda_g)^\sigma} \varphi(h^2P)u_0\|_{L^\infty} \leq Ch^{-d}(1+|t|h^{-1})^{-d/2} \|u_0\|_{L^1}, \label{dispersive fractional schrodinger on manifolds}
\end{align}
for all $t\in [-t_0, t_0]$.
\end{prop}
Next, we recall the following version of $TT^\star$-criterion of Keel and Tao (see \cite{Zhang},  \cite{KeelTaoTTstar} or \cite{Zworski}).
\begin{prop} \label{prop TTstar}
Let $I \subseteq \R$ be an interval and $(T(t))_{t\in I}$ a family of linear operators satisfying for some constant $C>0$ and $\delta, \tau, h >0$,
\begin{align}
\|T(t)\|_{\Lc(L^2)} &\leq  C, \label{energy estimate} \\
\|T(t)T(s)^\star\|_{\Lc(L^1,L^\infty)} &\leq C h^{-\delta} (1+|t-s|h^{-1})^{-\tau}, \label{dispersive estimate}
\end{align}
for all $t,s \in I$. Then for all $(p,q)$ satisfying
\begin{align}
p \in [2,\infty], \quad q\in [1,\infty], \quad (p,q,\tau) \ne (2,\infty,1), \quad \frac{1}{p} \leq \tau\left(\frac{1}{2}-\frac{1}{q} \right), \nonumber 
\end{align}
we have
\[
\|Tv\|_{L^p(I,L^q}  \leq C h^{-\kappa} \|v\|_{L^2},
\]
where $\kappa=\delta(1/2-1/q)-1/p$. 
\end{prop}
Proposition $\ref{prop TTstar}$ together with energy estimate and dispersive estimate $(\ref{dispersive fractional schrodinger on manifolds})$ give the following result.
\begin{coro} \label{coro Keel Tao}
Let $\sigma \in (0,\infty)\backslash \{1\}, \varphi \in C^\infty_0(\R \backslash \{0\})$ and $t_0$ be as in $\emph{Theorem }$ $\ref{theorem wkb for fractional schrodinger}$. Denote $I=[-t_0,t_0]$. Then
for all $(p,q)$ fractional admissible, there exists $C>0$ such that
\begin{align}
\|\varphi(h^2P) e^{-ith^{-1}(h\Lambda_g)^\sigma} v\|_{L^p(I,L^q)} \leq C h^{-\kappa_{p,q}}\|v\|_{L^2}, \label{homogeneous strichartz with spectral localization}
\end{align}
where $\kappa_{p,q}=d/2-d/q-1/p$. Moreover,
\begin{align}
\Big\|\int_0^t \varphi^2(h^2P)e^{-i(t-s)h^{-1}(h\Lambda_g)^\sigma} G(s)ds \Big\|_{L^p(I,L^q)} \leq C h^{-\kappa_{p,q}} \|G\|_{L^{1}(I,L^{2})}. \label{inhomogeneous strichartz with spectral localization}
\end{align}
\end{coro}
\begin{proof}
The homogeneous estimates $(\ref{homogeneous strichartz with spectral localization})$ follow directly from Proposition $\ref{prop dispersive estimates fractional schrodinger}$ and Proposition $\ref{prop TTstar}$ with $T(t)=\varphi(h^2P)e^{-ith^{-1}(h\Lambda_g)^\sigma}$. It remains to prove the inhomogeneous estimates $(\ref{inhomogeneous strichartz with spectral localization})$. Let us set
\[
U_h(t):= h^{\kappa_{p,q}} \varphi(h^2P)e^{-ith^{-1}(h\Lambda_g)^\sigma}.
\]
Using the homogeneous Strichartz estimates $(\ref{homogeneous strichartz with spectral localization})$, we see that $U_h(t)$ is a bounded operator from $L^2$ to $L^p(I,L^q)$. Similarly, we have $U_h(s) = \varphi(h^2P)e^{-ish^{-1}(h\Lambda_g)^\sigma}$ is a bounded operator from  $L^2$ to $L^\infty(I,L^2)$. Here we use the fact that $(\infty,2)$ is fractional admissible with $\kappa_{\infty,2}=0$. Thus the adjoint $U_h(s)^\star$, namely 
\[
U_h(s)^\star: G \in L^1(I,L^2) \mapsto  \int_{I} \varphi(h^2P) e^{ish^{-1}(h\Lambda_g)^\sigma} G(s)ds \in L^2
\]
is also a bounded operator. This implies $U_h(t)U_h(s)^\star$ is a bounded operator from $L^1(I,L^2)$ to $L^p(I,L^q)$. In particular, we have
\[
\Big\|\int_{I} h^{\kappa_{p,q}} \varphi^2(h^2P) e^{-i(t-s)h^{-1}(h\Lambda_g)^\sigma} G(s)ds\Big\|_{L^p(I,L^q)} \leq C\|G\|_{L^1(I,L^2)}.
\]
The Christ-Kiselev Lemma (see Lemma $\ref{lem Christ-Kiselev lemma}$) implies that for all $(p,q)$ fractional admissible,
\[
\Big\|\int_0^t \varphi^2(h^2P) e^{-i(t-s)h^{-1}(h\Lambda_g)^\sigma} G(s)ds\Big\|_{L^p(I,L^q)} \leq C h^{-\kappa_{p,q}}\|G\|_{L^1(I,L^2)}.
\]
This completes the proof.
\end{proof}
\subsection{From local Strichartz estimates to global Strichartz estimates} 
We now show how to upgrade the local in time Strichartz estimates given in Corollary $\ref{coro Keel Tao}$ to the global in time ones $(\ref{reduction chi u fractional schrodinger})$. We emphasize that the non-trapping assumption is supposed here. \newline
\indent Let us set $v(t)=\scal{x}^{-1}f(h^2P) e^{-ith^{-1}(h\Lambda_g)^\sigma}u_0$. By choosing $f_1\in C^\infty_0(\R \backslash \{0\})$ with $f_1=1$ near $\text{supp}(f)$, we see that the study of $\|v\|_{L^p(\R,L^q)}$ is reduced to the one of $\|f_1(h^2P)v\|_{L^p(\R,L^q)}$. Indeed, we can write
\[
v(t)=f_1(h^2P) v(t)+ (1-f_1)(h^2P) v(t),
\]
where the term $(1-f_1)(h^2P) v(t)$ can be written as
\[
((1-f_1)(h^2P) \scal{x}^{-1} \tilde{f}_1(h^2P) \scal{x} ) \scal{x}^{-1}f(h^2P)e^{-ith^{-1}(h\Lambda_g)^\sigma} u_0,
\]
with $\tilde{f}_1 \in C^\infty_0(\R \backslash \{0\})$ such that $f_1=1$ near $\text{supp}(\tilde{f}_1)$ and $\tilde{f}_1=1$ near $\text{supp}(f)$. By pseudo-differential calculus, we have
\[
(1-f_1)(h^2P) \scal{x}^{-1} \tilde{f}_1(h^2P) \scal{x} = O_{\Lc(L^2, L^q)}(h^{\infty}),
\]
for all $q\geq 2$. This implies that there exists $C>0$ such that for all $N \geq 1$,
\begin{align}
\| v - f_1(h^2P)v\|_{L^p(\R,L^q)}& \leq C h^N\|\scal{x}^{-1} f(h^2P) e^{-ith^{-1}(h\Lambda_g)^\sigma} u_0\|_{L^p(\R,L^2)} \nonumber \\
&\leq C h^{N} \|f(h^2P)u_0\|_{L^2} \leq C h^{-\kappa_{p,q}}\|f(h^2P)u_0\|_{L^2} \label{lp lq estimate v}
\end{align}
provided that $N$ is taken large enough. 
Here we use $(\ref{Lp global integrability})$ with $N_0=1$ due to the non-trapping condition.\newline 
\indent We next write
\[
v(t)=\scal{x}^{-1}f(h^2P) e^{-ith^{-1} \psi(h^2P)}u_0,
\]
where $\psi(\lambda)=\tilde{f}(\lambda) \sqrt{\lambda}^\sigma$ with $\tilde{f} \in C^\infty_0(\R \backslash \{0\})$ and $\tilde{f}=1$ near $\text{supp}(f)$. Now, let $t_0>0$ be as in Corollary $\ref{coro Keel Tao}$. We next choose $\theta \in C^\infty_0(\R, [0,1])$ satisfying $\theta =1$ near 0 and $\text{supp}(\theta) \subset (-1,1)$ such that $\sum_{k \in \Z} \theta(t-k)=1$, for all $t\in \R$. We then write $v(t)=\sum_{k\in \Z} v_k(t)$, where $v_k(t)= \theta((t - t_k)/t_0) v(t)$ with $t_k=t_0k$. By the Duhamel formula, we have
\[
v_k(t)= e^{-ith^{-1}\psi(h^2P)} v_k(0) + ih^{-1} \int_{0}^{t} e^{-i(t-s)h^{-1}\psi(h^2P)} (hD_s+\psi(h^2P)) v_k(s)ds.
\]
For $k\ne 0$, we compute the action of $hD_s +\psi(h^2P)$ on $v_k(s)$ and get
\begin{align*}
(hD_s+\psi(h^2P)) v_k(s)&= h(it_0)^{-1} \theta' ((s-t_k)/t_0) v(s) \\ 
&+ \theta ((s-t_k)/t_0) \left[\psi(h^2P), \scal{x}^{-1} \right] f(h^2P) e^{-is h^{-1} \psi(h^2P)} u_0 =:  v^1_k(s)+v^2_k(s). \nonumber
\end{align*}
Due to the support property of $\theta$, we have $v_k(0)=0$.  Now, we have for $k\ne 0$,
\[
f_1(h^2P) v_k(t)= ih^{-1} \int_{0}^{t} e^{-i(t-s)h^{-1}\psi(h^2P)} f_1(h^2P) (v^1_k(s)+v^2_k(s))ds.
\]
We remark that both $t,s$ belong to $I_k= (t_k-t_0, t_k+t_0)$. Up to a translation in time $t\mapsto t-t_k$ and the same for $s$, we can apply the inhomogeneous Strichartz estimates given in Corollary $\ref{coro Keel Tao}$ with $\varphi^2=f_1$ and obtain
\begin{align}
\|f_1(h^2P) v_k\|_{L^p(\R,L^q)}&=\|f_1(h^2P) v_k\|_{L^p(I_k,L^q)} \nonumber \\
&\leq C h^{-\kappa_{p,q}-1} \left(\|v^1_k\|_{L^1(I_k,L^2)} + \|v^2_k\|_{L^1(I_k,L^2)} \right). \nonumber
\end{align}
Here $\kappa_{p,q}$ is given in Corollary $\ref{coro Keel Tao}$. We have
\begin{align}
\|v^1_k\|_{L^1(I_k,L^2)} &= \| h (i t_0)^{-1} \theta' ((s -t_k)/t_0) \scal{x}^{-1} f(h^2P) e^{-ish^{-1}(h\Lambda_g)^\sigma} u_0 \|_{L^1(I_k,L^2)} \nonumber \\
&\leq \| h (i t_0)^{-1} \theta' ((s -t_k)/t_0) \|_{L^2(I_k)} \|\scal{x}^{-1} f(h^2P) e^{-ish^{-1}(h\Lambda_g)^\sigma}u_0\|_{L^2(I_k,L^2)} \nonumber \\
&\leq Ch\|\scal{x}^{-1} f(h^2P) e^{-ish^{-1}(h\Lambda_g)^\sigma}u_0\|_{L^2(I_k,L^2)}, \nonumber
\end{align}
where we use Cauchy Schwarz inequality to go from the first to the second line. Similarly
\begin{align}
\|v^2_k\|_{L^1(I_k,L^2)} & \leq    \| [ \psi(h^2P),\scal{x}^{-1}] f(h^2P) e^{-ith^{-1}(h\Lambda_g)^\sigma}u_0 \|_{L^2(I_k,L^2)} \nonumber \\
& \leq C h\|\scal{x}^{-1} f(h^2P) e^{-ith^{-1}(h\Lambda_g)^\sigma} u_0 \|_{L^2(I_k,L^2)}, \nonumber
\end{align}
where we use the fact that $[\psi(h^2P),\scal{x}^{-1} ] \tilde{f}_1(h^2P) \scal{x} $ is of size $O_{\Lc(L^2)}(h)$ by pseudo-differential calculus. This implies that for $k \ne 0$,
\[
\|f_1(h^2P) v_k\|_{L^p(I_k,L^q)} \leq C h^{-\kappa_{p,q}} \|\scal{x}^{-1} f(h^2P) e^{-ith^{-1}(h\Lambda_g)^\sigma} u_0 \|_{L^2(I_k,L^2)}.
\]
For $k=0$, we have
\[
\|f_1(h^2P) v_0\|_{L^p(\R,L^q)} \leq C \|f(h^2P) e^{-ith^{-1}(h\Lambda_g)^\sigma} u_0\|_{L^p(I,L^q)} \leq C h^{-\kappa_{p,q}} \|f(h^2P)u_0\|_{L^2}.
\]
Here the first inequality follows from the facts that $\theta(t/t_0)$ and $f_1(h^2P)\scal{x}^{-1}$ are bounded in $\Lc(L^p(\R))$ and $\Lc(L^q)$ respectively. The second inequality follows from homogeneous Strichartz estimates $(\ref{homogeneous strichartz with spectral localization})$. By almost orthogonality in time and the fact that $p\geq 2$, we have
\begin{align*}
\|f_1(h^2P)v\|_{L^p(\R,L^q)} &\leq C\Big(\sum_{k\in \Z} \|f_1(h^2P) v_k\|^2_{L^p(\R,L^q)}  \Big)^{1/2} \nonumber \\
&\leq Ch^{-\kappa_{p,q}}\Big(\sum_{k\in \Z \backslash 0} \|\scal{x}^{-1} f(h^2P) e^{-ith^{-1}(h\Lambda_g)^\sigma} u_0 \|^2_{L^2(I_k,L^2)} + \|f(h^2P)u_0\|^2_{L^2} \Big)^{1/2}  \nonumber \\
&\leq  C h^{-\kappa_{p,q}}\Big(\|\scal{x}^{-1} f(h^2P) e^{-ith^{-1}(h\Lambda_g)^\sigma} u_0 \|_{L^2(\R,L^2)} + \|f(h^2P)u_0\|_{L^2}\Big) \nonumber \\
&\leq C h^{-\kappa_{p,q}}\|f(h^2P)u_0\|_{L^2}, \nonumber
\end{align*}
the last inequality comes from Proposition $\ref{prop global L2 integrability fractional schrodinger}$ with $N_0=1$. By using $(\ref{lp lq estimate v})$, we obtain
\[
\|\scal{x}^{-1}f(h^2P) e^{-ith^{-1}(h\Lambda_g)^\sigma} u_0\|_{L^p(\R,L^q)} \leq C h^{-\kappa_{p,q}}\|f(h^2P)u_0\|_{L^2}.
\]
This implies that for all $\chi \in C^\infty_0(\R^d)$, 
\[
\Big\|\chi f(h^2P) e^{-ith^{-1}(h\Lambda_g)^\sigma} u_0\Big\|_{L^p(\R,L^q)} \leq C h^{-\kappa_{p,q}} \|f(h^2P)u_0\|_{L^2}.
\]
Therefore, by scaling in time, we get
\[
\Big\|\chi f(h^2P) e^{-it\Lambda_g^\sigma} u_0\Big\|_{L^p(\R,L^q)} \leq C h^{-\gamma_{p,q}} \|f(h^2P)u_0\|_{L^2}.
\]
The proof of $(\ref{reduction chi u fractional schrodinger})$ is now complete. 
\defendproof
\section{Strichartz estimates outside compact sets} \label{section strichartz estimates outside compact}
\setcounter{equation}{0}
\subsection{The Isozaki-Kitada parametrix}
\paragraph{Notations and the Hamilton-Jacobi equations.}
For any $J \Subset (0,+\infty)$ an open interval, any $R>0$, any $\tau \in (-1,1)$, we define the outgoing region $\Ga^+(R,J,\tau)$ and the incoming region $\Ga^-(R,J,\tau)$ by 
\[
\Ga^\pm(R,J,\tau):= \Big\{(x,\xi)\in \R^{2d}, |x| >R, |\xi|^2 \in J, \pm \frac{x\cdot \xi}{|x\|\xi|}>\tau \Big\}. 
\]
\indent Let $\sigma \in (0,\infty)$\footnote{The construction of the Isozaki-Kitada parametrix we present here works well for the (half) wave equation, i.e. $\sigma=1$.}. We will use the so called Isozaki-Kitada parametrix to give an approximation at \textbf{high frequency} of the form
\begin{align}
e^{-it h^{-1} \psi(h^2P)} Op^h(\chi^\pm) = J^\pm_h(a^\pm (h)) e^{-it h^{-1}(h\Lambda)^\sigma} J^\pm_h(b^\pm(h))^\star + R^\pm_N(h), \label{IK high freq construction}
\end{align}
with $\Lambda=\sqrt{-\Delta}$ where $\Delta$ is the free Laplacian operator on $\R^d$ and $\psi(\cdot)=\tilde{f}(\cdot) \sqrt{\cdot}^\sigma \in C^\infty_0(\R \backslash \{0\})$ for some $\tilde{f} \in C^\infty_0(\R \backslash\{0\})$ satisfying $\tilde{f}=1$ near $\text{supp}(f)$. The functions $\chi^\pm$ are supported in $\Ga^\pm(R^4,J_4,\tau_4)$ (see Proposition $\ref{prop find b}$ for the choice of $J_4$ and $\tau_4$) and 
\[
J^\pm_h(a^\pm(h))= \sum_{j=1}^{N-1} h^j J^\pm_h(a^\pm_j),
\]
where
\[
J^\pm_h(a^\pm)u(x)= (2\pi h)^{-d} \iint_{\R^{2d}} e^{ih^{-1}\left(S^\pm_{R}(x,\xi)- y\cdot\xi \right)} a^\pm(x,\xi) u(y) dy d\xi, \quad u \in \Sch(\R^d).
\]
The amplitude functions $a^\pm_j$ are supported in $\Ga^\pm(R,J_1,\tau_1)$ (see Proposition $\ref{prop hamilton-jacobi equation}$) and the phase functions $S^\pm_{R}:= S^\pm_{1,R}$ will be described later. The same notation for $J^\pm_h(b^\pm(h))$ is used with $b^\pm_k$ in place of $a^\pm_j$. \newline
\indent The Isozaki-Kitada parametrix at \textbf{low frequency} is of the form
\begin{align}
e^{-it \ep \psi(\ep^{-2}P)} Op_\ep(\chi^\pm_\ep) \zeta(\ep x) = \Jc^\pm_\ep(a^\pm_\ep) e^{-it\ep\Lambda^\sigma} \Jc^\pm_\ep(b^\pm_\ep)^\star + \Rc^\pm_N(t,\ep), \label{IK low freq construction}
\end{align}
where $\psi$ is as above and $\zeta \in C^\infty(\R^d)$ supported outside $B(0,1)$ and equal to 1 near infinity. The functions $\chi^\pm_\ep$ are supported in $\Ga^\pm(R^4, J_4, \tau_4)$ and 
\[
\Jc^\pm_\ep(a^\pm_\ep) = \sum_{j=1}^{N} \Jc^\pm_\ep(a^\pm_{\ep, j}),
\]
where 
\begin{align}
\Jc^\pm_\ep(a):= D_\ep J^\pm_\ep(a), \quad \Jc^\pm_\ep(b)^\star:=J^\pm_\ep(b)^\star D_\ep^{-1}, \label{rescaled FIO}
\end{align}
with $D_\ep$ as in $(\ref{define D_epsilon})$,
\[
J^\pm_\ep(a)u(x):= (2\pi)^{-d}\iint_{\R^{2d}} e^{i(S^\pm_{\ep,R}(x,\xi)-y\cdot\xi)} a(x,\xi) u(y)dy d\xi,
\]
and 
\[
J^\pm_\ep(b)^\star u(x) = (2\pi)^{-d} \iint_{\R^{2d}} e^{i(x\cdot\xi -S^\pm_{\ep,R}(y,\xi))} \overline{b(y,\xi)} u(y)dyd\xi.
\]
The amplitude functions $a^\pm_{\ep,j}$ are supported in $\Ga^\pm(R, J_1, \tau_1)$ and the phase functions $S^\pm_{\ep,R}$ will be described in the next proposition. The same notation for $\Jc^\pm_\ep(b^\pm_\ep)$ will be used with $b^\pm_\ep$ in place of $a^\pm_\ep$. 
\begin{prop}\label{prop hamilton-jacobi equation}
Fix $J_1 \Subset (0,+\infty)$ and $\tau_1 \in (-1,1)$. Then there exists two families of smooth functions $(S^\pm_{\ep,R})_{R\gg 1}$ satisfying the following Hamilton-Jacobi equation
\begin{align}
p_\ep(x,\nabla_x S^\pm_{\ep,R}(x,\xi))= |\xi|^2, \label{hamilton-jacobi equation}
\end{align}
for all $(x,\xi) \in \Ga^\pm(R,J_1, \tau_1)$, where $p_\ep$ is given in $(\ref{define p_epsilon})$. Moreover, for all $\alpha, \beta \in \N^d$, there exists $C_{\alpha\beta}>0$ such that 
\begin{align} \label{estimate on phase}
\Big| \partial^\alpha_x \partial^\beta_\xi \Big(S^\pm_{\ep,R}(x,\xi)-x\cdot \xi \Big) \Big| \leq C_{\alpha\beta} \min\Big\{
R^{1-\rho -|\alpha|}, \scal{x}^{1-\rho-|\alpha|} \Big\},
\end{align}
for all $x,\xi \in \R^d$, all $\ep \in (0,1]$ and $R \gg 1$. 
\end{prop}
\begin{rem} \label{rem hamilton-jacobi equation}
From $(\ref{estimate on phase})$, we see that for $R>0$ large enough, the phase functions satisfy for all $x,\xi \in \R^d$ and all $\ep \in (0,1]$,
\begin{align} \label{matrix phase estimate}
\Big\| \nabla_x \cdot  \nabla_\xi S^\pm_{\ep,R} (x,\xi) - \text{Id}_{\R^d} \Big\| \leq \frac{1}{2},
\end{align}
and for all $|\alpha| \geq 1$ and all $|\beta| \geq 1$, 
\begin{align}\label{high order phase estimate}
| \partial^\alpha_x  \partial^\beta_\xi S^\pm_{\ep,R} (x,\xi)| \leq C_{\alpha \beta}.
\end{align}
The estimates $(\ref{matrix phase estimate})$ and $(\ref{high order phase estimate})$ are useful in the construction of Isozaki-Kitada parametrix as well as the $L^2$-boundedness of Fourier integral operators.
\end{rem}
\noindent \textit{Proof of \emph{Proposition} $\ref{prop hamilton-jacobi equation}$.} We firstly note that the case $\ep=1$ is given in \cite[Proposition 3.1]{BTlocalstrichartz}. Let $J_1 \Subset J_0 \Subset (0,+\infty)$ and $-1< \tau_0 < \tau_1 <1$. By using Lemma $\ref{rem property rescaled symbol}$, in the region $\Ga^\pm(R/2, J_0, \tau_0)$ which implies that $|x|>1$, we see that the function $p_\ep(x,\xi)$ satisfies for all $\alpha, \beta \in \N^d$, there exists $C_{\alpha\beta}>0$ such that for all $(x, \xi) \in \Ga^\pm(R/2, J_0, \tau_0)$ and all $\ep \in (0,1]$,
\[
|\partial^\alpha_x \partial^\beta_\xi p_\ep(x,\xi)| \leq C_{\alpha\beta}\scal{\xi}^{2-|\beta|}.
\]
Thanks to this uniform bound, by using the argument given in \cite[Proposition 4.1]{Robertsmoothing}, we can solve (for $R>0$ large enough) the Hamilton-Jacobi equation $(\ref{hamilton-jacobi equation})$ in $\Ga^\pm(R/2,J_0, \tau_0)$ uniformly with respect to $\ep \in (0,1]$. We denote such solutions by $\tilde{S}^\pm_\ep$. Next, by choosing a special cutoff (see \cite{BTlocalstrichartz}, see also $(\ref{define cutoff})$) $\chi^\pm_R \in S(0,-\infty)$ such that $\chi^\pm_R(x,\xi) =1$ for $(x,\xi) \in \Ga^\pm(R,J_1, \tau_1)$ and $\text{supp}(\chi^\pm_R) \subset \Ga^\pm(R/2, J_0, \tau_0)$, then the functions
\[
S^\pm_{\ep,R}(x,\xi) = \chi^\pm_R(x,\xi) \tilde{S}^\pm_{\ep}(x,\xi) + (1-\chi^\pm_R)(x,\xi) \scal{x,\xi}
\]
satisfy the properties of Proposition $\ref{prop hamilton-jacobi equation}$, where $\scal{x,\xi}=x\cdot\xi$.
\defendproof 
\paragraph{Construction of the parametrix.}
Let us firstly consider the high frequency case $(\ref{IK high freq construction})$. The construction in the low frequeny case $(\ref{IK low freq construction})$ is similar up to some modifications (see after Theorem $\ref{theorem IK parametrix fractional schrodinger equation}$). We only treat the outgoing case (+), the incoming one is similar. We start with the following Duhamel formula
\begin{multline}
e^{-ith^{-1}\psi(h^2P)} J^+_h(a^+(h)) = J^+_h(a^+(h)) e^{-ith^{-1}(h\Lambda)^\sigma}   \\ 
 -ih^{-1} \int_{0}^{t} e^{-i(t-s)h^{-1}\psi(h^2P)} \Big(\psi(h^2P)J^+_h(a^+(h)) - J^+_h(a^+(h)) (h\Lambda)^\sigma\Big)e^{-is h^{-1}(h\Lambda)^\sigma} ds. \label{IK Duhamel formula}
\end{multline}
We want the term $\psi(h^2P)J^+_h(a^+(h)) - J^+_h(a^+(h)) (h\Lambda)^\sigma$ to have a small contribution. To do so, we firstly introduce a special cutoff. For any $J_2 \Subset J_1 \Subset (0,+\infty)$ and $-1 <\tau_1 <\tau_2 <1$, we define 
\begin{align} \label{define cutoff}
\chi^+_{1\rightarrow 2}(x,\xi) = \kappa\left(\frac{|x|}{R^2}\right) \rho_{1\rightarrow 2}(|\xi|^2) \theta_{1\rightarrow 2}\left(+\frac{x\cdot\xi}{|x\|\xi|}\right),
\end{align}
where $\kappa \in C^\infty(\R)$ is non-decreasing such that 
\[
\kappa(t)= \left\{
\begin{array}{ll}
1 & \text{ when } t \geq 1/2 \\
0 & \text{ when } t \leq 1/4 
\end{array}, \right.
\]
and $\rho_{1\rightarrow 2} \in C^\infty(\R)$ is non-decreasing such that $\rho_{1\rightarrow 2}=1$ near $J_2$, supported in $J_1$ and $\theta_{1 \rightarrow 2} \in C^\infty_0(\R)$ such that 
\[
\theta_{1\rightarrow 2}(t)= \left\{
\begin{array}{ll}
1 & \text{ when } t > \tau_2 -\varepsilon \\
0 & \text{ when } t < \tau_1 +\varepsilon
\end{array}, \right.
\]
with $\varepsilon \in (0,\tau_2-\tau_1)$. We see that $\chi^+_{1\rightarrow 2} \in S(0,-\infty)$ and for $R\gg 1$, 
\[
\text{supp}(\chi^+_{1\rightarrow 2}) \subset \Ga^+(R,J_1,\tau_1), \quad \chi^+_{1 \rightarrow 2} =1 \text{ near } \Ga^+(R^2,J_2, \tau_2).
\]
\begin{prop} \label{prop find a}
Let $S^+_{R}:=S^+_{1,R}$ be the solution of $(\ref{hamilton-jacobi equation})$ given as in $\emph{Proposition}$ $\ref{prop hamilton-jacobi equation}$. Let $J_2$ be an arbitrary open interval such that $J_2 \Subset J_1 \Subset (0,+\infty)$ and $\tau_2$ be an arbitrary real number such that $-1 < \tau_1 <\tau_2 <1$. Then for $R> 0$ large enough, we can find a sequence of symbols $a^+_j \in S(-j,-\infty)$ supported in $\Ga^+(R,J_1,\tau_1)$ such that for all $N \geq 1$, 
\begin{align}
\psi(h^2P) J^+_h(a^+(h)) -J^+_h(a^+(h)) (h\Lambda)^\sigma &= h^N R_N(h) J^+_h(a^+(h)) + h^N J^+_h(r_N^+(h)) + J^+_h(\check{a}^+(h)),  \label{express of small contribution}\\
\sup_{\Ga^+(R,J_1,\tau_1)} |a_0^+(x,\xi)| &\gtrsim 1, \label{non vanishing a_0}
\end{align}
where $a^+(h)=\sum_{j=0}^{N-1}h^j a^+_j$ and $(r_N^+(h))_{h\in (0,1]}$ is bounded in $S(-N,-\infty)$, $R_N(h)$ is as in $\emph{Proposition}$ $\ref{prop parametrix phi}$, $(\check{a}^+(h))_{h\in (0,1]}$ is bounded in $S(0,-\infty)$ and is a finite sum depending on $N$ of the form
\begin{align} \label{define a check}
\check{a}^+(h)=\sum_{|\alpha| \geq 1} \check{a}^+_\alpha(h) \partial^\alpha_x \chi^+_{1\rightarrow 2}, 
\end{align}
with $(\check{a}^+_\alpha(h))_{h\in (0,1]}$ bounded in $S(0,-\infty)$ and $\chi^+_{1 \rightarrow 2}$ given in $(\ref{define cutoff})$.
\end{prop}
\begin{proof}
We firstly use the parametrix of $\psi(h^2P)$ given in Proposition $\ref{prop parametrix phi}$ and get
\begin{align}
\psi(h^2P)= Op^h(q(h))+h^N R_N(h), \label{parametrix psi high freq}
\end{align}
where $q(h)=\sum_{k=0}^{N-1} h^k q_k$ and $q_k \in S(-k,-\infty), k=0,...,N-1$. Note that $q_0(x,\xi)=\psi(p(x,\xi))=\tilde{f}(p(x,\xi))\sqrt{p(x,\xi)}^\sigma$ and $\text{supp}(q_k)\subset \text{supp}(\psi\circ p)$. Up to remainder term, we consider the action of $Op^h(q(h))$ on $J^+_h(a^+(h))$. To do this, we need the following action of a pseudo-differential operator on a Fourier integral operator (see e.g. \cite[Theorem IV-19]{Robert}, \cite[Appendix]{BoucletthesisphD} or \cite{RuzhanskySugimoto}).
\begin{prop} \label{prop action PDO on FIO}
Let $a \in S(\mu_1,-\infty)$ and $b \in S(\mu_2,-\infty)$ and $S$ satisfy $(\ref{matrix phase estimate})$ and $(\ref{high order phase estimate})$. Then
\[
Op^h(a)\circ J_h(S,b) = \sum_{j=0}^{N-1}h^j J_h(S,(a \triangleleft b)j)+h^NJ_h(S,r_N(h)),
\]
where $(a\triangleleft b)_j$ is an universal linear combination of 
\[
\partial^\beta_\xi a(x, \nabla_xS(x,\xi)) \partial^{\beta-\alpha}_x b(x,\xi) \partial^{\alpha_1}_x  S(x,\xi)\cdots \partial^{\alpha_k}_x S(x,\xi),
\]
with $\alpha \leq \beta, \alpha_1+\cdots +\alpha_k =\alpha$ and $|\alpha_l|\geq 2$ for all $l=1,...,k$ and $|\beta|=j$. The maps $(a,b) \mapsto (a\triangleleft b)_j$ and $(a,b)\mapsto r_N(h)$ are continuous from $S(\mu_1,-\infty) \times S(\mu_2,-\infty)$ to $S(\mu_1+\mu_2-j,-\infty)$ and $S(\mu_1+\mu_2-N,-\infty)$ respectively. In particular, we have
\begin{align}
(a\triangleleft b)_0(x,\xi)&= a(x,\nabla_xS(x,\xi)) b(x,\xi), \nonumber \\
i (a\triangleleft b)_1(x,\xi) &= \nabla_\xi a(x,\nabla_xS(x,\xi)) \cdot\nabla_xb(x,\xi) + \frac{1}{2}\emph{tr}\left( \nabla^2_\xi a (x,\nabla_xS(x,\xi))\cdot\nabla^2_xS(x,\xi) \right) b(x,\xi). \nonumber 
\end{align}
\end{prop}
Using this result, we have 
\begin{align}
Op^h(q(h))J^+_h(a^+(h))= \sum_{k+j+l=0}^{N-1}h^{k+j+l} J^+_h((q_k \triangleleft a^+_j)_l)+h^N J^+_h(r^+_N(h)). \nonumber
\end{align}
On the other hand, we have
\[
J^+_h(a^+(h)) (h\Lambda)^\sigma= J^+_h(a^+(h)|\xi|^\sigma).
\]
Thus we get
\begin{align}
\psi(h^2P)J^+_h(a^+(h))- J^+_h(a^+(h)) (h\Lambda)^\sigma& =  \sum_{r=0}^{N-1} h^r J^+_h \left( \sum_{k+j+l=r}(q_k\triangleleft a^+_j)_l - a^+_r |\xi|^\sigma \right) \nonumber \\
 &  \mathrel{\phantom{=}}  + h^N J^+_h(r^+_N(h)) + h^N R_N(h) J^+_h(a^+(h)). \nonumber
\end{align}
In order to make the left hand side of $(\ref{express of small contribution})$ small, we need to find $a^+_j \in S(-j,-\infty)$ supported in $\Ga^+(R,J_1,\tau_1)$ such that
\[
\sum_{k+j+l=r}(q_k\triangleleft a^+_j)_l - a^+_r |\xi|^\sigma =0, \quad r=0,...,N-1.
\]
In particular,
\[
\left(q_0(x,\nabla_x S^+_R(x,\xi))-|\xi|^\sigma\right)a^+_0(x,\xi)=0.
\]
By noting that if $p(x,\xi) \in \text{supp}(f)$ (see after $(\ref{IK high freq construction})$), then $q_0(x,\xi)= \sqrt{p(x,\xi)}^\sigma$. Thus in the region where the Hamilton-Jacobi equation $(\ref{hamilton-jacobi equation})$ with $\ep=1$ is satisfied, we need to show the following transport equations
\begin{align}
(q_0 \triangleleft a^+_0)_1 + (q_1 \triangleleft a^+_0)_0 &= 0 \label{transport_0} \\
(q_0 \triangleleft a^+_r)_1 + (q_1 \triangleleft a^+_r)_0 &= - \sum_{k+j+l=r+1 \atop j \leq r-1} (q_k \triangleleft a^+_j)_l, \quad r=1,...,N-1. \label{transport_r}
\end{align}
Here $(q_0\triangleleft a^+)_1+(q_1 \triangleleft a^+)_0$ can be written as
\[
i\Big[(q_0\triangleleft a^+)_1(x,\xi)+(q_1 \triangleleft a^+)_0(x,\xi)\Big]= \sum_{j=1}^{d} V^+_j(x,\xi) \partial_{x_j} a^+(x,\xi) + p^+_0(x,\xi) a^+(x,\xi),
\]
where 
\begin{align}
V^+_j(x,\xi) &= (\partial_{\xi_j}q_0)(x,\nabla_x S^+_R(x,\xi)), \nonumber \\
p^+_0(x,\xi) &= i q_1(x,\nabla_x S^+_R(x,\xi))+ \frac{1}{2} \text{tr} \Big[ \nabla^2_\xi q_0(x,\nabla_x S^+_R(x,\xi))\cdot \nabla^2_xS^+_R(x,\xi) \Big]. \nonumber
\end{align}
We now consider the flow $X^+(t,x,\xi)$ associated to $V^+=(V^+_j)_{j=1}^d$ as
\begin{align} \label{transport flow}
\left\{
\begin{array}{lcl}
\dot{X}^+(t) &=& V^+(X^+(t),\xi), \\
X^+(0) &=& x.
\end{array}
\right.
\end{align}
We have the following result (see \cite[Proposition 3.2]{Boucletdistribution} or  \cite[Appendix]{BoucletthesisphD}).
\begin{prop} \label{prop transport flow}
Let $\sigma \in (0,\infty)$, $J_1 \Subset (0,+\infty)$ and $-1<\tau_1 <1$. There exists $R>0$ large enough and $e_1>0$ small enough such that for all $(x,\xi)\in \Ga^+(R,J_1,\tau_1)$, the solution $X^+(t,x,\xi)$ to $(\ref{transport flow})$ is defined for all $t\geq 0$ and satisfies
\begin{align}
|X^+(t,x,\xi)| &\geq e_1(t+|x|), \label{estimate_X} \\
(X^+(t,x,\xi),\xi) &\in \Ga^+(R,J_1,\tau_1). \label{contain map X}
\end{align}
Moreover, for all $\alpha, \beta \in \N^d$, there exists $C_{\alpha\beta}>0$ such that for all $t\geq 0$ and all $h\in (0,1]$,
\begin{align}
|\partial^\alpha_x \partial^\beta_\xi (X^+(t,x,\xi)- x - \sigma t\xi |\xi|^{\sigma-2} )| \leq C_{\alpha\beta} \scal{t} \scal{x}^{-\rho -|\alpha|},
\end{align}
for all $(x,\xi)\in \Ga^+(R,J_1,\tau_1)$.
\end{prop}
Now, we can define for $(x,\xi)\in \Ga^+(R, J_1,\tau_1)$ the functions 
\begin{align}
A^+_0(x,\xi) &= \exp \Big( \int_{0}^{+\infty} p^+_0(X^+(t,x,\xi),\xi) dt  \Big), \nonumber \\
A^+_r(x,\xi) &=  \int_{0}^{+\infty} p^+_r(X^+(t,x,\xi),\xi) \exp \Big( \int_{0}^{t} p^+_0(X^+(s,x,\xi),\xi) ds \Big) dt,  \nonumber
\end{align}
for $r=1,...,N-1$, where
\[
p^+_r(x,\xi) = i\sum_{k+j+l=r+1 \atop j \leq r-1} (q_k\triangleleft A^+_j)_l (x,\xi).
\]
Using $(\ref{estimate_X})$ and the fact that $p^+_r \in S(-1-\rho-r,-\infty)$ for $r=0,...,N-1$, we see that $p^+_r(X^+(t,x,\xi))$ are integrable with respect to $t$. Hence $A^+_r(x,\xi)$ are well-defined. Moreover, we have (see e.g. \cite[Proposition 3.1]{Boucletdistribution}) that for all $(x,\xi)\in \Ga^+(R,J_1,\tau_1)$,
\begin{align}
|\partial^\alpha_x \partial^\beta_\xi (A^+_0(x,\xi)-1)| &\leq C_{\alpha\beta} \scal{x}^{-|\alpha|}, \label{non vanishing a_0 proof}\\
|\partial^\alpha_x \partial^\beta_\xi A^+_r(x,\xi)| & \leq  C_{\alpha\beta} \scal{x}^{-r-|\alpha|}. \nonumber
\end{align}
We also have that $A^+_0, A^+_r$ for  $r=1,...,N-1$ solve $(\ref{transport_0})$ and $(\ref{transport_r})$ respectively in $\Ga^+(R,J_1,\tau_1)$. Now, by setting $a^+_r=\chi^+_{1 \rightarrow 2}A^+_r$ (see $(\ref{define cutoff})$), we see that $a^+_r$ are globally defined on $\R^{2d}$ and $a^+_r \in S(-r,-\infty)$. It is easy to see $(\ref{non vanishing a_0})$ from $(\ref{non vanishing a_0 proof})$. We next insert $a^+(h)= \sum_{j=1}^{N-1}h^ja^+_j$ into the left hand side of $(\ref{express of small contribution})$ and get
\begin{align}
\psi(h^2P)J^+_h(a^+(h))- J^+_h(a^+(h)) (h\Lambda)^\sigma&=  \sum_{r=0}^{N-1} h^r J^+_h \left( \sum_{k+j+l=r}(q_k\triangleleft \chi^+_{1 \rightarrow 2}A^+_j)_l - \chi^+_{1\rightarrow 2}A^+_r |\xi|^\sigma \right) \nonumber \\
 & \mathrel{\phantom{=}} + h^N J^+_h(r^+_N(h)) + h^N R_N(h) J^+_h(a^+(h)). \nonumber
\end{align}
Using the expression of $(a\triangleleft b)_l$ given in Proposition $\ref{prop action PDO on FIO}$, we see that 
\[
(q_k\triangleleft \chi^+_{1 \rightarrow 2}A^+_j)_l = \chi^+_{1 \rightarrow 2}(q_k\triangleleft A^+_j)_l + \text{ terms in which derivatives fall into } \chi^+_{1 \rightarrow 2}.
\]
This gives $(\ref{express of small contribution})$ with $\check{a}^+(h)$ as in $(\ref{define a check})$. The proof is complete.
\end{proof}
We now are able to construct the symbols $b^+_k$, for $k=0,...,N-1$. 
\begin{prop} \label{prop find b}
Let $J_3,J_4$ and $\tau_3,\tau_4$ be such that $J_4 \Subset J_3 \Subset J_2$ and $-1< \tau_2 <\tau_3<\tau_4<1$. Then for $R>0$ large enough and all $\chi^+$ supported in $\Ga^+(R^4,J_4,\tau_4)$, there exists a sequence of symbols $b^+_k \in S(-k,-\infty)$, for $k=0,...,N-1$, supported in $\Ga^+(R^3,J_3,\tau_3)$
such that
\begin{align}
J^+_h(a^+(h))J^+_h(b^+(h))^\star = Op^h(\chi^+)+ h^N Op^h(\tilde{r}^+_N(h)), \label{expansion find b}
\end{align}
where $a^+(h)=\sum_{j=0}^{N-1}h^ja^+_j$ is given in \emph{Proposition} $\ref{prop find a}$ and $b^+(h)=\sum_{k=0}^{N-1}h^k b^+_k$ and $(\tilde{r}^+_N(h))_{h\in (0,1]}$ is bounded in $S(-N,-\infty)$.
\end{prop}
Before giving the proof, we need the following result (see \cite[Appendix]{BoucletthesisphD} or \cite[Lemma 3.3]{Boucletdistribution}). 
\begin{lem} \label{lem variable eta}
Let $S^+_{R}:=S^+_{1,R}$ be as in \emph{Proposition} $\ref{prop hamilton-jacobi equation}$. For $x,y,\xi \in \R^d$, we define
\begin{align}
\eta^+(R,x,y,\xi):= \int_{0}^{1} \nabla_x S^+_R (y + \lambda (x-y),\xi) d\lambda. \label{define eta}
\end{align}
Then for $R>0$ large enough, we have the following properties.
\begin{itemize}
\item[i.] For all $x,y \in \R^d$, the map $\xi \mapsto \eta^+(R,x,y,\xi)$ is a diffeomorphism from $\R^d$ onto itself. Let $\eta \mapsto \xi^+(R,x,y,\eta)$ be its inverse.
\item[ii.] There exists $C>1$ such that for all $x,y, \eta \in \R^d$, 
\[
C^{-1} \scal{\eta} \leq \scal{ \xi^+(R,x,y, \eta)} \leq C \scal{\eta}.
\]
\item[iii.] For all $\alpha, \alpha', \beta \in \N^d$, there exists $C_{\alpha\alpha'\beta}>0$ such that for all $x,y, \eta \in \R^d$ and all $k \leq |\alpha|, k' \leq |\alpha'|$,
\[
|\partial^\alpha_x \partial^{\alpha'}_y \partial^\beta_\eta \left( \xi^+(R,x,y,\eta)- \eta\right) | \leq C_{\alpha\alpha'\beta} \scal{x}^{-k} \scal{y}^{-\rho-k'} \scal{x-y}^{\rho+k+k'}.
\]
\end{itemize}
\end{lem}
\noindent \textit{Proof of \emph{Proposition} $\ref{prop find b}$.} We firstly consider the general term $J^+_h(a^+)J^+_h(b^+)^\star$ and write its kernel as
\[
K^+_h(x,y)= (2\pi h)^{-d} \int_{\R^d} e^{i h^{-1}\left( S^+_R(x,\xi)-S^+_R(y,\xi)\right)} a^+(x,\xi)\overline{b^+(y,\xi)} d\xi.
\]
By Taylor's formula, we have
\[
S^+_R(x,\xi)-S^+_R(y,\xi)= \scal{x-y, \eta^+(R,x,y,\xi)},
\]
where $\eta^+$ given in $(\ref{define eta})$. By change of variable $\xi \mapsto \xi^+(R,x,y,\eta)$, the kernel becomes
\[
K^+_h(x,y)=(2\pi h)^{-d} \int_{\R^d} e^{ih^{-1}(x-y)\eta} a^+(x,\xi^+(R,x,y,\eta)) \overline{b^+(y,\xi^+(R,x,y,\eta))} |\det \partial_\eta \xi^+(R,x,y,\eta)| d\eta.
\]
Now, using Lemma $(\ref{lem variable eta})$, the symbolic calculus gives
\[
J^+_h(a^+)J^+_h(b^+)^\star= \sum_{l=0}^{N-1} h^l Op^h((a^+\triangleright b^+)_l)+ h^N Op^h(\tilde{r}^+_N(h)),
\]
where $(a^+ \triangleright b^+)_l \in S(-l,-\infty)$ is of the form 
\[
(a^+\triangleright b^+)_l(x,\eta) = \sum_{|\alpha|=l}  \frac{\left.\partial^\alpha_y D_\eta^{\alpha} c^+(x,y,\eta) \right|_{y=x} }{\alpha!}, 
\]
for $l=0,...,N-1$ with
\[
c^+(x,y,\eta)= a^+(x,\xi^+(R,x,y,\eta)) \overline{b^+(y,\xi^+(R,x,y,\eta))} |\det \partial_\eta \xi^+ (R,x,y,\eta)|,
\]
and $(\tilde{r}_N^+(h))_{h\in(0,1]}$ is bounded in $S(-N,-\infty)$. We have now
\begin{align}
J^+_h (a^+(h))J^+_h(b^+(h))^\star&= \sum_{j,k} h^{j+k} J^+_h(a^+_j)J^+_h(b^+_k)^\star \nonumber \\
&= \sum_{j+k+l=0}^{N-1} h^{j+k+l} Op^h ((a^+_j\triangleleft b^+_k)_l) + h^N Op^h(\tilde{r}^+_N(h)). \nonumber
\end{align}
Compare with $(\ref{expansion find b})$, the result follows if we solve the following equations:
\begin{align}
(a^+_0 \triangleleft b^+_0)_0 &= \chi^+, \nonumber \\
(a^+_0 \triangleleft b^+_r)_0 &= - \sum_{j+k+l=r \atop k \leq r-1} (a^+_j \triangleleft b^+_k)_l, \quad r=1,...,N-1. \nonumber
\end{align}
We can define $b^+_0,...,b^+_{N-1}$ iteratively by
\begin{align}
\overline{b^+_0(x,\xi)} &= \chi^+(x,\eta^+(R,x,x,\xi)) \Big( a^+_0(x,\xi) \left| \det \partial_\eta \xi^+ (R,x,x,\eta^+(R,x,x,\xi))\right| \Big)^{-1}, \nonumber \\
\overline{b^+_r(x,\xi)} &= - \sum_{j+k+l=r \atop k \leq r-1} (a^+_j\triangleleft b^+_k)_l(x,\eta^+(R,x,x,\xi)) \Big( a^+_0(x,\xi) \left| \det \partial_\eta \xi (R,x,x,\eta^+(R,x,x,\xi))\right| \Big)^{-1},  \nonumber
\end{align}
for $r=1,...,N-1$. Note that by $(\ref{non vanishing a_0})$ and Lemma $\ref{lem variable eta}$, the term in $(\cdots)^{-1}$ cannot vanish on the support of $\chi^+(\cdot,\eta^+(R,\cdot,\cdot,\cdot))$. Thus the above functions are well-defined. Moreover, by choosing $R>0$ large enough with the fact
\[
\eta^+(R,x,x,\xi)= \nabla_x S^+_R(x,\xi)=\xi + O(\min\{R^{-\rho},\scal{x}^{-\rho}\}),
\]
we see that the support of $\chi^+(x,\eta^+(R,x,x,\xi))$ is contained in $\Ga^+(R^3,J_3,\tau_3)$. This completes the proof of Proposition $\ref{prop find b}$. 
\defendproof \newline
\indent By $(\ref{IK Duhamel formula})$, Proposition $\ref{prop find a}$ and Proposition $\ref{prop find b}$, we are able to state the Isozaki-Kitada parametrix for the fractional Schr\"odinger equation at high frequency.
\begin{theorem} \label{theorem IK parametrix fractional schrodinger equation}
Let $\sigma \in (0,\infty)$. Fix $J_4 \Subset (0,+\infty)$ open interval containing $\emph{supp}(f)$ and $-1 <\tau_4 <1$. Choose arbitrary open intervals $J_1, J_2, J_3 $ such that $J_4 \Subset J_3 \Subset J_2 \Subset J_1 \Subset (0,+\infty)$ and arbitrary $\tau_1,\tau_2,\tau_3$ such that $-1<\tau_1 <\tau_2 <\tau_3 < \tau_4 <1$. Then for $R>0$ large enough, we can find sequences of symbols 
 \[
 a^\pm_j \in S(-j,-\infty), \quad \emph{supp}(a^\pm_j) \subset \Ga^\pm(R,J_1,\tau_1),
 \]
such that for all 
 \[
 \chi^\pm \in S(0,-\infty), \quad \emph{supp}(\chi^\pm) \subset \Ga^\pm(R^4,J_4,\tau_4),
 \]
there exist sequences of symbols 
 \[
 b^\pm_k \in S(-k,-\infty), \quad \emph{supp}(b^\pm_k) \subset \Ga^\pm(R^3,J_3,\tau_3), 
 \]
such that for all $N \geq 1$, for all $h \in (0,1]$ and all $\pm t \geq 0$,
 \begin{align}
 e^{-ith^{-1}\psi(h^2P)} Op^h(\chi^\pm) = J^\pm_h(a^\pm(h)) e^{-ith^{-1}(h\Lambda)^\sigma} J^\pm_h(b^\pm(h))^\star + R^\pm_N(t,h), \nonumber
 \end{align}
where the phase functions $S^\pm_{R}:=S^\pm_{1,R}$ are as in \emph{Proposition} $\ref{prop hamilton-jacobi equation}$ and the remainder terms 
 \[
  R^\pm_N(t,h) = R^\pm_1(N,t,h)+R^\pm_2(N,t,h)+R^\pm_3(N,t,h)+R^\pm_4(N,t,h),
 \]
with
 \begin{align}
R^\pm_1(N,t,h) &= -h^{N-1} e^{-ith^{-1}\psi(h^2P)} Op^h (\tilde{r}^\pm_N(h)), \nonumber  \\
R^\pm_2(N,t,h) &= -i h^{N-1} \int_{0}^{t} e^{-i(t-s)h^{-1}\psi(h^2P)} R_N(h) J^\pm_h(a^\pm(h)) e^{-ish^{-1}(h\Lambda)^\sigma} J^\pm_h(b^\pm(h))^\star ds, \nonumber \\
R^\pm_3(N,t,h) &= -i h^{N-1} \int_{0}^{t} e^{-i(t-s)h^{-1}\psi(h^2P)} J^\pm_h(r^\pm_N(h)) e^{-ish^{-1}(h\Lambda)^\sigma} J^\pm_h(b^\pm(h))^\star ds, \nonumber  \\
R^\pm_4(N,t,h)&=  -i h^{-1} \int_{0}^{t} e^{-i(t-s)h^{-1}\psi(h^2P)} J^\pm_h(\check{a}^\pm(h)) e^{-ish^{-1}(h\Lambda)^\sigma} J^\pm_h(b^\pm(h))^\star ds. \nonumber
 \end{align}
Here $(\tilde{r}^\pm_N(h))_{h\in (0,1]}, (r^\pm_N(h))_{h\in (0,1]}$ are bounded in $S(-N,-\infty)$, $R_N(h)$ is as in $(\ref{parametrix psi high freq})$, $(\check{a}^\pm(h))_{h\in (0,1]}$ are bounded in $S(0,-\infty)$ and are finite sums depending on $N$ of the form 
 \begin{align}
 \check{a}^\pm(h)=\sum_{|\alpha| \geq 1} \check{a}^\pm_\alpha(h) \partial^\alpha_x \chi^\pm_{1\rightarrow 2}, \label{a check +-}
 \end{align}
where $(\check{a}^\pm_\alpha(h))_{h\in (0,1]}$ are bounded in $S(0,-\infty)$ and $\chi^\pm_{1 \rightarrow 2}$ are given in $(\ref{define cutoff})$.
\end{theorem}
We now give the main steps for the construction of the Isozaki-Kitada parametrix at low frequency. For simplicity, we omit the $\pm$ sign. Let us start with the following Duhamel formula
\[
e^{-it\ep\psi(\ep^{-2}P)} \Jc_\ep(a_\ep) =  \Jc_\ep(a_\ep) e^{-it\ep\Lambda^\sigma} \nonumber - i \ep \int_{0}^{t} e^{-i(t-s)\ep\psi(\ep^{-2}P)} \Big(\psi(\ep^{-2}P)\Jc_\ep(a_\ep) - \Jc_\ep(a_\ep) \Lambda^\sigma \Big) e^{-is\ep\Lambda^\sigma} ds.
\]
Thanks to the support of $a_\ep$, we can write
\[
\psi(\ep^{-2}P)\Jc_\ep(a_\ep) = \psi(\ep^{-2}P) \zeta_1(\ep x) \Jc_\ep (a_\ep),
\] 
where $\zeta_1 \in C^\infty(\R^d)$ is supported outside $B(0,1)$ and satisfies $\zeta_1(x)=1$ for $|x|>R$. Using the parametrix of $\psi(\ep^{-2}P) \zeta_1(\ep x)$ given in Proposition $\ref{prop parametrix phi low freq}$ (by taking the adjoint), we have
\[
\psi(\ep^{-2}P) \zeta_1(\ep x)= \sum_{k=0}^{N-1}\tilde{\zeta_1}(\ep x) Op_\ep(q_{\ep,k}) \zeta_1(\ep x) + R_N(\ep),
\]
where $q_{\ep,0}(x,\xi)=\psi(p_\ep(x,\xi))=\tilde{f}(p_\ep(x,\xi)) \sqrt{p_\ep(x,\xi)}^\sigma$, $\text{supp}(q_{\ep,k}) \subset \text{supp}(\psi \circ p_\ep)$ and $(R_N(\ep))_{\ep \in (0,1]}$ satisfies $(\ref{estimate remainder low freq})$. Here $\tilde{\zeta_1} \in C^\infty(\R^d)$ is supported outside $B(0,1)$ and $\tilde{\zeta_1} =1$ near $\text{supp}(\zeta_1)$. We want to find $a_\ep = \sum_{j=0}^{N-1} a_{\ep,j}$ so that the term $\psi(\ep^{-2}P)\Jc_\ep(a_\ep) - \Jc_\ep(a_\ep) \Lambda^\sigma$ has a small contribution. By the choice of cutoff functions and the action of pseudo-differential operators on Fourier integral operators given in Proposition $\ref{prop action PDO on FIO}$ with $h=1$, we have
\begin{align}
\psi (\ep^{-2}P) \Jc_\ep(a_\ep) - \Jc_\ep(a_\ep) \Lambda^\sigma &= \sum_{r=0}^{N-1} \left( \sum_{k+j+l=r}\Jc_\ep((q_{\ep,k} \triangleleft a_{\ep,j})_l) - \Jc_\ep(a_{\ep,r} |\xi|^\sigma) \right) \nonumber \\ 
& \mathrel{\phantom{=}} + R_N(\ep) \Jc_\ep(a_\ep) + \Jc_\ep(r_N(\ep)), \label{intertwin approximation}
\end{align}
where $(r_N(\ep))_{\ep \in (0,1]}$ is bounded in $S(-N,-\infty)$. This implies that we need to find $(a_{\ep,j})_{\ep \in (0,1]}$ bounded  in $S(-j,-\infty)$ supported in $\Ga(R,J_1, \tau_1)$ such that
\[
\sum_{k+j+l=r}(q_{\ep,k}\triangleleft a_{\ep,j})_l - a_{\ep,r} |\xi|^\sigma =0, \quad r=0,...,N-1.
\]
By noting that if $p_\ep(x,\xi) \in \text{supp}(f)$, then $q_{\ep, 0}(x,\xi)= \sqrt{p_\ep(x,\xi)}^\sigma$. This leads to the following Hamilton-Jacobi and transport equations,
\begin{align}
p_\ep(x,\nabla_x S_{\ep,R}(x,\xi))  &= |\xi|^2, \label{Hamilton-Jacobi equation low freq} \\
(q_{\ep,0}\triangleleft a_{\ep,0})_1 + (q_{\ep,1} \triangleleft a_{\ep,0})_0 &= 0 \label{transport_0 low freq} \\
(q_{\ep,0} \triangleleft a_{\ep,r})_1 + (q_{\ep,1} \triangleleft a_{\ep,r})_0 &= - \sum_{k+j+l=r+1 \atop j \leq r-1} (q_{\ep,k} \triangleleft a_{\ep,j})_l, \quad r=1,...,N-1. \label{transport_r low freq}
\end{align}
We can solve $(\ref{Hamilton-Jacobi equation low freq})$ on $\Ga^\pm(R,J_1,\tau_1)$ using Proposition $\ref{prop hamilton-jacobi equation}$. We then solve $(\ref{transport_0 low freq}), (\ref{transport_r low freq})$ on $\Ga^\pm(R,J_1,\tau_1)$ and extend solutions globally on $\R^{2d}$. We obtain
\[
\psi(\ep^{-2}P) \Jc_\ep(a_\ep) - \Jc_\ep(a_\ep) \Lambda^\sigma = R_N(\ep) \Jc_\ep(a_\ep) + \Jc_\ep(r_N(\ep)) + \Jc_\ep(\check{a}(\ep)),
\]
where $(\check{a}(\ep))_{\ep \in (0,1]}$ is bounded in $S(0,-\infty)$ and is a finite sum depending on $N$ of the form 
\[
\check{a}(\ep) = \sum_{|\alpha| \geq 1} \check{a}_\alpha(\ep) \partial^\alpha_x \chi_{1\rightarrow 2},
\]
with $(\check{a}_\alpha(\ep))_{\ep \in (0,1]}$ bounded in $S(0,-\infty)$ and $\chi_{1\rightarrow 2}$ as in $(\ref{define cutoff})$. \newline
\indent Next, we can find bounded families of symbols $b_{\ep,k} \in S(-k,-\infty)$ for $k=0,...,N-1$ supported in $\Ga(R^3,J_3,\tau_3)$ such that
\[
\Jc_\ep(a_\ep)\Jc_\ep(b_\ep)^\star = Op_\ep(\chi_\ep)\zeta(\ep x) + Op_\ep(\tilde{r}_N(\ep)) \zeta(\ep x),
\]
where $b_\ep =\sum_{k=0}^{N-1} b_{\ep,k}$ and $(\tilde{r}_N(\ep))_{\ep \in (0,1]}$ is bounded in $S(-N,-\infty)$. This is possible by writing for $R$ large enough $\Jc_\ep(b_\ep)=\zeta(\ep x) \Jc_\ep(b_\ep)$ and taking the adjoint. We have the following Isozaki-Kitata parametrix for the fractional Schr\"odinger equation at low frequency.
\begin{theorem} \label{theorem Isozaki-Kitada parametrix low freq}
Let $\sigma \in (0,\infty)$, $\zeta \in C^\infty(\R^d)$ be supported outside $B(0,1)$ and equal to $1$ near infinity. Fix $J_4 \Subset (0,+\infty)$ open interval containing $\emph{supp}(f)$ and $-1 < \tau_4 <1$. Choose arbitrary open intervals $J_1, J_2, J_3$ such that $J_4 \Subset J_3 \Subset J_2 \Subset J_1 \Subset (0,+\infty)$ and arbitrary $\tau_1, \tau_2, \tau_3$ such that $-1 <\tau_1 <\tau_2 <\tau_3 <\tau_4 <1$. Then for $R>0$ large enough, we can find bounded families of symbols
\[
(a^\pm_{\ep,j})_{\ep \in (0,1]} \in S(-j, -\infty), \quad \emph{supp}(a^\pm_{\ep,j}) \subset \Ga^\pm(R,J_1,\tau_1),
\]
such that for all
\[
(\chi^\pm_\ep)_{\ep \in (0,1]} \in S(0,-\infty), \quad \emph{supp}(\chi^\pm_\ep) \subset \Ga^\pm(R^4, J_4, \tau_4),
\]
there exists families of symbols
\[
(b^\pm_{\ep,k})_{\ep \in (0,1]} \in S(-k, -\infty), \quad \emph{supp}(b^\pm_{\ep,k}) \subset \Ga^\pm(R^3, J_3, \tau_3),
\]
such that for all $N\geq 1$, for all $\ep \in (0,1]$ and all $\pm t\geq 0$,
\begin{align}
e^{-it\ep \psi(\ep^{-2}P)} Op_\ep (\chi^\pm_\ep) \zeta(\ep x) = \Jc^\pm_\ep(a^\pm_{\ep}) e^{-it\ep\Lambda^\sigma} \Jc^\pm_\ep(b^\pm_\ep)^\star + \Rc^\pm_N(t,\ep), \nonumber
\end{align}
where the phase functions $S^\pm_{\ep,R}$ are given in \emph{Proposition} $\ref{prop hamilton-jacobi equation}$ and the remainder terms
\[
 \Rc^\pm_N(t,\ep) = \Rc^\pm_1(N,t,\ep)+\Rc^\pm_2(N,t,\ep)+\Rc^\pm_3(N,t,\ep)+\Rc^\pm_4(N,t,\ep),
\]
with
 \begin{align}
\Rc^\pm_1(N,t,\ep) &= -e^{-it\ep \psi(\ep^{-2}P)} Op_\ep (\tilde{r}^\pm_N(\ep)) \zeta(\ep x), \nonumber  \\
\Rc^\pm_2(N,t,\ep) &= -i\ep \int_{0}^{t} e^{-i(t-s)\ep \psi(\ep^{-2}P)} R_N(\ep) \Jc^\pm_\ep(a^\pm_\ep) e^{-is\ep\Lambda^\sigma} \Jc^\pm_\ep(b^\pm_\ep)^\star ds, \nonumber \\
\Rc^\pm_3(N,t,\ep) &= -i \ep  \int_{0}^{t} e^{-i(t-s)\ep \psi(\ep^{-2}P)} \Jc^\pm_\ep(r^\pm_N(\ep)) e^{-is\ep\Lambda^\sigma} \Jc^\pm_\ep(b^\pm_\ep)^\star ds, \nonumber  \\
\Rc^\pm_4(N,t,\ep) &= -i \ep \int_{0}^{t} e^{-i(t-s)\ep \psi(\ep^{-2}P)} \Jc^\pm_\ep(\check{a}^\pm(\ep)) e^{-is\ep\Lambda^\sigma} \Jc^\pm_\ep(b^\pm_\ep)^\star ds. \nonumber
 \end{align}
Here $(\tilde{r}^\pm_N(\ep))_{\ep\in (0,1]}, (r^\pm_N(\ep))_{\ep\in (0,1]}$ are bounded in $S(-N,-\infty)$, $(R_N(\ep))_{\ep \in (0,1]}$ is given in \emph{Proposition} $\ref{prop parametrix phi low freq}$, $(\check{a}^\pm(\ep))_{\ep\in (0,1]}$ are bounded in $S(0,-\infty)$ and are finite sums depending on $N$ of the form 
\[
\check{a}^\pm(\ep) = \sum_{|\alpha| \geq 1} \check{a}^\pm_\alpha(\ep) \partial^\alpha_x \chi^\pm_{1\rightarrow 2}, \label{define a check +- low freq}
\]
where $(\check{a}^\pm_\alpha(\ep))_{\ep \in (0,1]}$ are bounded in $S(0,-\infty)$ and $\chi^\pm_{1\rightarrow 2}$ are as in $(\ref{define cutoff})$.
\end{theorem}
We have the following dispersive estimates for the main terms of the Isozaki-Kitada parametrix both at high and low frequencies.
\begin{prop} \label{prop estimate main terms IK}
Let $\sigma \in (0,\infty)\backslash \{1\}$, $S^\pm_{\ep, R}$ be as in \emph{Proposition} $\ref{prop hamilton-jacobi equation}$ and $(a^\pm_\ep)_{\ep \in (0,1]}, (b^\pm_\ep)_{\ep \in (0,1]}$ be bounded in $S(0,-\infty)$ compactly supported in $\xi$ away from zero.
\begin{itemize}
\item[1.] Then for $R>0$ large enough, there exists $C>0$ such that for all $t \in \R$ and all $h\in (0,1]$,
\begin{align}
\| J_h^\pm(a^\pm)e^{-ith^{-1}(h\Lambda)^\sigma} J_h^\pm(b^\pm)^\star  \|_{\Lc(L^1,L^\infty)} \leq C h^{-d}(1+|t|h^{-1})^{-d/2}, \label{dispersie main term IK high freq}
\end{align} 
where $a^\pm:= a^\pm_{\ep=1}, b^\pm:= b^\pm_{\ep=1}$.
\item[2.] Then for $R>0$ large enough, there exists $C>0$ such that for all $t \in \R$ and all $\ep \in (0,1]$,
\begin{align}
\|\Jc^\pm_\ep(a^\pm_\ep) e^{-it\ep \Lambda^\sigma} \Jc^\pm_\ep(b^\pm_\ep)^\star\|_{\Lc(L^1, L^\infty)} \leq C \ep^d (1+\ep |t|)^{-d/2}. \label{dispersive main term IK low freq} 
\end{align} 
\end{itemize}
\end{prop}
\begin{proof}
1. For simplicity, we drop the superscript $\pm$. The kernel of $J_h(a)e^{-ith^{-1}(h\Lambda)^\sigma} J_h(b)^\star$ reads
\[
K_h(t,x,y)=(2\pi h)^{-d} \int_{\R^d} e^{ih^{-1}\left(S_{R}(x,\xi) -S_{R}(y,\xi)-t|\xi|^\sigma \right)} a(x,\xi)\overline{b(y,\xi)}d\xi.
\]
The estimates $(\ref{dispersie main term IK high freq})$ are in turn equivalent to 
\begin{align}
|K_h(t,x,y)| \leq  Ch^{-d}(1+|t|h^{-1})^{-d/2}, \label{dispersive estimates main term IK high freq equivalent}
\end{align}
for all $t \in \R, h \in (0,1]$ and $x,y \in \R^d$. We only consider $t\geq 0$, the case $t\leq 0$ is similar. Let us denote the compact support of the amplitude by $\Kc$. Since $a,b$ are bounded uniformly in $x,y \in \R^d$, we have
\[
|K_h(t,x,y)| \leq C h^{-d},
\]
for all $t \in \R$ and all $x, y \in \R^d$. If $0 \leq t \leq h$ or $1+t h^{-1} \leq 2$, then
\begin{align}
|K_h(t,x,y)| \leq Ch^{-d} \leq C h^{-d}(1+th^{-1})^{-d/2}. \nonumber
\end{align}
So, we can assume that $t \geq h$ or $(1+th^{-1}) \leq 2th^{-1}$ and denote the phase function
\[
\Phi(R,t,x,y,\xi)= (S_{R}(x,\xi)-S_{R}(y,\xi))/{t}-|\xi|^\sigma,
\]
and parameter $\lambda= t h^{-1} \geq 1$. We can rewrite 
\[
\Phi(R,t,x,y,\xi)= \scal{(x-y)/t, \eta(R,x,y,\xi)} -|\xi|^\sigma,
\]
where 
\[
\eta(R,x,y,\xi)=\int_{0}^{1} \nabla_x S_R(y+\lambda(x-y),\xi) d\lambda.
\]
Using the properties of the phase functions $S_{R}$ given in $(\ref{estimate on phase})$, we have that
\[
\eta(R,x,y,\xi)= \xi + Q(R,x,y,\xi),
\]
where $Q(R,x,y,\xi)$ is a vector in $\R^d$ satisfying for $R >0$ large enough,
\begin{align}
|\partial^\beta_\xi Q(R,x,y,\xi)| \leq C_\beta R^{-\rho}, \label{estimate of Q}
\end{align}
for all $x,y \in \R^d$ and $\xi\in \Kc$. We have
\[
\nabla_\xi \Phi(R,t,x,y,\xi) = \frac{x-y}{t}\cdot(\text{Id}_{\R^d}+\nabla_\xi Q(R,x,y,\xi))-\sigma \xi |\xi|^{\sigma-2}.
\]
\indent If $|(x-y)/{t}| \geq C$ for some constant $C>0$ large enough then for $R>0$ large enough, there exists $C_1>0$,
\[
|\nabla_\xi \Phi (R,t,x,y,\xi)| \geq \frac{1}{2}\Big| \frac{x-y}{t}\Big| \geq C_1.
\]
Thus the phase is non-stationary. By using integration by parts with respect to $\xi$ together with the fact
\begin{align}
|\partial^\beta_\xi \Phi(R,t,x,y,\xi)| \leq C_\beta \Big| \frac{x-y}{t}\Big|, \quad |\beta|\geq 2, \nonumber
\end{align}
we have that for all $N\geq 1$,
\[
|K_h(t,x,y)| \leq C h^{-d} (th^{-1})^{-N} \leq Ch^{-d} (1+th^{-1})^{-d/2},
\]
provided $N$ is taken bigger than $d/2$. The same result still holds for $|(x-y)/t| \leq c$ for some $c>0$ small enough. \newline
\indent Therefore, we can assume that $c\leq |{x-y}/{t}| \leq C$. In this case, we write
\[
\nabla^2_\xi \Phi(R,t,x,y,\xi)= \frac{x-y}{t} \cdot \nabla^2_\xi Q(R,x,y,\xi) - \sigma |\eta|^{\sigma-2} \Big(\text{Id}_{\R^d} +(\sigma-2)\frac{\eta\cdot \eta^t}{|\eta|^2}\Big).
\]
Using the fact that $\sigma \in (0,\infty)\backslash \{1\}$ and
\[
\Big|\det \sigma |\eta|^{\sigma-2} \Big(\text{Id}_{\R^d} +(\sigma-2)\frac{\eta\cdot \eta^t}{|\eta|^2}\Big) \Big| = \sigma^{d}|\sigma-1\|\eta|^{(\sigma-2)d} \geq C
\]
and $(\ref{estimate of Q})$, we see that for $R>0$ large enough, the map $\xi \mapsto \nabla_\xi \Phi(R,t,x,y,\xi)$ is a local diffeomorphism from a neighborhood of $\mathcal{K}$ to its range. Moreover, for all $\beta\in \N^d$ satisfying $|\beta| \geq 1$, we have $|\partial^\beta_\xi \Phi(R, t,x,y,\xi)| \leq C_\beta$. The stationary phase theorem then implies that for $R>0$ large enough, all $t \geq h$ and all $x,y \in \R^d$ satisfying $c \leq |(x-y)/t| \leq C$,
\[
|K_h(t,x,y)| \leq C h^{-d} \lambda^{-d/2} \leq C h^{-d}(1+th^{-1})^{-d/2}.
\]
This gives $(\ref{dispersive estimates main term IK high freq equivalent})$. \newline
2. We are now in position to show $(\ref{dispersive main term IK low freq})$. As above, we drop the superscript $\pm$ for simplicity. We see that up to a conjugation by $D_\ep$, the kernel of $\Jc_\ep(a_\ep) e^{-it\ep \Lambda^\sigma} \Jc_\ep(b_\ep)^\star$ reads
\[
K_\ep(t,x,y)=(2\pi)^{-d} \int_{\R^d} e^{i(S_{\ep,R}(x,\xi)-t\ep|\xi|^\sigma-S_{\ep,R}(y,\xi))} a_\ep(x,\xi) \overline{b_\ep(y,\xi)} d\xi.
\]
The dispersive estimates $(\ref{dispersive main term IK low freq})$ follow from 
\begin{align}
|K_\ep(t,x,y)| \leq C (1+\ep|t|)^{-d/2}, \label{dispersive estimates main term IK low freq equivalent} 
\end{align}
for all $t\in \R$ uniformly in $x,y \in \R^d$, $\ep \in (0,1]$ and the fact that 
\[
\|D_\ep\|_{\Lc(L^\infty)} = \ep^{d/2}, \quad \|D_\ep^{-1}\|_{\Lc(L^1)} = \ep^{d/2}.
\]
The estimates $(\ref{dispersive estimates main term IK low freq equivalent})$ are proved by repeating the same line as above.
The proof is complete.
\end{proof}
\paragraph{Micro-local propagation estimates.}
In this paragraph, we will prove some propagation estimates which are useful for our purpose. To do this, we need the following result (see \cite[Lemma 4.1]{BTglobalstrichartz}).
\begin{lem} \label{lem sigma + -}
Let $\tau_+, \tau_- \in (-1,1)$.
\begin{itemize}
\item[1.] For all $x,y,\xi \in \R^d \backslash \{0\}$ satisfying $\pm {x\cdot \xi}/{|x\|\xi|} > \tau_\pm$ and $\pm t \geq 0$, we have
\begin{align} \label{scalar estimate 1}
\pm \frac{(x+t\xi)\cdot \xi}{|x+t\xi| |\xi|}> \tau_\pm \text{ and } |x+t\xi| \geq c_\pm(|x|+|t\xi|),
\end{align}
where $c_\pm = \sqrt{1+\tau_\pm}/\sqrt{2}$.
\item[2.] If $\tau_- +\tau_+>0$, then there exists $c=c(\tau_-,\tau_+)>0$ such that for all $x,y,\xi \in \R^d \backslash \{0\}$ satisfying $+{x\cdot \xi}/{|x\|\xi|}>\tau_+$ and $-{y\cdot \xi}/{|y\|\xi|}> \tau_-$, we have
\begin{align}
|x-y| \geq c(|x|+|y|). \label{scalar estimate 2}
\end{align}
\end{itemize}
\end{lem}
We start with the following estimates.
\begin{lem} \label{lem remainder 4}
Let $\sigma \in (0,\infty)$ and $\chi \in C^\infty_0(\R^d)$ satisfying $\chi (x)=1$ for $|x| \leq 1$. 
\begin{itemize}
\item[1.] Using the notations given in $\emph{Theorem}$ $\ref{theorem IK parametrix fractional schrodinger equation}$, if $R>0$ is large enough, then for all $m \geq 0$, there exists $C>0$ such that for all $\pm s \geq 0$ and all $h \in(0,1]$,
\begin{align}
\|\chi\left(x/R^2\right) J^\pm_h(\check{a}^\pm(h))e^{-ish^{-1}(h\Lambda)^\sigma} J^\pm_h(b^\pm(h))^\star \scal{x}^m \|_{\Lc(H^{-m},H^m)} \leq  Ch^m \scal{s}^{-m}. \label{estimate remainder R 4 one} 
\end{align}
Moreover, 
\begin{align}
\|\scal{x}^m (1-\chi)\left(x/R^2\right) J^\pm_h(\check{a}^\pm(h))e^{-ish^{-1}(h\Lambda)^\sigma}  J^\pm_h(b^\pm(h))^\star \scal{x}^m\|_{\Lc(H^{-m},H^m)} \leq  Ch^m \scal{s}^{-m}.  \label{estimate remainder R 4 two}
\end{align}
In particular
\begin{align}
\|\scal{x}^m J^\pm_h(\check{a}^\pm(h))e^{-ish^{-1}(h\Lambda)^\sigma}  J^\pm_h(b^\pm(h))^\star \scal{x}^m\|_{\Lc(H^{-m},H^m)} \leq  Ch^m \scal{s}^{-m}. \label{estimate remainder R 4 combined}
\end{align}
\item[2.] Using the notations given in $\emph{Theorem}$ $\ref{theorem Isozaki-Kitada parametrix low freq}$, if $R>0$ is large enough, then for all $m \geq 0$, there exists $C>0$ such that for all $\pm s \geq 0$ and all $\ep \in(0,1]$,
\begin{align}
\|\chi(\ep x/R^2) \Jc^\pm_\ep(\check{a}^\pm(\ep))e^{-is\ep\Lambda^\sigma} \Jc^\pm_\ep(b^\pm_\ep)^\star \scal{\ep x}^m \|_{\Lc(L^2)} \leq  C \scal{\ep s}^{-m}. \label{estimate remainder R 4 one low freq} 
\end{align}
Moreover,
\begin{align}
\|\scal{\ep x}^m (1-\chi)(\ep x/R^2) \Jc^\pm_\ep(\check{a}^\pm(\ep))e^{-is\ep\Lambda^\sigma} \Jc^\pm_\ep(b^\pm_\ep)^\star \scal{\ep x}^m\|_{\Lc(L^2)} \leq  C \scal{\ep s}^{-m}. \label{estimate remainder R 4 two low freq}
\end{align}
In particular
\begin{align}
\|\scal{\ep x}^m \Jc^\pm_\ep(\check{a}^\pm(\ep))e^{-is\ep\Lambda^\sigma} \Jc^\pm_\ep(b^\pm_\ep)^\star \scal{\ep x}^m\|_{\Lc(L^2)} \leq  C \scal{\ep s}^{-m}. \label{estimate remainder R 4 final low freq}
\end{align}
\end{itemize}
\end{lem}
\begin{proof} 1. We firstly consider the high frequency case. The proof in this case is essentially given in \cite{BTglobalstrichartz}. For reader's convenience, we will give a sketch of the proof. 
The kernel of the operator in the left hand side of $(\ref{estimate remainder R 4 one})$ reads
\[
K^\pm_h(s,x,y)=(2\pi h)^{-d}\chi(x/R^2) \int_{\R^d} e^{ih^{-1}\Phi^\pm(R,s,x,y,\xi)} \check{a}^\pm(h,x,\xi) \overline{b^\pm(h,y,\xi)} d\xi \scal{y}^m,
\]
where the phase $\Phi^\pm(R,s,x,y,\xi)= S^\pm_R(x,\xi)-s|\xi|^\sigma -S^\pm_R(y,\xi)$. Using $(\ref{estimate on phase})$, we have 
\[
|\nabla_\xi \Phi^\pm(R,s,x,y,\xi)| =|x-\sigma s \xi |\xi|^{\sigma-2}-y +O(1)|\geq |\sigma s \xi|\xi|^{\sigma-2} +y| - |x| +O(1),
\]
where $|x| \leq C R^2$ and $(y,\xi)\in \Ga^\pm(R^3,J_3,\tau_3)$. We then apply $(\ref{scalar estimate 1})$ with $\pm {y\cdot \xi}/{|y\|\xi|} > \tau_3$ and $\pm t=\pm \sigma s|\xi|^{\sigma-2} \geq 0$ to get
\begin{align}
| \sigma s \xi|\xi|^{\sigma-2} +y | \geq C(|s|+|y|), \label{peetre inequality}
\end{align}
for all $\pm s \geq 0$. We next use $|y| > R^3$ to control $|x| \lesssim R^2$ and obtain
\[
|\nabla_\xi \Phi^\pm(R,s,x,y,\xi)| \geq C(1+|s|+|x|+|y|),
\]
for all $\pm s \geq 0$. By integrations by part with respect to $\xi$ with remark that higher derivatives of $\partial_\xi \Phi^\pm$ are controlled by $|\nabla_\xi\Phi^\pm|$, we get for all $N\geq 0$,
\[
\left| \chi(x/R^2) \int_{\R^d} e^{ih^{-1}\Phi^\pm(R,s,x,y,\xi)} \check{a}^\pm(h,x,\xi) \overline{b^\pm(h,y,\xi)} d\xi \right| \leq Ch^{N} (1+|s|+|x|+|y|)^{-N}.
\]
By choosing $N$ large enough, we can dominate $\scal{y}^m$ and get
\[
|K^\pm_h(s,x,y)| \leq C h^{N}(1+|s|+|x|+|y|)^{-N},
\]
for all $N$ large enough, therefore for all $N\geq 0$. We do the same for higher derivatives $\partial^\alpha_x \partial^\beta_y K_h(s,x,y)$ and the result follows. The kernel of the operator in the left hand side of $(\ref{estimate remainder R 4 two})$ reads
\[
K^\pm_h(s,x,y)= (2\pi h)^{-d} \scal{x}^m (1-\chi)(x/R^2) \int_{\R^d}e^{ih^{-1}\Phi^\pm(R,s,x,y,\xi)}\check{a}^\pm(h,x,\xi) \overline{b^\pm(h,y,\xi)} d\xi \scal{y}^m.
\]
We use the form of $\check{a}^\pm(h)$ given in $(\ref{a check +-})$. In the case derivatives fall on $\kappa(x/R^2)$, we have that $|x|\leq C R^2$ and we can proceed as above. Note that we have from $(\ref{scalar estimate 1})$ with $\pm {y\cdot\xi}/{|y\|\xi|} > \sigma_3$ and $\pm t=\pm \sigma s|\xi|^{\sigma-2} \geq 0$ that
\[
\pm \frac{(y+ \sigma s \xi|\xi|^{\sigma-2})\xi}{| y+\sigma s \xi|\xi|^{\sigma-2}\|\xi|} > \sigma_3 \text{ and } |y+\sigma s \xi|\xi|^{\sigma-2} | \geq c_\pm (|s|+|y|).
\]
In the case derivatives fall on $\theta_{1\rightarrow 2}$, we have 
\[
\tau_1+\varepsilon \leq \pm \frac{x\cdot \xi}{|x\|\xi|} \leq \tau_2-\varepsilon \text{ or } \mp \frac{x\cdot \xi}{|x\|\xi|} \geq -\tau_2 +\varepsilon > -\tau_2+\varepsilon/2.
\]
By choosing $\varepsilon>0$ small enough such that $\tau_3 - \tau_2 +\varepsilon/2>0$,  $(\ref{scalar estimate 2})$ gives
\[
| y+\sigma s\xi|\xi|^{\sigma-2} -x| \geq c\left( |y+\sigma s\xi|\xi|^{\sigma-2}| +|x| \right) \geq C (|s|+|x|+|y|).
\]
Thus $|\nabla_\xi \Phi^\pm| \geq C(1+|s|+|x|+|y|)$ for $\pm s \geq 0$ and $(\ref{estimate remainder R 4 two})$ follows as above. \newline
2. The proof for the low frequency case is the same as above up to the conjugation by the unitary map $D_\ep$ in $L^2(\R^d)$. For instance, the kernel of the operator in the left hand side of $(\ref{estimate remainder R 4 one low freq})$ reads
\[
K^\pm_\ep(s,x,y)=(2\pi)^{-d} \chi(x/R^2) \int_{\R^d} e^{i\Phi^\pm_\ep(R,s,x,y,\xi)} \check{a}^\pm(\ep,x,\xi) \overline{b^\pm_\ep(y,\xi)} d\xi \scal{y}^m,
\]
where the phase $\Phi^\pm_\ep(R,s,x,y,\xi)= S^\pm_{\ep,R}(x,\xi)-\ep s|\xi|^\sigma -S^\pm_{\ep,R}(y,\xi)$. 
\end{proof}
\begin{lem} \label{lem remainder 3}
Let $\sigma \in (0,\infty)$. 
\begin{itemize}
\item[1.] Under the notations of $\emph{Theorem}$ $\ref{theorem IK parametrix fractional schrodinger equation}$, for all $m \geq 0$ and all $N$ large enough, there exists $C>0$ such that for all $\pm s \geq 0$ and all $h \in (0,1]$,
\begin{align}
\|\scal{x}^{N/8} J^\pm_h(r^\pm_N(h)) e^{-ish^{-1}(h\Lambda)^\sigma} J^\pm_h(b^\pm(h))^\star \scal{x}^{N/4} \|_{\Lc(H^{-m},H^m)} \leq Ch^{-d-2m} \scal{s}^{-N/4}. \label{estimate remainder R 3}
\end{align}
\item[2.] Under the notations of $\emph{Theorem}$ $\ref{theorem Isozaki-Kitada parametrix low freq}$, for all $N$ large enough, there exists $C>0$ such that for all $\pm s \geq0$ and all $\ep \in (0,1]$,
\begin{align}
\|\scal{\ep x}^{N/8} \Jc^\pm_\ep(r^\pm_N(\ep)) e^{-i\ep s\Lambda^\sigma} \Jc^\pm_\ep(b^\pm_\ep)^\star \scal{\ep x}^{N/4} \|_{\Lc(L^{2})} \leq C \scal{\ep s}^{-N/4}. \label{estimate remainder R 3 low freq}
\end{align}
\end{itemize}
\end{lem}
\begin{proof}
We only give the proof for the high frequency case, the low frequency one is similar. The kernel of the operator in the left hand side of $(\ref{estimate remainder R 3})$ reads
\[
K^\pm_h(s,x,y)= (2\pi h)^{-d} \int_{\R^d} e^{ih^{-1}\Phi^\pm(R,s,x,y,\xi)} A^\pm(h,x,y,\xi) d\xi,
\]
where the amplitude $A^\pm(h,x,y,\xi)= \scal{x}^{N/8} r^\pm_N(h,x,\xi)\overline{b^\pm(h,y,\xi)} \scal{y}^{N/4}$ and is compactly supported in $\xi$. We have from Proposition $\ref{prop hamilton-jacobi equation}$ and $(\ref{peetre inequality})$ that $\nabla_\xi \Phi^\pm(R,s,x,y,\xi)= x-\sigma s \xi |\xi|^{\sigma -2} -y +O(1)$ and $|\sigma s \xi |\xi|^{\sigma -2} +y| \geq C(|s|+|y|)$ for all $\pm s \geq 0$. By Peetre's inequality, we see that
\[
\scal{\nabla_\xi\Phi^\pm}^{-1} \leq \scal{x}\scal{y+\sigma s\xi|\xi|^{\sigma-2}}^{-1} \leq C \scal{x}(\scal{y}+\scal{s})^{-1}.
\]
We next write
\[
1=\chi(\nabla_\xi\Phi^\pm) +(1-\chi)(\nabla_\xi\Phi^\pm),
\]
where $\chi \in C^\infty_0(\R^d)$ with $\chi=1$ near 0. Then $K^\pm_h(s,x,y)$ is split into two terms. For the first term
\[
I_1=(2\pi h)^{-d} \int_{\R^d} e^{ih^{-1}\Phi^\pm(R,s,x,y,\xi)} \chi(\nabla_\xi\Phi^\pm) A^\pm(h,x,y,\xi) d\xi,
\]
by using the fact that
\begin{align}
|\chi(\nabla_\xi\Phi^\pm)| &\leq C \scal{\nabla_\xi \Phi^\pm}^{-3N/4} \leq C \scal{x}^{3N/4} (\scal{y}+\scal{s})^{-3N/4} \nonumber \\
&\leq C\scal{x}^{3N/4} \scal{y}^{-N/2}\scal{s}^{-N/4}, \label{negative power nabla phi}
\end{align}
and $A^\pm(h,x,y,\xi)=O(\scal{x}^{-7N/8} \scal{y}^{N/4})$, it is bounded by $Ch^{-d} \scal{x}^{-N/8} \scal{y}^{-N/4}\scal{s}^{-N/4}$. For the second term
\[
I_2=(2\pi h)^{-d} \int_{\R^d} e^{ih^{-1}\Phi^\pm(R,s,x,y,\xi)} (1-\chi)(\nabla_\xi\Phi^\pm) A^\pm(h,x,y,\xi) d\xi,
\] 
thanks to the support of $(1-\chi)$, we can integrate by parts with respect to $\Lc:= \frac{h \nabla_\xi \Phi^\pm}{i|\nabla_\xi \Phi^\pm|^2}\circ \nabla_\xi$ to get many negative powers of $|\nabla_\xi \Phi^\pm|$ as we wish and estimate as in $(\ref{negative power nabla phi})$. Combine two terms and Schur's lemma, we have $(\ref{estimate remainder R 3})$ for $m=0$. For $m\geq 1$, we can do the same with
$\partial^\alpha_x\partial^\beta_y K^\pm_h(s,x,y)$ with $|\alpha|\leq m, |\beta|\leq m$. This completes the proof.
\end{proof}
Combining Lemma $\ref{lem remainder 4}$ and Lemma $\ref{lem remainder 3}$, we have the following result.
\begin{prop}\label{prop remainder 2+3+4}
\begin{itemize}
\item[1.] Using the notations given in $\emph{Theorem}$ $\ref{theorem IK parametrix fractional schrodinger equation}$, for all $0 \leq m \leq d+1$ and all $N$ large enough, we can write for $k=2, 3, 4$,
\[
R^\pm_k(N,t,h) =  h^{N/2} \int_{0}^{t} e^{-i(t-s)h^{-1}\psi(h^2P)} \scal{x}^{-N/8} B^\pm_m(N,s,h) \scal{x}^{-N/4} ds, 
\]
with
\begin{align}
\|B^\pm_m(N,s,h)\|_{\Lc(H^{-m}, H^m)} \leq C \scal{s}^{-N/4},  \label{estimate on B}
\end{align}
for all $\pm s \geq 0$ and  $h \in (0,1]$.
\item[2.] Using the notations given in $\emph{Theorem}$ $\ref{theorem Isozaki-Kitada parametrix low freq}$ and for all $N$ large enough, we can write for $k=2, 3, 4$,
\[
\Rc^\pm_k(N,t,\ep) = \ep \int_{0}^{t} e^{-i(t-s)\ep \psi(\ep^{-2}P)} \scal{\ep x}^{-N/8} \mathcal{B}^\pm_N(s,\ep) \scal{\ep x}^{-N/4} ds, 
\]
with
\begin{align}
\|\mathcal{B}^\pm_N(s,\ep)\|_{\Lc(L^{2})} \leq C \scal{\ep s}^{-N/4},  \label{estimate on B low freq}
\end{align}
for all $\pm s \geq 0$ and all $\ep \in (0,1]$.
\end{itemize}
\end{prop}
\begin{proof}
The cases $k=3, 4$ follow immediately from Lemma $\ref{lem remainder 4}$ and Lemma $\ref{lem remainder 3}$. It remains to show the case $k=2$. Let us consider the high frequency case. We can write $R_N(h) E^\pm(h)$ as
\[
\scal{x}^{-N/8} \Big(\scal{x}^{N/8} R_N(h)\scal{x}^{7N/8}\Big) \Big(\scal{x}^{N/8}\scal{x}^{-N}E^\pm(h)\scal{x}^{N/4}\Big) \scal{x}^{-N/4},
\]
where $E^\pm(h):=J^\pm_h(a^\pm(h))e^{-ish^{-1}(h\Lambda)^\sigma} J^\pm_h(b^\pm(h))^\star$. The first bracket is bounded in $\Lc(L^2)$ using Proposition $\ref{prop parametrix phi}$. The second one is bounded in $\Lc(H^{-m},H^m)$ using Lemma $\ref{lem remainder 3}$ with the fact that $\scal{x}^{-N}J^\pm_h(a^\pm(h))= J^\pm_h(\tilde{\tilde{r}}^\pm_N(h))$ where $\tilde{\tilde{r}}^\pm_N(h)$ are bounded in $S(-N,-\infty)$. The low frequency case is similar using Proposition $\ref{prop parametrix phi low freq}$.
\end{proof}
Next, we have the following micro-local propagation estimates both at high and low frequencies.
\begin{prop} \label{prop microlocal propagation estimate}
Let $\sigma \in (0,\infty)$, $f \in C^\infty_0(\R \backslash \{0\})$, $J_4 \Subset (0,+\infty)$ be an open interval and $-1< \tau_4 <1$.
\begin{itemize}
\item[1.]  Consider $\R^d, d\geq 2$ equipped with a smooth metric $g$ satisfying $(\ref{assump elliptic}), (\ref{assump long range})$ and suppose that $(\ref{assump resolvent})$ is satisfied. Then for $R>0$ large enough and $\chi^\pm \in S(0,-\infty)$ supported in $\Ga^\pm(R^4,J_4,\tau_4)$, we have the following estimates. 
\begin{itemize}
\item[\text{i}.] For all $m\in \N$ and all integer $l$ large enough, there exists $C>0$ such that for all $\pm t \leq 0$ and all $h \in (0,1]$,
\begin{align}
\|Op^h(\chi^\pm)^\star e^{-ith^{-1}(h\Lambda_g)^\sigma} f(h^2P) \scal{x}^{-l}\|_{\Lc(L^2, H^m)} \leq Ch^{-m} \scal{t}^{-3l/4}. \label{propagation estimate 1}
\end{align}
\item[\text{ii}.] For all $m \in \N$, all $\chi \in C^\infty_0(\R^d)$ and all $l \geq 1$, there exists $C>0$ such that for all $\pm t \leq 0$ and all $h \in (0,1]$,
\begin{align}
\|Op^h(\chi^\pm)^\star e^{-it h^{-1}(h\Lambda_g)^\sigma} f(h^2P) \chi(x/R^2)\|_{\Lc(L^2, H^m)}  \leq Ch^l \scal{t}^{-l}. \label{propagation estimate 2}
\end{align} 
\item[\text{iii}.] For all $\tilde{\chi}^\mp \in S(0,-\infty)$ supported in $\Ga^\mp (R, J_1, \tilde{\tau}_1)$ with $ -\tau_4 < \tilde{\tau}_1 <1$ and $J_4 \Subset J_1$ and all $l \geq 1$, there exists $C>0$ such that for all $\pm t \leq 0$ and all $h \in (0,1]$,
\begin{align}
\|Op^h(\chi^\pm)^\star e^{-it h^{-1}(h\Lambda_g)^\sigma} f(h^2P) Op^h(\tilde{\chi}^\mp)\|_{\Lc(L^\infty)} \leq Ch^{l}\scal{t}^{-l}. \label{propagation estimate 3}
\end{align}
\end{itemize}
\item[2.] Consider $\R^d, d\geq 3$ equipped with a smooth metric $g$ satisfying $(\ref{assump elliptic}), (\ref{assump long range})$. Let $\zeta \in C^\infty(\R^d)$ be supported outside $B(0,1)$ and equal to $1$ near infinity. Then for $R>0$ large enough and all $(\chi^\pm_\ep)_{\ep \in (0,1]}$ bounded families in $S(0,-\infty)$ supported in $\Ga^\pm(R^4,J_4,\tau_4)$, we have the following estimates.
\begin{itemize}
\item[\text{i}.] For all integer $l$ large enough, there exists $C>0$ such that for all $\pm t \leq 0$ and all $\ep \in (0,1]$,
\begin{align}
\|\zeta(\ep x)Op_\ep(\chi^\pm_\ep)^\star e^{-it\ep(\ep^{-1}\Lambda_g)^\sigma} f(\ep^{-2}P) \scal{\ep x}^{-l}\|_{\Lc(L^2)} \leq C \scal{\ep t}^{-3l/4}. \label{propagation estimate 1 low freq}
\end{align}
\item[\text{ii}.] For all $\chi \in C^\infty_0(\R^d)$ and all $l \geq 1$, there exists $C>0$ such that for all $\pm t \leq 0$ and all $\ep \in (0,1]$,
\begin{align}
\|\zeta(\ep x)Op_\ep(\chi^\pm_\ep)^\star e^{-it\ep(\ep^{-1}\Lambda_g)^\sigma} f(\ep^{-2}P) \chi(\ep x/R^2)\|_{\Lc(L^2)}  \leq C \scal{\ep t}^{-l}. \label{propagation estimate 2 low freq}
\end{align} 
\item[\text{iii}.] For all $\tilde{\zeta} \in C^\infty(\R^d)$ supported outside $B(0,1)$ and equal to 1 near infinity and all $(\tilde{\chi}^\mp_\ep)_{\ep \in (0,1]}$ bounded families in $S(0,-\infty)$ supported in $\Ga^\mp (R, J_1, \tilde{\tau}_1)$ with $ -\tau_4 < \tilde{\tau}_1 <1$ and $J_4 \Subset J_1$ and all $l \geq 1$, there exists $C>0$ such that for all $\pm t \leq 0$ and all $\ep \in (0,1]$,
\begin{align}
\|\zeta(\ep x)Op_\ep(\chi^\pm_\ep)^\star e^{-it\ep(\ep^{-1}\Lambda_g)^\sigma} f(\ep^{-2}P) Op_\ep(\tilde{\chi}^\mp_\ep) \tilde{\zeta}(\ep x)\|_{\Lc(L^2)} \leq C\scal{\ep t}^{-l}. \label{propagation estimate 3 low freq}
\end{align}
\end{itemize}
\end{itemize}
\end{prop}
\begin{proof} We only give the proof for the low frequency case, the proof at high frequency is similar and essentially given in \cite[Proposition 4.5]{BTglobalstrichartz}.  \newline
\indent i. We only consider the case $\chi^+_\ep$ and $t \leq 0$, the case $\chi^-_\ep$ and $t \geq 0$ is similar. By taking the adjoint, $(\ref{propagation estimate 1 low freq})$ is equivalent to 
\begin{align}
\|\scal{\ep x}^{-l}f(\ep^{-2}P) e^{-it\ep(\ep^{-1}\Lambda_g)^\sigma} Op_\ep(\chi^+_\ep)\zeta(\ep x)\|_{\Lc(L^2(\R^d))} \leq C \scal{\ep t}^{-3l/4}, \quad t \geq 0, \label{reduce propagation estimate 1 low freq}
\end{align}
uniformly in $\ep \in (0,1]$. 
Thanks to the spectral localization, we can apply the Isozaki-Kitada parametrix given in Theorem $\ref{theorem Isozaki-Kitada parametrix low freq}$ and obtain 
\[
e^{-it\ep(\ep^{-1}\Lambda_g)^\sigma} Op_\ep(\chi^+_\ep)\zeta(\ep x)=\Jc^+_\ep(a^+_\ep) e^{-it\ep \Lambda^\sigma} \Jc^+_\ep(b^+_\ep)^\star + \Rc^+_N(t,\ep).
\]
The main term can be written as
\[
\scal{\ep x}^{-l}f(\ep^{-2}P) \scal{\ep x}^l \scal{\ep x}^{-n} \scal{\ep x}^{n-l} \Jc^+_\ep(a^+_\ep) e^{-it\ep\Lambda^\sigma} \Jc^+_\ep(b^+_\ep)^\star \scal{\ep x}^{n} \scal{\ep x}^{-n}.
\]
By using Corollary $\ref{coro rescaled speudo-differential}$, we have the terms $\scal{\ep x}^{-l}f(\ep^{-2}P) \scal{\ep x}^l$ and $\scal{\ep x}^{-n}$ are bounded in $\Lc(L^2)$. It suffices to show for $l$ large enough,
\[
\|\scal{\ep x}^{n-l} \Jc^+_\ep(a^+_\ep) e^{-it\ep\Lambda^\sigma} \Jc^+_\ep(b^+_\ep)^\star \scal{\ep x}^{n}\|_{\Lc(L^2)} \leq C\scal{\ep t}^{-3l/4}, \quad t \geq 0,
\]
uniformly in $\ep \in (0,1]$. This expected estimate follows by using the same process as in Lemma $\ref{lem remainder 3}$. We now study the remainders. \newline
\indent For $k=1$, we have 
\begin{align*}
\|\scal{\ep x}^{-l}f(\ep^{-2}P)\Rc_1^+(N,t,\ep)\|_{\Lc(L^2)} &= \|\scal{\ep x}^{-l}f(\ep^{-2}P) e^{-it\ep(\ep^{-1}\Lambda_g)^\sigma} Op_\ep(\tilde{r}^+_N(\ep)) \zeta(\ep x)\|_{\Lc(L^2)} \nonumber \\
&\leq C \scal{\ep t}^{1-l}. \nonumber
\end{align*}
Here we insert $\scal{\ep x}^{-l}\scal{\ep x}^l$ in the middle and use $(\ref{local energy decay fractional schrodinger low freq})$ and rescaled pseudo-differential calculus. \newline  
\indent For $k=2,3,4$, Item 2 of Proposition $\ref{prop remainder 2+3+4}$ yields 
\[
\scal{\ep x}^{-l}f(\ep^{-2}P)\Rc_k^+(N,t,\ep)= \ep\int_{0}^{t} \scal{\ep x}^{-l} f(\ep^{-2}P)e^{-i(t-s)\ep(\ep^{-1}\Lambda_g)^\sigma} \scal{\ep x}^{-N/8} \mathcal{B}_N(s,\ep) \scal{\ep x}^{-N/4} ds.
\]
Using again $(\ref{local energy decay fractional schrodinger low freq})$ and the fact that $\scal{\ep x}^{l-N/8}$ and $\scal{\ep x}^{-N/4}$ are of size $O_{\Lc(L^2)}(1)$ for $N$ large enough and $(\ref{estimate on B low freq})$, we obtain
\begin{align}
\|\scal{\ep x}^{-l}f(\ep^{-2}P)\Rc_k^+(N,t,\ep)\|_{\Lc(L^2)} \leq C \ep \int_{0}^{t} \scal{\ep(t-s)}^{1-l} \scal{\ep s}^{-N/4}ds \leq  C \scal{\ep t}^{1-l}. \nonumber
\end{align}
By choosing $l$ large enough such that $l-1 \geq 3l/4$, it shows $(\ref{reduce propagation estimate 1 low freq})$. \newline
\indent ii. We do the same for $(\ref{propagation estimate 2 low freq})$, it is equivalent to show
\begin{align}
\|\chi(\ep x/R^2) f(\ep^{-2}P)e^{-it\ep(\ep^{-1}\Lambda_g)^\sigma} Op_\ep(\chi^+_\ep) \zeta(\ep x)\|_{\Lc(L^2(\R^d))} \leq C \scal{\ep t}^{-l}, \quad t\geq 0,  \label{reduce propagation estimate 2 low freq}
\end{align}
uniformly in $\ep \in (0,1]$. We again use the Isozaki-Kitada parametrix. Let us firstly study remainder terms. We write the first remainder term $\chi(\ep x/R^2)f(\ep^{-2}P) \Rc^+_1(N,t,\ep)$ as 
\[
\chi(\ep x/R^2) \scal{\ep x}^{l} \scal{\ep x}^{-l}f(\ep^{-2}P) e^{-it\ep(\ep^{-1}\Lambda_g)^\sigma} \scal{\ep x}^{-l} \scal{\ep x}^l Op_\ep(\tilde{r}^+_N(\ep)) \zeta(\ep x).
\]
Using $(\ref{local energy decay fractional schrodinger low freq})$ and the fact that $\chi(\ep x/R^2) \scal{\ep x}^{l}$ and $\scal{\ep x}^l Op_\ep(\tilde{r}^+_N(\ep)) \zeta(\ep x)$ are bounded in $\Lc(L^2)$ due to the support property of $\chi$ and rescaled pseudo-differential calculus given as in Proposition $\ref{prop parametrix phi low freq}$, we get
\begin{align}
\|\chi(\ep x/R^2)f(\ep^{-2}P) \Rc^+_1(N,t,\ep)\|_{\Lc(L^2)} \leq C \scal{\ep t}^{1-l}. \nonumber
\end{align}
\indent For $k=2,3,4$, we have
\begin{align}
\|\chi(\ep x/R^2)f(\ep^{-2}P) \Rc^+_k(N,t,\ep)\|_{\Lc(L^2)} \leq C \ep \int_{0}^{t} \scal{\ep (t-s)}^{1-l} \scal{\ep s}^{-N/4}ds \leq C \scal{\ep t}^{1-l}. \nonumber
\end{align}
For the main term, we can write 
\[
\chi(\ep x/R^2) \scal{\ep x}^{l} \scal{\ep x}^{-l}f(\ep^{-2}P)\scal{\ep x}^l \scal{\ep x}^{-n} \scal{\ep x}^{n-l} \Jc^+_\ep(a^+_\ep) e^{-it\ep\Lambda^\sigma} \Jc^+_\ep(b^+_\ep)^\star \scal{\ep x}^{n} \scal{\ep x}^{-n}.
\]
Thanks to the $L^2$-boundedness of $\chi(\ep x/R^2) \scal{\ep x}^{l}$, $\scal{\ep x}^{-l}f(\ep^{-2}P)\scal{\ep x}^l$, $\scal{\ep x}^{-n}$, it suffices to prove
\[
\|\scal{\ep x}^{n-l} \Jc^+_\ep(a^+_\ep) e^{-it\ep\Lambda^\sigma} \Jc^+_\ep(b^+_\ep)^\star \scal{\ep x}^{n}\|_{\Lc(L^2)} \leq C \scal{\ep t}^{-l}, \quad t \geq 0,
\]
uniformly in $\ep \in (0,1]$. This expected estimate again follows from Lemma $\ref{lem remainder 4}$ by taking $l$ large enough. This proves $(\ref{reduce propagation estimate 2 low freq})$. \newline
\indent iii. For $(\ref{propagation estimate 3 low freq})$, we firstly use the Isozaki-Kitada parametrix for $\tilde{\chi}^-_\ep$, namely
\begin{align}
e^{-it\ep\psi(\ep^{-2}P)} Op_\ep(\tilde{\chi}^-_\ep) \tilde{\zeta}(\ep x) = \Jc^-_\ep(\tilde{a}^-_\ep) e^{-it\ep\Lambda^\sigma} \Jc^-_\ep(\tilde{b}^-_\ep)^\star + \sum_{k=1}^{4} \tilde{\Rc}_k^-(N,t,\ep), \label{IK parametrix for chi - low freq}
\end{align}
where $\text{supp} (\tilde{a}^-_\ep) \subset \Ga^-(R^{1/4}, \tilde{J}_{1/4}, \tilde{\tau}_{1/4})$ and $\text{supp} (\tilde{b}^-_\ep) \subset \Ga^-(R^{3/4}, \tilde{J}_{3/4}, \tilde{\tau}_{3/4})$ with $\tilde{J}_{3/4} \Subset \tilde{J}_{1/4} $ small neighborhood of $J_1$ and $\tilde{\tau}_{1/4}, \tilde{\tau}_{3/4}$ can be chosen so that
\[
-1 < -\tau_4 < \tilde{\tau}_{1/4} < \tilde{\tau}_{3/4} < \tilde{\tau}_1 <1.
\]
Multiplying $\zeta(\ep x)Op_\ep(\chi^+_\ep)^\star f(\ep^{-2}P)$ to the left of $(\ref{IK parametrix for chi - low freq})$, the terms $\zeta(\ep x)Op_\ep(\chi^+_\ep)^\star f(\ep^{-2}P)\tilde{\Rc}_k^-(N,t,\ep)$ for  $k=1,2,3,4$ satisfy the required estimate using the estimate $(\ref{propagation estimate 1 low freq})$, Lemma $\ref{lem remainder 4}$ and $(\ref{estimate on B low freq})$. Therefore, it remains to show
\[
\|\zeta(\ep x)Op_\ep(\chi^+_\ep)^\star f(\ep^{-2}P) \Jc^-_\ep(\tilde{a}^-_\ep) e^{-it\ep\Lambda^\sigma} \Jc^-_\ep(\tilde{b}^-_\ep)^\star\| _{\Lc(L^2)} \leq C \scal{\ep t}^{-l}, \quad \pm t\leq 0,
\]
uniformly in $ \ep \in (0,1]$. Thanks to the support of $\tilde{a}^-_\ep$, we can write $\Jc^-_\ep(\tilde{a}^-_\ep)= \zeta_1(\ep x) \Jc^-_\ep(\tilde{a}^-_\ep)$ with $\zeta_1 \in C^\infty(\R^d)$ supported outside $B(0,1)$ such that $\zeta_1(x)=1$ for $|x| >R^{1/4}$. The parametrix of $f(\ep^{-2}P)\zeta_1(\ep x)$ given in Proposition $\ref{prop parametrix phi low freq}$ and symbolic calculus give 
\[
\zeta(\ep x)Op_\ep(\chi^+_\ep)^\star f(\ep^{-2}P)\zeta_1(\ep x)= Op_\ep(c^+_\ep) + B^+_N(\ep) \scal{\ep x}^{-N},
\]
where $(c^+_\ep)_{\ep \in (0,1]} \in S(0,-\infty)$ with $\text{supp}(c^+_\ep) \subset \text{supp}(\chi^+_\ep)$ and $B^+_N(\ep)=O_{\Lc(L^2)}(1)$ uniformly in $\ep \in (0,1]$. We treat the remainder term by using Lemma $\ref{lem remainder 3}$. For the main terms, we need to recall the following version of Proposition $\ref{prop action PDO on FIO}$ which is essentially \footnote{See $(\ref{define D_epsilon})$, $(\ref{rescaled FIO})$ and use Lemma 4.6 of \cite{BTglobalstrichartz} with $h=1$.} given in \cite[Lemma 4.6]{BTglobalstrichartz}.  
\begin{lem}\label{lem action hDO on FIO low freq}
Given $J \Subset (0,+\infty)$, $-1 <\tau <1$ and the associated families of phase functions $(S^\pm_{\ep,R})_{R \gg 1}$ as in \emph{Proposition} $\ref{prop hamilton-jacobi equation}$. Let $(a_\ep)_{\ep \in (0,1]}$ and $(c_\ep)_{\ep \in (0,1]}$ be bounded families in $S(0,-\infty)$. Then for all $N\geq 1$,
\[
Op_\ep(c_\ep)\Jc^\pm_\ep(a_\ep) = \sum_{j=0}^{N-1} \Jc^\pm_\ep(e_{\ep,j}) + \Jc^\pm_\ep(e_N(\ep)),
\]
where $(e_{\ep,j})_{\ep \in (0,1]}$ and $(e_N(\ep))_{\ep\in (0,1]}$ are bounded families in $S(0,-\infty)$ and $S(-N,-\infty)$ respectively. In particular, for all $\varepsilon>0$ small enough, by choosing $R>0$ large enough, we have
\[
\emph{supp}(c_\ep) \subset \Ga^\pm(R,J,\tau) \Longrightarrow \emph{supp}(e_{\ep,j}) \subset \Ga^\pm(R,J+(-\varepsilon,\varepsilon), \tau-\varepsilon)
\]
since $\nabla_x S^\pm_{\ep,R}(x,\xi) = \xi + O(R^{-\rho})$.
\end{lem}  
Using this lemma, we expand $Op_\ep(c^+_\ep) J^-_\ep(\tilde{a}^-_\ep)$ and treat the remainder terms using again Lemma $\ref{lem remainder 3}$. It remains to prove the required estimate for the general term, namely
\[
\|\Jc^-_\ep(e^+_\ep) e^{-it \ep \Lambda^\sigma} \Jc^-_\ep(\tilde{b}^-_\ep)^\star\|_{\Lc(L^2)} \leq C \scal{\ep t}^{-l}, \quad \pm t\leq 0, 
\]
uniformly in $\ep \in (0,1]$, where $(e^+_\ep)_{\ep \in (0,1]} \in S(0,-\infty)$ and $\text{supp}(e^+_\ep) \in \Ga^+(R^4,J_4+(-\varepsilon, \varepsilon), \tau_4-\varepsilon)$. 
Up to the conjugation by $D_\ep$, the kernel of the left hand side operator reads
\[
K_\ep(t,x,y)= (2\pi)^{-d}\int_{\R^d} e^{i\Phi_\ep(R,t,x,y,\xi)} e^+_\ep(x,\xi) \overline{\tilde{b}^-_\ep(y,\xi)} d\xi,
\]
where $\Phi_\ep(R,t,x,y,\xi)= S^-_{\ep,R}(x,\xi)-\ep t|\xi|^\sigma -S^-_{\ep,R}(y,\xi)$. 
Since $\text{supp}(e^+_\ep) \subset \Ga^+(R^4,J_4+(-\varepsilon, \varepsilon), \tau_4-\varepsilon)$ and $\text{supp}(\tilde{b}^-_\ep) \subset \Ga^-(R^{3/4},\tilde{J}_{3/4},\tilde{\tau}_{3/4})$, we have
\[
\frac{x\cdot \xi}{|x\|\xi|} > \tau_4 -\varepsilon,\quad -\frac{y\cdot\xi}{|y\|\xi|} > \tilde{\tau}_{3/4}.
\]
By choosing $R>0$ large enough, we have that $\tau_4-\varepsilon+\tilde{\tau}_{3/4}>0$. Thus by Item 2 of Lemma $\ref{lem sigma + -}$, we have
\[
|\nabla_\xi \Phi_\ep| \geq C(1+\ep|t|+|x|+|y|).
\]
Using the non-stationary phase argument as in the proof of Lemma $\ref{lem remainder 4}$, we have 
\[
\|\Jc^+_\ep(e^+_\ep) e^{-it \ep\Lambda^\sigma} \Jc^-_\ep(\tilde{b}^-_\ep)^\star\|_{\Lc(L^2)} \leq C \scal{\ep t}^{-l}, \quad \pm t\leq 0, 
\]
uniformly in $\ep\in (0,1]$. The proof of Proposition $\ref{prop microlocal propagation estimate}$ is now complete.
\end{proof}
\subsection{Strichartz estimates}
\paragraph{High frequencies.}
In this paragraph, we give the proof of $(\ref{reduction 1-chi u fractional schrodinger})$. By scaling in time, it is in turn equivalent to prove
\[
\|(1-\chi) e^{-ith^{-1} (h\Lambda_g)^\sigma} f(h^2P) u_0 \|_{L^p(\R,L^q)} \leq C h^{-\kappa_{p,q}} \|f(h^2P) u_0\|_{L^2},
\]
where $\kappa_{p,q}=d/2-d/q-1/p$. By choosing $\tilde{f}\in C^\infty_0(\R \backslash 0)$ such that $\tilde{f}=1$ near $\text{supp}(f)$, we can write for all $l\in \N$,
\[
(1-\chi) \tilde{f}(h^2P) =\sum_{k=0}^{N-1} h^k Op^h(a_k)^\star + h^N B_N(h)\scal{x}^{-l},
\]
where for $q\geq 2$,
\begin{align}
\|B_N(h)\|_{\Lc(L^2,L^q)} \leq Ch^{-(d/2-d/q)}. \label{estimate remainder B N}
\end{align}
Thus $(1-\chi)e^{-ith^{-1}(h\Lambda_g)^\sigma} f(h^2P)u_0$  becomes
\[
\sum_{k=0}^{N-1}h^k Op^h(a_k)^\star  e^{-ith^{-1}(h\Lambda_g)^\sigma} f(h^2P) u_0 + h^N B_N(h)\scal{x}^{-l} e^{-ith^{-1}(h\Lambda_g)^\sigma} f(h^2P) u_0.
\]
Using $(\ref{estimate remainder B N})$ and $(\ref{Lp global integrability})$, $\|B_N(h)\scal{x}^{-l} e^{-ith^{-1}(h\Lambda_g)^\sigma} f(h^2P) u_0\|_{L^p(\R, L^q)}$ is bounded by 
\begin{align}
Ch^{-(d/2-d/q)} \| \scal{x}^{-l} e^{-ith^{-1}(h\Lambda_g)^\sigma} f(h^2P) u_0\|_{L^p(\R, L^2)} \leq  Ch^{-(d/2-d/q)+(1-N_0)/p} \|f(h^2P)u_0\|_{L^2}. \nonumber 
\end{align}
Hence, by taking $N$ large enough, the remainder is bounded by $C h^{-\kappa_{p,q}}\|f(h^2P)u_0\|_{L^2}$. 
For the main terms, by choosing $\chi_0 \in C^\infty_0(\R^d)$ such that $\chi_0=1$ for $|x| \leq 2$ and setting $\chi (x)= \chi_0(x/R^4)$, we see that $(1-\chi)$ is supported in $\{x\in \R^d, |x| \geq 2R^4 >R^4 \}$. For $R>0$ large enough and $\text{supp}(\tilde{f})$ close enough to $\text{supp}(f)$ and $J_4 \Subset (0,+\infty)$ any open interval containing $\text{supp}(f)$, we have
\begin{align}
\text{supp}(a_k)  \subset \left\{ (x,\xi)\in \R^{2d}, |x| >R^4, |\xi|^2 \in J_4 \right\}, \quad k = 0,...,N-1. \label{support of a k}
\end{align}
We want to show
\[
\|Op^h(a_k)^\star e^{-ith^{-1}(h\Lambda_g)^\sigma} f(h^2P)u_0\|_{L^p(\R,L^q)} \leq C h^{-\kappa_{p,q}} \|f(h^2P)u_0\|_{L^2}, \quad k=0,...,N-1.
\]
Let us consider a general term, namely $Op^h(a)^\star e^{-ith^{-1}(h\Lambda_g)^\sigma} f(h^2P)u_0$ with $a \in S(0,-\infty)$ satisfying $(\ref{support of a k})$. Next, by choosing a suitable partition of unity $\theta^- +\theta^+ =1$ such that $\text{supp}(\theta^-) \subset (-\infty,-\tau_4)$ and $\text{supp}(\theta^+)\subset (\tau_4, +\infty)$ and setting 
\[
\chi^\pm(x,\xi)= a(x,\xi) \theta^\pm\left(\pm\frac{x\cdot\xi}{|x\|\xi|}\right),
\]
we have that $\chi^\pm \in S(0,-\infty)$, $\text{supp}(\chi^\pm) \subset \Ga^\pm(R^4,J_4,\tau_4)$ and 
\[
Op^h(a)^\star e^{-ith^{-1}(h\Lambda_g)^\sigma} f(h^2P)u_0 = ( Op^h(\chi^-)^\star + Op^h(\chi^+)^\star) e^{-ith^{-1}(h\Lambda_g)^\sigma} f(h^2P)u_0.
\]
We only prove the estimate for $\chi^+$, i.e.
\[
\|Op^h(\chi^+)^\star e^{-ith^{-1}(h\Lambda_g)^\sigma} f(h^2P)u_0\|_{L^p(\R,L^q)} \leq C h^{-\kappa_{p,q}} \|f(h^2P)u_0\|_{L^2},
\]
the one for $\chi^-$ is similar. Since $Op^h(\chi^+)^\star e^{-ith^{-1}(h\Lambda_g)^\sigma} f(h^2P)$ is bounded in $\Lc(L^2)$ uniformly in $h \in (0,1]$ and $t\in \R$, by Proposition $\ref{prop TTstar}$, it suffices to prove the dispersive estimates, i.e.
\[
\|Op^h(\chi^+)^\star e^{-ith^{-1}(h\Lambda_g)^\sigma} f^2(h^2P) Op^h(\chi^+)\|_{\Lc(L^1,L^\infty)} \leq Ch^{-d}(1+|t|h^{-1})^{-d/2}, 
\]
for all $t \in \R$ uniformly in $h\in(0,1]$. By taking the adjoint, it reduces to prove
\begin{align}
\|Op^h(\chi^+)^\star e^{-ith^{-1}(h\Lambda_g)^\sigma} f^2(h^2P) Op^h(\chi^+)\|_{\Lc(L^1,L^\infty)} \leq Ch^{-d}(1+|t|h^{-1})^{-d/2}, \label{reduce keel tao IK}
\end{align}
for all $t\leq 0$ uniformly in $h\in(0,1]$. We now prove $(\ref{reduce keel tao IK})$. By using the Isozaki-Kitada parametrix with $J_4$ and $\tau_4$ as above together with arbitrary open intervals $J_1, J_2, J_3$ such that $J_4 \Subset J_3 \Subset J_2 \Subset J_1 \Subset (0,+\infty)$ and arbitrary real numbers $\tau_1, \tau_2, \tau_3$ satisfying $-1 < \tau_1 < \tau_2 <\tau_3 < \tau_4 <1$, the operator in the left hand side of $(\ref{reduce keel tao IK})$ is written as
\[
Op^h(\chi^+)^\star f^2(h^2P) \left( J^+_h(a^+(h)) e^{-ith^{-1}(h\Lambda)^\sigma} J^+_h(b^+(h))^\star + \sum_{k=1}^{4}R^+_k(N,t,h) \right).
\]
Using the fact that $Op^h(\chi^+)^\star f^2(h^2P)$ is bounded in $\Lc(L^\infty)$ and Proposition $\ref{prop estimate main terms IK}$, we have
\[
\|Op^h(\chi^+)^\star f^2(h^2P)J^+_h(a^+(h)) e^{-ith^{-1}(h\Lambda)^\sigma} J^+_h(b^+(h))^\star\|_{\Lc(L^1,L^\infty)} \leq C h^{-d}(1+|t|h^{-1})^{-d/2},
\]
for all $t\in \R$ and $h\in (0,1]$. It remains to study the remainder terms. \newline
\indent For $k=1$, using the Sobolev embedding with $m> d/2$, $(\ref{propagation estimate 1})$ and the fact that $\scal{x}^l Op^h (\tilde{r}^+_N(h))$ is of size $O_{\Lc(H^{-m},L^2)}(h^{-m})$ by pseudo-differential calculus, we have 
\begin{align}
\|Op^h(\chi^+)^\star f^2(h^2P) R^+_1(N,t,h)\|_{\Lc(L^1,L^\infty)}\leq Ch^{N-1-2m}\scal{t}^{-3l/4} \leq  C h^{-d}(1+|t|h^{-1})^{-d/2}, \nonumber
\end{align}
for all $t\leq 0$ and all $h\in (0,1]$. The last estimate follows by taking $l=2d/3$ and $N$ large enough.\newline
\indent For $k=2$, 
by using $(\ref{propagation estimate 1})$ and the Sobolev embedding with $m>d/2$, we have for $t-s \leq 0$,
\begin{align}
\|Op^h(\chi^+)^\star e^{-i(t-s)h^{-1}(h\Lambda_g)^\sigma} f^2(h^2P) \scal{x}^{-l}\|_{\Lc(L^2,L^\infty)} \leq C h^{-m}\scal{t-s}^{-3l/4}. \label{propagation estimate 1 + sobolev embedding}
\end{align}
We also have that $\scal{x}^{l} R_N(h)$ is bounded in $\Lc(L^\infty,L^2)$ due to Proposition $\ref{prop parametrix phi}$ provided $N>l$. 
Thus for $N$ and $l$ large enough, Proposition $\ref{prop estimate main terms IK}$ implies that
\begin{align*}
\|Op^h(\chi^+)^\star  f^2(h^2P)& R^+_2(N,t,h)\|_{\Lc(L^1,L^\infty)} \\
&\leq Ch^{N-1-m-d} \int_0^t \scal{t-s}^{-3l/4} (1+|s|h^{-1})^{-d/2} ds 
\leq Ch^{-d} (1+|t|h^{-1})^{-d/2}. \nonumber
\end{align*}
\indent For $k=3$, by 
inserting $\scal{x}^{-l} \scal{x}^{l-N} \scal{x}^N$ and using the fact that $\scal{x}^{l-N}=O_{\Lc(L^\infty,L^2)}(1)$ for $N$ large enough, $(\ref{propagation estimate 1 + sobolev embedding})$ and Proposition $\ref{prop estimate main terms IK}$ with $J^+_h(a^+)=\scal{x}^NJ^+_h(r^+_N(h))$, we see that this remainder term satisfies the required estimate as for the second one. \newline 
\indent For $k=4$, we rewrite $Op^h(\chi^+)^\star f^2(h^2P)R^+_4(N,t,h)$ as $-ih^{-1}$ times
\[
\int_{0}^{t} Op^h(\chi^+)^\star f^2(h^2P) e^{-i(t-s)h^{-1}(h\Lambda_g)^\sigma} (\chi+(1-\chi))(x/R^2) J^+_h(\check{a}^+(h)) e^{-ish^{-1}(h\Lambda)^\sigma} J^+_h(b^+(h))^\star ds,
\]
where $\chi\in C^\infty_0(\R^d)$ satisfying $\chi(x)=1$ for $|x|\leq 2$. The first term can be treated similarly as the second remainder using $(\ref{propagation estimate 2})$ instead of $(\ref{propagation estimate 1})$. For the second term, we need the following lemma (see \cite[Proposition 5.2]{BTglobalstrichartz}).
\begin{lem}
Choose $\tilde{\tau}_1$ such that $-\tau_4 < \tilde{\tau}_1 < -\tau_2$. If $R>0$ is large enough, we may choose $\tilde{\chi}^- \in S(0,-\infty)$ satisfying $\emph{supp}(\tilde{\chi}^-) \subset \Ga^-(R, J_1,\tilde{\tau}_1)$ such that for all $m$ large enough, 
\[
f(h^2P) (1-\chi)(x/R^2) J^+_h(\check{a}^+(h)) = Op^h(\tilde{\chi}^-) J^+_h(\tilde{e}_m(h)) + h^m \tilde{R}_m(h) 
\]
where 
\[
\tilde{R}_m(h) = J^+_h(\tilde{r}_m(h)) + \scal{x}^{-m/2} R_m(h) \scal{x}^{-m/2} J^+_h(\check{a}^+(h)),
\]
with $(\tilde{e}_m(h))_{h\in (0,1]}$ and $(\tilde{r}_m(h))_{h\in (0,1]}$ bounded families in $S(0,-\infty)$ and $S(-m,-\infty)$ respectively and $R_m(h)=O_{\Lc(L^\infty)}(1)$ uniformly in $h\in (0,1]$.
\end{lem}
Using this lemma, the second term is written as $-ih^{-1}$ times
\[
\int_{0}^{t} Op^h(\chi^+)^\star e^{-i(t-s)h^{-1}(h\Lambda_g)^\sigma} \left(Op^h(\tilde{\chi}^-) J^+_h(\tilde{e}_m(h)) + h^m \tilde{R}_m(h) \right)e^{-ish^{-1}(h\Lambda)^\sigma} J^+_h(b^+(h))^\star ds.
\]
The remainder terms are treated similarly as the second remainder term using $(\ref{propagation estimate 1})$. The term involving $Op^h(\tilde{\chi}^-) J^+_h(\tilde{e}_m(h))$ is studied by the same analysis as the second term using $(\ref{propagation estimate 3})$ instead of $(\ref{propagation estimate 1})$. This completes the proof.
\defendproof

\paragraph{Low frequencies.}
In this paragraph, we will prove $(\ref{reduce strichartz outside compact low freq})$. By scaling in time, it is equivalent to show 
\begin{align}
\|(1-\chi)(\ep x) f(\ep^{-2}P) e^{-it\ep (\ep^{-1}\Lambda_g)^\sigma} u_0\|_{L^p(\R, L^q)} \leq C \ep^{\kappa_{p,q}} \|f(\ep^{-2}P) u_0\|_{L^2}, \nonumber 
\end{align}
where $\kappa_{p,q}=d/2-d/q-1/p$. By choosing $\tilde{\tilde{f}}\in C^\infty_0(\R \backslash 0)$ such that $\tilde{\tilde{f}}=1$ near $\text{supp}(f)$, we can write $(1-\chi)(\ep x) f(\ep^{-2}P)= (1-\chi)(\ep x)\tilde{\tilde{f}}(\ep^{-2}P) f(\ep^{-2}P)$. Next, we choose $\zeta \in C^\infty(\R^d)$ supported in $\R^d \backslash B(0,1)$ such that $\zeta =1$ near $\text{supp}(1-\chi)$ and use Proposition $\ref{prop parametrix phi low freq}$ to have
\[
(1-\chi)(\ep x) \tilde{\tilde{f}}(\ep^{-2}P)= \sum_{k=0}^{N-1} \zeta(\ep x) Op_\ep(a_{\ep,k})^\star + R_N(\ep),
\]
where $R_N(\ep)= \zeta(\ep x) (\ep^{-2}P+1)^{-N} B_N(\ep) \scal{\ep x}^{-N}$ with $(B_N(\ep))_{\ep\in (0,1]}$ bounded  $\Lc(L^2)$. Thus $(1-\chi)(\ep x) f(\ep^{-2}P) e^{-it\ep (\ep^{-1}\Lambda_g)^\sigma} u_0$ reads
\[
\sum_{k=0}^{N-1} \zeta(\ep x) Op_\ep(a_{\ep,k})^\star e^{-it\ep (\ep^{-1}\Lambda_g)^\sigma}f(\ep^{-2}P)u_0+ R_N(\ep) e^{-it\ep (\ep^{-1}\Lambda_g)^\sigma} f(\ep^{-2}P)u_0.
\] 
\indent We firstly consider the remainder term.
\begin{prop} \label{prop remainder low freq}
Let $N \geq (d-1)/2+1$. Then for all $(p,q)$ fractional admissible, there exists $C>0$ such that for all $\ep \in (0,1]$,
\[
\|R_N(\ep)e^{-it\ep (\ep^{-1}\Lambda_g)^\sigma}f(\ep^{-2}P) u_0\|_{L^p(\R, L^q)} \leq C \ep^{\kappa_{p,q}} \|u_0\|_{L^2}.
\]
\end{prop}
\begin{proof}
This result follows from the $TT^\star$ criterion given in Proposition $\ref{prop TTstar}$ with $\ep^{-1}$ in place of $h$ and $T(t)= R_N(\ep) e^{-it\ep (\ep^{-1}\Lambda_g)^\sigma}f(\ep^{-2}P)$. The $\Lc(L^2)$ bounds of $T(t)$ are obvious. Thus we need to prove the dispersive estimates. 
Using $(\ref{L2 Lq estimate low freq parametrix})$ with $q=\infty$ and $(\ref{local energy decay fractional schrodinger low freq})$ with $N \geq d/2+1$, we have
\begin{align}
\|T(t)T(s)^\star\|_{L^1\rightarrow L^\infty} &\leq C \ep^d \|\scal{\ep x}^{-N} e^{-i(t-s)\ep (\ep^{-1}\Lambda_g)^\sigma} f^2(\ep^{-2}P) \scal{\ep x}^{-N}\|_{\Lc(L^2)} \nonumber \\
&\leq C \ep^d \scal{\ep(t-s)}^{1-N} \leq C \ep^{d}(1+\ep|t-s|)^{-d/2}. \nonumber
\end{align}
This completes the proof.
\end{proof}
For the main terms, by choosing $\chi_0 \in C^\infty_0(\R^d)$ such that $\chi_0=1$ for $|x| \leq 2$ and setting $\chi (x)= \chi_0(x/R^4)$, we see that $(1-\chi)$ is supported in $\{x\in \R^d, |x| >R^4 \}$. For $R>0$ large enough and $\text{supp}(\tilde{f})$ close enough to $\text{supp}(f)$ and $J_4 \Subset (0,+\infty)$ any open interval containing $\text{supp}(f)$, we have
\begin{align}
\text{supp}(a_{\ep,k})  \subset \left\{ (x,\xi)\in \R^{2d}, |x| >R^4, |\xi|^2 \in J_4 \right\}, \quad k = 0,...,N-1. \label{support of a k low freq}
\end{align}
We want to show for $k=0,...,N-1$,
\[
\|\zeta(\ep x) Op_\ep(a_{\ep,k})^\star e^{-it\ep (\ep^{-1}\Lambda_g)^\sigma}f(\ep^{-2}P)u_0\|_{L^p(\R,L^q)} \leq C \ep^{\kappa_{p,q}} \|f(\ep^{-2}P)u_0\|_{L^2}.
\]
Let us consider the general term, namely $\zeta(\ep x)Op_\ep(a_\ep)^\star e^{-it\ep (\ep^{-1}\Lambda_g)^\sigma} f(\ep^{-2}P)u_0$ with $(a_\ep)_{\ep \in (0,1]} \in S(0,-\infty)$ satisfying $(\ref{support of a k low freq})$. Next, by choosing a suitable partition of unity $\theta^- +\theta^+ =1$ such that $\text{supp}(\theta^-) \subset (-\infty,-\tau_4)$ and $\text{supp}(\theta^+)\subset (\tau_4, +\infty)$ and setting 
\[
\chi^\pm_\ep(x,\xi)= a_\ep(x,\xi) \theta^\pm\left(\pm\frac{x\cdot\xi}{|x\|\xi|}\right),
\]
we have that $(\chi^\pm_\ep)_{\ep \in (0,1]} \in S(0,-\infty)$, $\text{supp}(\chi^\pm_\ep) \subset \Ga^\pm(R^4,J_4,\tau_4)$ and 
\[
\zeta(\ep x)Op_\ep(a_\ep)^\star e^{-it\ep (\ep^{-1}\Lambda_g)^\sigma} f(\ep^{-2}P)u_0 = \zeta(\ep x) (Op_\ep(\chi^-_\ep)^\star + Op_\ep(\chi^+_\ep)^\star) e^{-it\ep (\ep^{-1}\Lambda_g)^\sigma} f(\ep^{-2}P)u_0.
\]
We only prove the estimate for $\chi^+_\ep$, i.e.
\[
\|\zeta(\ep x)Op_\ep(\chi^+_\ep)^\star e^{-it\ep (\ep^{-1}\Lambda_g)^\sigma} f(\ep^{-2}P)u_0\|_{L^p(\R,L^q)} \leq C \ep^{\kappa_{p,q}} \|f(\ep^{-2}P)u_0\|_{L^2},
\]
the one for $\chi^-_\ep$ is similar. By $TT^\star$ criterion and that $T(t):=\zeta(\ep x) Op_\ep(\chi^+_\ep)^\star e^{-it\ep (\ep^{-1}\Lambda_g)^\sigma} f(\ep^{-2}P)$ is bounded in $\Lc(L^2)$ for all $t\in \R$ and all $\ep \in (0,1]$, it suffices to prove dispersive estimates, i.e. 
\[
\|\zeta(\ep x) Op_\ep(\chi^+_\ep)^\star e^{-it\ep (\ep^{-1}\Lambda_g)^\sigma} f^2(\ep^{-2}P) Op_\ep(\chi^+_\ep)\zeta(\ep x)\|_{\Lc(L^1,L^\infty)} \leq C \ep^{d}(1+\ep |t|)^{-d/2},
\]
for all $t \in \R$ uniformly in $\ep \in(0,1]$. By taking the adjoint, it reduces to prove
\begin{align}
\|\zeta(\ep x)Op_\ep(\chi^+_\ep)^\star e^{-it\ep (\ep^{-1}\Lambda_g)^\sigma} f^2(\ep^{-2}P) Op_\ep(\chi^+_\ep)\zeta(\ep x)\|_{\Lc(L^1,L^\infty)} \leq C\ep^{d}(1+\ep|t|)^{-d/2}, \label{reduce keel tao IK low freq}
\end{align}
for all $t\leq 0$ uniformly in $\ep \in(0,1]$. 
Let us prove $(\ref{reduce keel tao IK low freq})$. For simplicity, we set
\[
A^+_\ep := \zeta(\ep x) Op_\ep(\chi^+_\ep)^\star f^2(\ep^{-2}P).
\] 
Using the Isozaki-Kitada parametrix given in Theorem $\ref{theorem Isozaki-Kitada parametrix low freq}$, we see that 
\[
A^+_\ep e^{-it\ep (\ep^{-1}\Lambda_g)^\sigma}Op_\ep(\chi^+_\ep)\zeta(\ep x) = A^+_\ep \left( \mathcal{J}^+_\ep(a^+_\ep) e^{-it\ep\Lambda^\sigma} \mathcal{J}^+_\ep(b^+_\ep)^\star + \sum_{k=1}^{4}\mathcal{R}^+_k(N,t,\ep) \right).
\]
We firstly note that $A^+_\ep$ is bounded in $\Lc(L^\infty)$. Indeed, we write
\[
\zeta(\ep x) Op_\ep(\chi^+_\ep)^\star f^2(\ep^{-2}P)=\zeta(\ep x) Op_\ep(\chi^+_\ep)^\star \zeta_1(\ep x) f^2(\ep^{-2}P),
\]
where $\zeta_1 \in C^\infty(\R^d)$ is supported outside $B(0,1)$ satisfying $\zeta_1(x)=1$ for $|x|>R^4$. This is possible since $Op_\ep (\chi^+_\ep) = \zeta_1(\ep x) Op_\ep(\chi^+_\ep)$. The factors $\zeta(\ep x) Op_\ep(\chi^+_\ep)^\star$ and $\zeta_1(\ep x) f^2(\ep^{-2}P)$ are bounded in $\Lc(L^\infty)$ by the rescaled pseudo-differential operator and Corollary $\ref{coro Lq Lr estimate of parametrix low freq}$ respectively. Thanks to the $\Lc(L^\infty)$-bound of $A^+_\ep$ and $(\ref{dispersive main term IK low freq})$, we have dispersive estimates for the main terms. 
It remains to prove dispersive estimates for remainder terms. By rescaled pseudo-differential calculus, we can write for $l > d/2$, 
\[
A^+_\ep = \tilde{\zeta}(\ep x) (\ep^{-2}P+1)^{-l}  \left(\zeta(\ep x) Op_\ep(\tilde{\chi}^+_\ep)^\star + \tilde{B}^+_l(\ep) \scal{\ep x}^{-l} \right) f^2(\ep^{-2}P),
\]
where $\tilde{\zeta} \in C^\infty(\R^d)$ is supported outside $B(0,1)$ and equal to 1 near $\text{supp}(\zeta)$ and $(\tilde{\chi}^+_\ep)_{\ep \in (0,1]} \in S(0,-\infty)$ satisfying $\text{supp}(\tilde{\chi}^+_\ep) \subset \text{supp}(\chi^+_\ep)$ and $\tilde{B}^+_l(\ep)= O_{\Lc(L^2)}(1)$ uniformly in $\ep \in (0,1]$. This follows by expanding $(\ep^{-2}P+1)^l \zeta(\ep x) Op_\ep (\chi^+_\ep)^\star$ by rescaled pseudo-differential calculus. \newline
\indent For $k=1$, using the Proposition $\ref{prop parametrix phi low freq}$, we can write
\[
\Rc^+_1(N,t,\ep)= e^{-it\ep (\ep^{-1}\Lambda_g)^\sigma} \scal{\ep x}^{-N} B^+_N(\ep) (\ep^{-2}P+1)^{-N} \zeta(\ep x),
\]
where $B^+_N(\ep)= O_{\Lc(L^2)}(1)$ uniformly in $\ep \in (0,1]$. Then, using Proposition $\ref{prop L2 Lq resolvent low freq}$ with $q=\infty$ and $(\ref{propagation estimate 1 low freq})$, we have
\begin{align*}
\|\tilde{\zeta}(\ep x) (\ep^{-2}P+1)^{-l} \zeta(\ep x) Op_\ep(\tilde{\chi}^+_\ep)^\star f^2(\ep^{-2}P) \Rc^+_1(N,t,\ep)\|_{\Lc(L^1, L^\infty)} &\leq C \ep^d \scal{\ep t}^{-3N/4} \\
&\leq C \ep^d (1+\ep|t|)^{-d/2}, \nonumber 
\end{align*}
for all $t\leq 0$ and all $\ep \in (0,1]$ provided $N$ is taken large enough. Moreover, using again Proposition $\ref{prop L2 Lq resolvent low freq}$ and $(\ref{local energy decay fractional schrodinger low freq})$, we also have
\begin{align*}
\|\tilde{\zeta}(\ep x) (\ep^{-2}P+1)^{-l} \tilde{B}_l(\ep) \scal{\ep x}^{-l} f^2(\ep^{-2}P) \Rc^+_1(N,t,\ep)\|_{\Lc(L^1, L^\infty)}  &\leq C \ep^d \scal{\ep t}^{1-l} \\
&\leq C \ep^d (1+\ep|t|)^{-d/2}, \nonumber 
\end{align*}
for all $t\leq 0$ and all $\ep \in (0,1]$ provided $l$ and $N$ are taken large enough. This implies
\[
\|A^+_\ep\Rc^+_1(N,t,\ep)\|_{\Lc(L^1, L^\infty)} \leq C \ep^d (1+\ep|t|)^{-d/2}, 
\]
for all $t\leq 0$ and all $\ep \in (0,1]$. \newline
\indent Next, thanks to the support of $b^+_\ep$, we can write
\begin{align}
\Jc^+_\ep(b^+_\ep)^\star = \Jc^+_\ep(\tilde{b}^+_\ep)^\star (\ep^{-2}P+1)^{-N}\zeta_1(\ep x), \label{express of J b star}
\end{align}
where $(\tilde{b}^+_\ep)_{\ep \in (0,1]} \in S(0,-\infty)$, $\text{supp}(\tilde{b}^+_\ep) \subset \Ga^+(R^3, J_3, \sigma_3)$ and $\zeta_1\in C^\infty(\R^d)$ is supported outside $B(0,1)$ such that $\zeta_1(x)=1$ for $|x| >R^3$. Indeed, we write for $\tilde{\zeta}_1 \in C^\infty(\R^d)$ supported outside $B(0,1)$ and $\tilde{\zeta}_1 =1$ in $\text{supp}(\zeta_1)$,
\[
\Jc^+_\ep(b^+_\ep)^\star = \Jc^+_\ep(b^+_\ep)^\star \tilde{\zeta}_1(\ep x) (\ep^{-2}P+1)^N \left( (\ep^{-2}P+1)^{-N} \zeta_1(\ep x)\right).
\] 
We have $(\ref{express of J b star})$ by taking the adjoint of $(\ep^{-2}P+1)^N \tilde{\zeta}_1(\ep x) \Jc^+_\ep(b^+_\ep) = \Jc^+_\ep (\tilde{b}^+_\ep)$. \newline
\indent For $k=2$, using $(\ref{L2 Lq estimate low freq parametrix})$ and its adjoint, $(\ref{propagation estimate 1 low freq})$, $(\ref{express of J b star})$, $\scal{\ep x}^{l} R_N(\ep) \scal{\ep x}^{N-l}=O_{\Lc(L^2)}(1)$ and estimating as in Lemma $\ref{lem remainder 3}$, we have
\begin{align*}
\|\tilde{\zeta}(\ep x) (\ep^{-2}P+1)^{-l} \zeta(\ep x) Op_\ep(\tilde{\chi}^+_\ep)^\star &f^2(\ep^{-2}P) \Rc^+_2(N,t,\ep)\|_{\Lc(L^1, L^\infty)}  \\
&\leq C \ep^d  \ep \int_{0}^{t} \scal{\ep (t-s)}^{-3l/4}\scal{\ep s}^{-N/4}ds \leq C \ep^d (1+\ep|t|)^{-d/2}, \nonumber 
\end{align*}
for $t\leq 0$ provided that $l$ and $N$ are taken large enough. Moreover, using $(\ref{local energy decay fractional schrodinger low freq})$ instead of $(\ref{propagation estimate 1 low freq})$, we have
\begin{align*}
\|\tilde{\zeta}(\ep x) (\ep^{-2}P+1)^{-l} \tilde{B}_l(\ep) \scal{\ep x}^{-l} &f^2(\ep^{-2}P) \Rc^+_2(N,t,\ep)\|_{\Lc(L^1, L^\infty)}  \\
&\leq C \ep^d \ep \int_{0}^{t} \scal{\ep (t-s)}^{1-l} \scal{\ep s}^{-N/4} ds \leq C \ep^d (1+\ep|t|)^{-d/2}, \nonumber 
\end{align*}
for all $t\leq 0$ and all $\ep \in (0,1]$. This implies
\[
\|A^+_\ep\Rc^+_2(N,t,\ep)\|_{\Lc(L^1, L^\infty)} \leq C \ep^d (1+\ep|t|)^{-d/2}, \quad \forall t\leq 0, \ep \in (0,1].
\]
\indent The third remainder term is treated similarly as the second one. It remains to study the last remainder term. To do so, we split 
\[
A^+_\ep \Rc^+_4(N,t,\ep)=-i\int_{0}^{t} A^+_\ep e^{-i(t-s)\ep (\ep^{-1}\Lambda_g)^\sigma} (\chi+(1-\chi))(\ep x/R^2) \Jc^+_\ep(\check{a}^+(\ep)) e^{-is\ep\Lambda^\sigma} \Jc^+_\ep(b^+_\ep)^\star ds,
\]
where $\chi\in C^\infty_0(\R^d)$ satisfying $\chi(x)=1$ for $|x|\leq 2$. The first term can be treated similarly as the second remainder using $(\ref{propagation estimate 2 low freq})$ instead of $(\ref{propagation estimate 1 low freq})$ and Lemma $\ref{lem remainder 4}$. For the second term, we need the following lemma (see \cite[Proposition 5.2]{BTglobalstrichartz}).
\begin{lem} \label{lem expansion rescaled pdo}
Choose $\tilde{\tau}_1$ such that $-\tau_4 < \tilde{\tau}_1 < -\tau_2$. If $R>0$ is large enough, we may choose a bounded family of symbols $\tilde{\chi}^-_\ep \in S(0,-\infty)$ satisfying $\emph{supp}(\tilde{\chi}^-_\ep) \subset \Ga^-(R, J_1,\tilde{\tau}_1)$ and $\tilde{\zeta}_2 \in C^\infty(\R^d)$ supported outside $B(0,1)$ satisfying $\tilde{\zeta}_2=1$ on $\emph{supp}(1-\chi)$ such that for all $m$ large enough, 
\[
f(\ep^{-2}P) (1-\chi)(\ep x/R^2) \Jc^+_\ep(\check{a}^+(\ep)) = Op_\ep(\tilde{\chi}^-_\ep) \tilde{\zeta}_2(\ep x) \Jc^+_\ep(\tilde{e}_m(\ep)) + \tilde{R}_m(\ep),
\]
where 
\[
\tilde{R}_m(\ep) = \Jc^+_\ep(\tilde{r}_m(\ep)) + \scal{\ep x}^{-m/2} R_m(\ep) \scal{\ep x}^{-m/2} \Jc^+_\ep(\check{a}^+(\ep)),
\]
with $(\tilde{e}_m(\ep))_{\ep \in (0,1]}$ and $(\tilde{r}_m(\ep))_{\ep \in (0,1]}$ bounded families in $S(0,-\infty)$ and $S(-m,-\infty)$ respectively and $R_m(\ep)=O_{\Lc(L^2)}(1)$ uniformly in $\ep \in (0,1]$.
\end{lem}
We set
\[
A^+_\ep = (A^+_{\ep,1} + A^+_{\ep,2})f(\ep^{-2}P),
\]
where 
\begin{align}
A^+_{\ep,1} &= \tilde{\zeta}(\ep x) (\ep^{-2}P+1)^{-l} \zeta(\ep x) Op_\ep(\tilde{\chi}^+_\ep)^\star f(\ep^{-2}P), \nonumber \\
A^+_{\ep,2} &= \tilde{\zeta}(\ep x) (\ep^{-2}P+1)^{-l} \tilde{B}_l(\ep) \scal{\ep x}^{-l} f(\ep^{-2}P). \nonumber
\end{align}
Using Lemma $\ref{lem expansion rescaled pdo}$, we firstly consider
\[
-ih^{-1}\int_{0}^{t} A^+_{\ep,1} e^{-i(t-s)\ep (\ep^{-1}\Lambda_g)^\sigma} \left(Op_\ep(\tilde{\chi}^-_\ep) \tilde{\zeta}_2(\ep x)\Jc^+_\ep(\tilde{e}_m(\ep)) + \tilde{R}_m(\ep) \right)e^{-is\ep \Lambda^\sigma} \Jc^+_\ep(b^+_\ep)^\star ds.
\]
The remainder terms are treated similarly as the second remainder term using $(\ref{propagation estimate 1 low freq})$ and Lemma $\ref{lem remainder 3}$. The term involving $Op_\ep(\tilde{\chi}^-_\ep) \tilde{\zeta}_2(\ep x) \Jc^+_\ep(\tilde{e}_m(\ep))$ is studied by the same analysis as the second term using $(\ref{propagation estimate 3 low freq})$ instead of $(\ref{propagation estimate 1})$. For the term
\[
-ih^{-1}\int_{0}^{t} A^+_{\ep,2} e^{-i(t-s)\ep (\ep^{-1}\Lambda_g)^\sigma} \left(Op_\ep(\tilde{\chi}^-_\ep)\tilde{\zeta}_2(\ep x) \Jc^+_\ep(\tilde{e}_m(\ep)) + \tilde{R}_m(\ep) \right)e^{-is\ep \Lambda^\sigma} \Jc^+_\ep(b^+_\ep)^\star ds,
\]
the required estimate follows by using $(\ref{local energy decay fractional schrodinger low freq})$ and Lemma $\ref{lem remainder 3}$.
This completes the proof.
\defendproof
\section{Inhomogeneous Strichartz estimates} \label{section inhomogeneous strichartz estimates}
\setcounter{equation}{0}
In this section, we will give the proofs of Proposition $\ref{prop global inhomogeneous strichartz frac schro}$ and Proposition $\ref{prop global inhomogeneous strichartz frac wave}$. The main tool is the homogeneous Strichartz estimates $(\ref{global strichartz frac schro})$ and the so called Christ-Kiselev Lemma. To do so, we recall the following result (see \cite{ChristKiselev} or \cite{Sogge}).
\begin{lem}\label{lem Christ-Kiselev lemma}
Let $X$ and $Y$ be Banach spaces and assume that $K(t,s)$ is a continuous function taking its values in the bounded operators from $Y$ to $X$. Suppose that $-\infty \leq c < d \leq \infty$, and set
\[
Af(t) = \int_c^d K(t,s)f(s) ds.
\]
Assume that 
\[
\|Af\|_{L^q([c,d],X)} \leq C \|f\|_{L^p([c,d],Y)}.
\]
Define the operator $\tilde{A}$ as
\[
\tilde{A}f(t) = \int_c^t K(t,s) f(s)ds,
\]
Then for $1 \leq p <q \leq \infty$, there exists $\tilde{C}>0$ such that
\[
\|\tilde{A}f\|_{L^q([c,d],X)} \leq \tilde{C} \|f\|_{L^p([c,d],Y)}.
\]  
\end{lem}
We are now able to prove the inhomogeneous Strichartz estimates $(\ref{global inhomogeneous strichartz frac schro})$ and $(\ref{global inhomogeneous strichartz frac wave})$.  
\paragraph{Inhomogeneous Strichartz estimates for fractional Schr\"odinger equation.}
We give the proof of Proposition $\ref{prop global inhomogeneous strichartz frac schro}$ by following a standard argument (see e.g. \cite{Zhang}). Let $u$ be the solution to $(\ref{fractional schrodinger equation})$. By Duhamel formula, we have
\[
u(t)= e^{-it\Lambda_g^\sigma} u_0 -i \int_{0}^{t} e^{-i(t-s)\Lambda_g^\sigma} F(s)ds =: u_{\text{hom}}(t)+ u_{\text{inh}}(t).
\]
Using $(\ref{global strichartz frac schro})$, we have
\[
\|u_{\text{hom}}\|_{L^p(\R, L^q)} \leq C\|u_0\|_{\dot{H}^{\gamma_{p,q}}_g}.
\]
It remains to prove the inhomogeneous part, namely
\[
\Big\|\int_{0}^{t} e^{-i(t-s)\Lambda_g^\sigma} F(s)ds \Big\|_{L^p(\R, L^q)} \leq C \|F\|_{L^{a'}(\R, L^{b'})}, \nonumber
\]
where $(p,q), (a,b)$ are fractional admissible pairs satisfying $(p,a) \ne (2,2)$ and the gap condition $(\ref{gap condition frac schro})$. By the Christ-Kiselev Lemma, it suffices to prove
\begin{align}
\Big\|\int_{\R} e^{-i(t-s)\Lambda_g^\sigma} F(s) ds \Big\|_{L^p(\R, L^q)} \leq C\|F\|_{L^{a'}(\R, L^{b'})}, \label{inhomo strichartz frac schro}
\end{align}
for all fractional admissible pairs satisfying $(\ref{gap condition frac schro})$ excluding the case $p=a'=2$. We now prove $(\ref{inhomo strichartz frac schro})$. Define
\[
T_{\gamma_{p,q}}: u_0 \in \Lch_g \mapsto \Lambda_g^{-\gamma_{p,q}} e^{-it\Lambda_g^\sigma}u_0 \in L^p(\R,L^q).
\]
Thanks to $(\ref{global strichartz frac schro})$, we see that $T_{\gamma_{p,q}}$ is a bounded operator. Similar result holds for $T_{\gamma_{a,b}}$. Next, we take the adjoint for $T_{\gamma_{a,b}}$ and obtain a bounded operator
\[
T^\star_{\gamma_{a,b}}: F \in L^{a'}(\R, L^{b'}) \mapsto \int_{\R} \Lambda_g^{-\gamma_{a,b}} e^{is\Lambda_g^\sigma}F(s) ds \in  \Lch'_g,
\] 
where $\Lch'_g$ is the dual space of $\Lch_g$. Using $(\ref{gap condition frac schro})$ or $\gamma_{a,b} = -\gamma_{a',b'}-\sigma=-\gamma_{p,q}$, we have
\[
\Big\|\int_{\R} e^{-i(t-s)\Lambda_g^\sigma} F(s) ds \Big\|_{L^p(\R, L^q)} = \|T_{\gamma_{p,q}} T^\star_{\gamma_{a,b}}F\|_{L^p(\R, L^q)} \leq C\|F\|_{L^{a'}(\R,L^{b'})},
\]
and $(\ref{inhomo strichartz frac schro})$ follows. \newline
\indent Next, we prove
\begin{align}
\|u\|_{L^\infty(\R, \dot{H}^{\gamma_{p,q}}_g)} \leq C\Big( \|u_0\|_{\dot{H}^{\gamma_{p,q}}_g}+\|F\|_{L^{a'}(\R, L^{b'})}\Big). \nonumber
\end{align}
By using the homogeneous Strichartz estimate for a fractional admissible pair $(\infty,2)$ with $\gamma_{\infty,2}=0$ and that $\|u\|_{L^\infty(\R, \dot{H}^{\gamma_{p,q}}_g)} = \|\Lambda_g^{\gamma_{p,q}}u\|_{L^\infty(\R, L^2)}$, we have
\begin{align}
\|u\|_{L^\infty(\R, \dot{H}^{\gamma_{p,q}}_g)} \leq C\Big( \|\Lambda_g^{\gamma_{p,q}}u_0\|_{L^2}+\Big\|\int_{0}^{t} \Lambda_g^{\gamma_{p,q}} e^{-i(t-s)\Lambda_g^\sigma} F(s)ds\Big\|_{L^\infty(\R, L^2)}\Big). \nonumber
\end{align}
Using the Christ-Kiselev Lemma, it suffices to prove
\[
\Big\|\int_{\R} \Lambda_g^{\gamma_{p,q}}  e^{-i(t-s)\Lambda_g^\sigma} F(s)ds\Big\|_{L^\infty(\R, L^2)} \leq C \|F\|_{L^{a'}(\R, L^{b'})}.
\]
Using the above notation, we have
\begin{align}
\Big\|\int_{\R} \Lambda_g^{\gamma_{p,q}}  e^{-i(t-s)\Lambda_g^\sigma} F(s)ds\Big\|_{L^\infty(\R, L^2)} &= \|T_0 T^\star_{\gamma_{a,b}}F\|_{L^\infty(\R, L^2)} \nonumber \\
&\leq C \|T^\star_{\gamma_{a,b}}F\|_{L^2} \leq C \|F\|_{L^{a'}(\R, L^{b'})}. \nonumber
\end{align}
This completes the proof of Proposition $\ref{prop global inhomogeneous strichartz frac schro}$.  \defendproof
\paragraph{Inhomogeneous Strichartz estimates for fractional wave equation.}
We give the proof of Proposition $\ref{prop global inhomogeneous strichartz frac wave}$. Let $v$ be the solution to $(\ref{linear fractional wave equation})$. By Duhamel formula, we have
\[
v(t)= \cos t\Lambda_g^\sigma u_0 + \frac{\sin t\Lambda_g^\sigma}{\Lambda_g^\sigma}u_1 + \int_{0}^{t} \frac{\sin (t-s)\Lambda_g^\sigma}{\Lambda_g^\sigma} F(s)ds =: v_{\text{hom}}(t)+ v_{\text{inh}}(t),
\]
where $v_{\text{hom}}$ is the sum of first two terms and $v_{\text{inh}}$ is the last one. We firstly prove
\begin{align}
\|v\|_{L^p(\R, L^q)} \leq C\Big( \|v_0\|_{\dot{H}^{\gamma_{p,q}}_g}+\|v_1\|_{\dot{H}^{\gamma_{p,q}-\sigma}_g}+\|F\|_{L^{a'}(\R, L^{b'})}\Big). \nonumber 
\end{align}
By observing that
\[
\cos t\Lambda_g^\sigma= \frac{e^{it\Lambda_g^\sigma} + e^{-it\Lambda_g^\sigma}}{2}, \quad \sin t\Lambda_g^\sigma= \frac{e^{it\Lambda_g^\sigma} - e^{-it\Lambda_g^\sigma}}{2i},
\]
and using $(\ref{global strichartz frac schro})$, we have
\[
\|v_{\text{hom}}\|_{L^p(\R, L^q)} \leq C \Big(\|v_0\|_{\dot{H}^{\gamma_{p,q}}_g} +\|v_1\|_{\dot{H}^{\gamma_{p,q}-\sigma}_g}\Big).
\]
Let us prove the inhomogeneous part which is in turn equivalent to 
\begin{align}
\Big\|\int_{0}^{t}\frac{e^{-i(t-s)\Lambda_g^\sigma}}{\Lambda_g^\sigma}F(s)ds \Big\|_{L^p(\R, L^q)} \leq C \|F\|_{L^{a'}(\R, L^{b'})}, \label{inhomo strichartz frac wave}
\end{align}
where $(p,q), (a,b)$ are fractional admissible satisfying the gap condition $(\ref{gap condition frac wave})$. We define the operator
\[
T_{\gamma_{p,q}}: u_0 \in \Lch_g \mapsto \Lambda_g^{-\gamma_{p,q}} e^{-it\Lambda_g^\sigma}u_0 \in L^p(\R,L^q).
\]
Thanks to $(\ref{global strichartz frac schro})$, we see that $T_{\gamma_{p,q}}$ is a bounded operator. Next, we take the adjoint for $T_{\gamma_{a,b}}$ and obtain a bounded operator
\[
T^\star_{\gamma_{a,b}}: F \in L^{a'}(\R, L^{b'}) \mapsto \int_{\R} \Lambda_g^{-\gamma_{a,b}} e^{is\Lambda_g^\sigma}F(s) ds \in  \Lch'_g.
\] 
Using $(\ref{gap condition frac wave})$ or $\gamma_{a,b} = -\gamma_{a',b'}-\sigma=-\gamma_{p,q}+\sigma$, we have
\[
\Big\|\int_{\R} \frac{e^{-i(t-s)\Lambda_g^\sigma}}{\Lambda_g^\sigma} F(s) ds \Big\|_{L^p(\R, L^q)} = \|T_{\gamma_{p,q}} T^\star_{\gamma_{a,b}}F\|_{L^p(\R, L^q)} \leq C\|F\|_{L^{a'}(\R,L^{b'})}.
\]
As in the proof of the inhomogeneous Strichartz estimates for the fractional Schr\"odinger equations, the Christ-Kiselev Lemma implies $(\ref{inhomo strichartz frac wave})$ for all fractional admissible pairs satisfying the gap condition $(\ref{gap condition frac wave})$ excluding the case $p=a'=2$. \newline
\indent Next, we prove
\begin{align}
\|v\|_{L^\infty(\R, \dot{H}^{\gamma_{p,q}}_g)} \leq C\Big( \|v_0\|_{\dot{H}^{\gamma_{p,q}}_g}+\|v_1\|_{\dot{H}^{\gamma_{p,q}-\sigma}_g}+\|F\|_{L^{a'}(\R, L^{b'})}\Big). \nonumber
\end{align}
By using the homogeneous Strichartz estimate for a fractional admissible pair $(\infty,2)$ with $\gamma_{\infty,2} =0$ and that $\|v\|_{L^\infty(\R, \dot{H}^{\gamma_{p,q}}_g)} = \|\Lambda_g^{\gamma_{p,q}}v\|_{L^\infty(\R, L^2)}$, we have
\begin{multline}
\|v\|_{L^\infty(\R, \dot{H}^{\gamma_{p,q}}_g)} \leq C\Big( \|\Lambda_g^{\gamma_{p,q}}v_0\|_{L^2}+\|\Lambda_g^{\gamma_{p,q}}v_1\|_{\dot{H}^{-\sigma}_g} \Big. \\
\Big.+\Big\|\int_{0}^{t} \Lambda_g^{(\gamma_{p,q}-\sigma)} \sin {(t-s)\Lambda_g^\sigma} F(s)ds\Big\|_{L^\infty(\R, L^{2})}\Big). \nonumber
\end{multline}
Using the Christ-Kiselev Lemma, it suffices to prove
\[
\Big\|\int_{\R} \Lambda_g^{(\gamma_{p,q}-\sigma)}  e^{-i(t-s)\Lambda_g^\sigma} F(s)ds\Big\|_{L^\infty(\R, L^2)} \leq C \|F\|_{L^{a'}(\R, L^{b'})}.
\]
Using the above notation, we have
\begin{align}
\Big\|\int_{\R} \Lambda_g^{(\gamma_{p,q}-\sigma)}  e^{-i(t-s)\Lambda_g^\sigma} F(s)ds\Big\|_{L^\infty(\R, L^2)} &= \|T_0 T^\star_{\gamma_{a,b}}F\|_{L^\infty(\R, L^2)} \nonumber \\
&\leq C \|T^\star_{\gamma_{a,b}}F\|_{L^2} \leq C \|F\|_{L^{a'}(\R, L^{b'})}. \nonumber
\end{align}
We repeat the same process for $\partial_t v$ and obtain
\[
\|\partial_t v\|_{L^\infty(\R, \dot{H}^{\gamma_{p,q}-\sigma}_g)} \leq C\Big( \|v_0\|_{\dot{H}^{\gamma_{p,q}}_g}+\|v_1\|_{\dot{H}^{\gamma_{p,q}-\sigma}_g}+\|F\|_{L^{a'}(\R, L^{b'})}\Big).
\]
This completes the proof of Proposition $\ref{prop global inhomogeneous strichartz frac wave}$.  \defendproof
\section*{Acknowledgments}
The author would like to express his deep thanks to his wife - Uyen Cong for her encouragement and support. He would like to thank his supervisor Prof. Jean-Marc Bouclet for the kind guidance, encouragement and careful reading of the manuscript. He also would like to thank the reviewer for his/her helpful comments and suggestions. 

\end{document}